\documentclass[aop]{imsart}

\RequirePackage{amsthm,amsmath,amsfonts,amssymb}
\RequirePackage[numbers,sort&compress]{natbib}
\RequirePackage[colorlinks,citecolor=blue,urlcolor=blue]{hyperref}
\RequirePackage{graphicx}
\usepackage{enumitem}
\usepackage{array}
\usepackage{bbm}

\startlocaldefs

\theoremstyle{plain}

\newtheorem{theorem}{Theorem}[section]
\newtheorem{lemma}[theorem]{Lemma}
\newtheorem{lem}[theorem]{Lemma}
\newtheorem{proposition}[theorem]{Proposition}
\newtheorem{propn}[theorem]{Proposition}
\newtheorem{corollary}[theorem]{Corollary}
\newtheorem{cor}[theorem]{Corollary}

\theoremstyle{remark}

\newtheorem{defn}[theorem]{Definition}

\newtheorem{rem}[theorem]{Remark}
\newtheorem{assu}{Assumption}

\endlocaldefs

\begin{document}

\begin{frontmatter}
\title{On the cover time of Brownian motion\\on the Brownian continuum random tree}
\runtitle{cover time of Brownian motion on the Brownian CRT}

\begin{aug}

\author[A]{\fnms{G.}~\snm{Andriopoulos}\ead[label=e1]{ga73@nyu.edu}},
\author[B]{\fnms{D.\ A.}~\snm{Croydon}\ead[label=e2]{croydon@kurims.kyoto-u.ac.jp}},
\author[C]{\fnms{V.}~\snm{Margarint}\ead[label=e3]{vmargari@charlotte.edu}},
\and
\author[D]{\fnms{L.}~\snm{M\'enard}\ead[label=e4]{laurent.menard@normalesup.org
}}

\address[A]{Department of Computer Science,
New York University Abu Dhabi\printead[presep={ ,\ }]{e1}}

\address[B]{Research Institute for Mathematical Sciences,  Kyoto University, Japan\printead[presep={,\ }]{e2}}

\address[C]{Department of Mathematics and Statistics,
University of North Carolina at Charlotte \printead[presep={,\ }]{e3}}

\address[D]{Modal’X, UMR CNRS 9023, UPL, Univ. Paris-Nanterre, F92000 Nanterre, France\printead[presep={,\ }]{e4}}
\end{aug}

\begin{abstract}
Upon almost-every realisation of the Brownian continuum random tree (CRT), it is possible to define a canonical diffusion process or `Brownian motion'. The main result of this article establishes that the cover time of the Brownian motion on the Brownian CRT (i.e.\ the time taken by the process in question to visit the entire state space) is equal to the infimum over the times at which the associated local times are strictly positive everywhere. The proof of this result depends on the recursive self-similarity of the Brownian CRT and a novel version of the first Ray-Knight theorem for trees, which is of independent interest. As a consequence, we obtain that the suitably-rescaled cover times of simple random walks on critical, finite variance Galton-Watson trees converge in distribution with respect to their annealed laws to the cover time of Brownian motion on the Brownian CRT. Other families of graphs that have the Brownian CRT as a scaling limit are also covered. Additionally, we partially confirm a conjecture made in 1991 by David Aldous regarding cover-and-return times.
\end{abstract}



\end{frontmatter}

\tableofcontents

\section{Introduction}

The cover time, that is, the time taken to explore the entire state space, is a natural quantity to study for a stochastic process. In this direction, a question that has been widely considered in the probability and computer science literature is how the cover time behaves asymptotically for a sequence of random walks on graphs of increasing sizes. However, the cover time is not always an easy quantity to pin down. Indeed, even when one has a functional scaling limit for a sequence of random walks, the scaling of the cover time can not typically be read off from this, since such results do not take into account the fine structure of the paths of the processes. For instance, to see a non-trivial functional limit theorem for the random walk on the discrete torus $(\mathbb{Z}/n\mathbb{Z})^d$, equipped with nearest neighbor edges, the time-scaling is $n^2$, but the cover time scales like $n^2(\log n)^2$ for $d=2$ and $n^d(\log n)$ for $d\geq 3$ (see \cite{DPRZ} for a more detailed result in two dimensions, and \cite[Chapter 11]{LPW} for an argument that yields the order of growth in higher dimensions). Moreover, although the random walk on the largest supercritical percolation cluster contained within the box $\{-n,-n+1,\dots,n\}^d$ satisfies the same global scaling (up to constants, and modulo the boundary condition), the cover time in this case is of order $n^d(\log n)^2$ for all $d\geq 2$, see \cite{Abe}. (It is reasonable to expect the same result for supercritical percolation on the discrete torus.) Thus we see that the cover time can be highly sensitive to the detailed geometry of the underlying space.

The goal of this article is to demonstrate convergence of the rescaled cover times of random walks on graphs that converge to the Brownian continuum random tree (CRT). The latter object has, since its introduction by Aldous in \cite{Aldous1, Aldous2, Aldous3}, come to be a central object in modern probability theory. It is known to appear as a scaling limit of critical Galton-Watson trees \cite{Aldous2}, uniform spanning trees on high-dimensional tori \cite{ANS}, critical percolation clusters on certain random hyperbolic maps \cite{AC,CM}, amongst many other examples, and is also conjectured to appear in the description of the scaling limit of critical percolation clusters on the integer lattice in high dimensions \cite{HS}. The limiting cover time in our main convergence result (see Theorem \ref{cor:convres} below) will be expressed in terms of the cover time of the natural diffusion, or `Brownian motion', on the Brownian CRT, which is a process whose definition was suggested by Aldous in \cite{Aldous2} and which was first rigorously defined by Krebs in \cite{Krebs}. In the course of our arguments, we establish various new results for this process.

In fact, cover times of random walks on graph trees, and more general graphs, was something of a research focus for Aldous. For example, in \cite{Aldcov}, he derived some scaling results concerning the cover time of random walks on some particular examples of trees, including giving precise asymptotics in the case of the balanced $b$-ary tree. In the same paper, he discussed the cover times of random walks on critical Galton-Watson trees, including giving a specific conjecture on the scaling limit of the expected `cover-and-return' time (i.e.\ the time taken to hit all vertices and then return to the starting point) for a tree with a Poisson(1) offspring distribution. As a corollary of our main result, we will partially confirm his conjecture, in the sense that we confirm a limit of the relevant sequence exists (see Corollary \ref{crcor} below). Moreover, since our limit is written in terms of the expectation of the cover-and-return time of the Brownian motion on the Brownian CRT, this also supports his conjecture on the exact value of this constant. We highlight that understanding the cover time of the latter process was listed in a section of Aldous' paper \cite{Aldous2} entitled `hard distributional properties [of Brownian motion on continuum trees]'.

Part of the difficulty of studying the cover time of Brownian motion on the Brownian CRT is that it does not concentrate on its mean. In \cite{Aldousthresh}, Aldous gave a characterisation of when concentration of the cover time occurs; namely, the phenomenon arises if the maximal expected hitting time of vertices is asymptotically of a smaller order than the expected cover time, but there exist starting states for which concentration is not seen if this is not the case. (See \cite{Abecov} for work that helps explain which kinds of graph fall into each of these categories. Also, a recent result in \cite{Hermon} gives a related characterisation of cover time concentration for a class of graphs that includes transitive ones in terms of the spectral gaps of the graphs in question.) In the case when concentration occurs, the first order behaviour of the cover time can be understood from its expectation, which was shown in the seminal work of \cite{DLP} to be strongly connected to the expectation of the maximum of an associated Gaussian field; see also \cite{Ding} for tighter estimates in the case of bounded degree graphs and trees (in the concentration regime). However, in `low-dimensional' settings, it is easy to find examples where the cover time does not concentrate on its mean, and so its expectation is not sufficient to describe the asymptotic distribution to first order. Consider, for example, one-dimensional Brownian motion on the unit interval $[0,1]$ and the simple random walks on $(n^{-1}\mathbb{Z})\cap[0,1]$ approximating this. Clearly the cover times of the graphs, when rescaled by $n^{-2}$, have a random limit, and it is a simple exercise to show that this can be described in terms of the cover time of the limiting diffusion. We will show a corresponding result when the limiting space is the Brownian CRT.

In order to state our main results, let us introduce some of the main objects of the subsequent discussion. In this introduction, we denote the Brownian CRT, equipped with its natural metric, by $(\mathcal{T},d_{\mathcal{T}})$ (see Section \ref{sec:rec} for a precise definition of this space in terms of a Brownian excursion, including our choice of normalisation). We also write $\rho$ for the root vertex of $\mathcal{T}$, $\mu_\mathcal{T}$ for its canonical measure, and $\mathbf{P}$ for the probability measure on the underlying probability space upon which $\mathcal{T}$ is built. Given a typical realisation of $(\mathcal{T},d_\mathcal{T},\mu_\mathcal{T})$, one has an associated, canonical Brownian motion $((X^\mathcal{T}_t)_{t\geq 0},(P^\mathcal{T}_x)_{x\in\mathcal{T}})$ (see Section \ref{bmsec} for the definition of such a process on a compact real tree more generally). Setting $X_{[0,t]}^\mathcal{T}:=\{X_s^\mathcal{T}:\:s\in[0,t]\}$ for the range of $X^\mathcal{T}$ up to time $t$, the cover time of $X^\mathcal{T}$ is defined as follows:
\begin{equation}\label{taucovdef}
\tau_{\mathrm{cov}}:=\inf\left\{t\geq 0:\:X^\mathcal{T}_{[0,t]}=\mathcal{T}\right\};
\end{equation}
as we will note in Section \ref{bmsec}, this is a finite random variable, almost-surely. Additionally, the Brownian motion on the Brownian CRT is known to admit jointly continuous local times $(L^\mathcal{T}_t(x))_{x\in \mathcal{T},\:t\geq 0}$, and a study of these will be central to our work. Whilst it is obviously the case for random walks on graphs that the cover time is the first time that the local times are strictly positive everywhere (possibly with a difference of one time-step in the case of discrete-time random walks), and the same property is readily checked for one-dimensional Brownian motion on $[0,1]$, for example, it is by no means clear for diffusions on more general spaces, especially ones with a fractal structure such as the Brownian CRT. Yet, as was discussed in \cite[Remark 7.4]{croydonmoduli} (cf.\ \cite{andblanket}) and will be used in our proof below, checking this property of local times yields a potential route for establishing convergence of the cover times of the random walks approximating the relevant diffusion. Thus it is a key target of ours to prove the following result. Indeed, the majority of the article is devoted to its proof.

\begin{theorem}\label{thm:mainres} For $\mathbf{P}$-a.e.\ realisation of $\mathcal{T}$, for every $x\in\mathcal{T}$, it $P^\mathcal{T}_x$-a.s.\ holds that
\begin{equation}\label{twodefs}
\tau_{\mathrm{cov}}=\inf\left\{t\geq0:\: L^\mathcal{T}_t(y)>0,\:\forall y\in\mathcal{T}\right\}.
\end{equation}
\end{theorem}

Before proceeding to the application of the above theorem to scaling limits of cover times, let us briefly outline our strategy for verifying it. As is at the heart of \cite{DLP}, we will appeal to the strong connection between local times and Gaussian fields via a Ray-Knight theorem. More specifically, we give a version of the classical first Ray-Knight theorem in which we describe the distribution of the local time at the hitting time of a specified vertex in terms of certain squared Bessel (BESQ) processes (tree-indexed in our case), see Lemma \ref{rkhit}. We will apply this at a certain approximation to the cover time that occurs (just after) the time $X^\mathcal{T}$ covers a subtree $\mathcal{T}^{(\varepsilon)}\subseteq\mathcal{T}$ with finitely many branches that approximates $\mathcal{T}$ to within $d_\mathcal{T}$-Hausdorff distance $\varepsilon$. When $\mathcal{T}$ is selected according to the It\^{o} excursion measure and $\mathcal{T}^{(\varepsilon)}$ is chosen according to a particular `height decomposition' (see Section \ref{sec:rec} for details), we have that the (closures of) the components of $\mathcal{T}\backslash\mathcal{T}^{(\varepsilon)}$ form a Poisson process of copies of rescaled Brownian CRTs, see Figure \ref{f1}. Together with our Ray-Knight theorem, this implies that the local times of $X^\mathcal{T}$ at our cover time approximation behave on these components as independent tree-indexed zero-dimensional BESQ processes. With this perspective, we are able to obtain the result of interest by appealing to a certain path property of BESQ-Brownian snakes, which we check by applying an argument that involves expressing the measures of certain events as solutions of differential equations; this part of the work is based on techniques developed by Le Gall in \cite{LG}. We emphasize this means that, in our approach, rather than sampling the entire tree, and then the processes $X^\mathcal{T}$ and $L^\mathcal{T}$ upon this, we first sample $\mathcal{T}^{(\varepsilon)}$, then the local times upon this set at a particular stopping time, before finally sampling the remaining parts of the tree and local times together. In particular, we interleave the selection of the random environment and the stochastic process that has this as its state space.

\begin{figure}[t]
\begin{center}
\includegraphics[width=9cm]{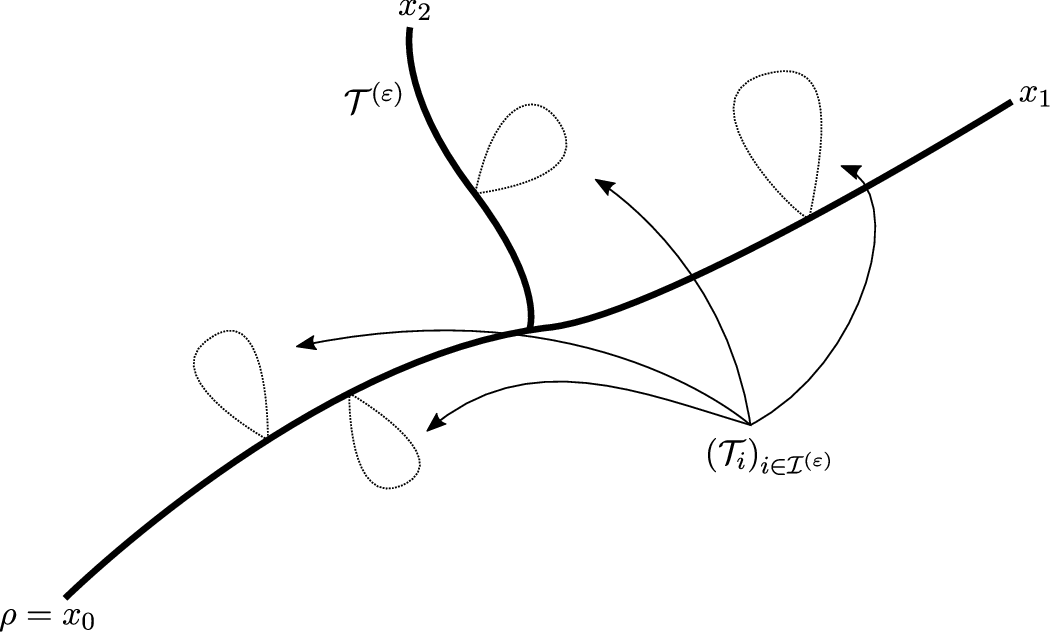}
\end{center}
\caption{The above illustration shows the decomposition of the Brownian CRT that we will apply. The tree $\mathcal{T}^{(\varepsilon)}$, shown with bold lines, spans a collection of leaves $(x_i)_{i=1}^{M(\varepsilon)}$ such that the Hausdorff distance between $\mathcal{T}^{(\varepsilon)}$ and $\mathcal{T}$ is smaller than $\varepsilon$. Moreover, our particular choice of leaves means that the closures of the connected components $\mathcal{T}\backslash \mathcal{T}^{(\varepsilon)}$, which are denoted $(\mathcal{T}_i)_{i\in\mathcal{I}^{(\varepsilon)}}$ are distributed as rescaled copies of the Brownian CRT. On $\mathcal{T}^{(\varepsilon)}$, the local times are easily checked to take strictly positive values just after the cover time. To check the same is true on the remaining parts of $\mathcal{T}\backslash \mathcal{T}^{(\varepsilon)}$, we use a Ray-Knight theorem, which enables us to understand the local times within these in terms of a tree-indexed zero-dimensional squared Bessel process, and thereby apply results from the literature for generalised Brownian snakes.}\label{f1}
\end{figure}

We now come to our main convergence result. Although our motivation is principally to understand cover times of random walks, we will not need to restrict to such. Rather, we will work in a more general setting of stochastic processes associated with resistance forms, as developed in \cite{Croydonscalelim,CHK,KigAOF,Kigq,Noda}, for example. In particular, we consider objects of the form
\begin{equation}\label{quintuple}
\mathcal{K}=\left(K,R_K,\mu_K,\rho_K,P^K\right),
\end{equation}
where $(K,R_K)$ is a compact resistance metric space that satisfies a certain metric entropy condition (see Section \ref{51} below for a definition of this and the other components that we now introduce), $\mu_K$ is a finite Borel measure on $K$ of full support, $\rho_K$ is a marked point of $K$, and $P^K$ is the law of the associated stochastic process $X^K=(X^K_t)_{t\geq 0}$ and its jointly continuous local times $L^K=(L^K_t(x))_{x\in K,\:t\geq 0}$, where we will always assume that $X^K$ is started from the root $\rho_K$. We note that the process $X^K$ takes values in $D(\mathbb{R}_+,K)$, the space of c\`{a}dl\`{a}g paths in $K$, equipped with the usual Skorokhod $J_1$ topology, and $L^K$ takes values in $C(K\times \mathbb{R}_+,\mathbb{R}_+)$. A convenient topology for describing convergence of such objects was introduced in \cite{Noda, Nodametric}, and this will be recalled below in Section \ref{51}. Essentially the convergence $\mathcal{K}_n\rightarrow\mathcal{K}$ means that all the spaces can be isometrically embedded into a common metric space in such a way that the sets converge with respect to Hausdorff distance, the measures converge weakly, the roots are identified, and the laws of the stochastic processes converge weakly, with the local times $L^{K_n}$ close to those of $L^K$ for nearby points (see Lemma \ref{embedding} for details). Following \cite{Noda}, we write the collection of $(K,R_K,\mu_K,\rho_K)$, identified if they are root-preserving and measure-preserving isometric, as $\check{\mathbb{F}}_c$, and the space containing elements of the form $\mathcal{K}$ as $\mathbb{M}_{L,c}$. (In fact, $\mathbb{M}_{L,c}$ is a larger space that contains elements with fewer restrictions, but this is unimportant here.)  In this setting, we have the following consequence of Theorem \ref{thm:mainres}, where we write $\tau_{\mathrm{cov}}(K)$ for the cover time of a process $X^K$, defined analogously to \eqref{taucovdef}.

\begin{theorem}\label{cor:convres} For $\mathbf{P}$-a.e.\ realisation of $\mathcal{T}$, if $(K_n,R_{K_n},\mu_{K_n},\rho_{K_n})\in \check{\mathbb{F}}_c$ are such that
\begin{equation}\label{kconv}
\left({K_n},R_{K_n},\mu_{K_n},\rho_{K_n},P^{K_n}\right)\rightarrow \left(\mathcal{T},2d_\mathcal{T},\mu_\mathcal{T},\rho,P_\rho^\mathcal{T}\right)
\end{equation}
in $\mathbb{M}_{L,c}$, and moreover, for each $n$, either $K_n$ is at most a countable set or $X^{K_n}$ is $P^{K_n}$-a.s.\ continuous, then the law of $\tau_{\mathrm{cov}}(K_n)$ under $P^{K_n}$ converges to that of $\tau_{\mathrm{cov}}(\mathcal{T})$ under $P_\rho^\mathcal{T}$.
\end{theorem}

For random spaces, it is also natural to deal with the annealed law of the cover time, i.e.\ the law averaged over the randomness of the space. Specifically, we write
\begin{equation}\label{annealedlaw}
\mathbb{P}^K\left(\tau_{\mathrm{cov}}(K)\in \cdot\right)=\int P^K\left(\tau_{\mathrm{cov}}(K)\in \cdot\right)\mathbf{P}\left(d\left(K,R_K,\mu_K,\rho_K\right)\right),
\end{equation}
where we suppose that the underlying probability space for $(K,R_K,\mu_K,\rho_K)$ is also denoted $\mathbf{P}$. We note that the measurability of the map $(K,R_K,\mu_K,\rho_K)\mapsto\mathcal{K}$ was checked in \cite[Proposition 6.1]{Noda}, which means the above integral is well-defined. From the previous result, it is straightforward to deduce the following.

\begin{corollary}\label{anncor}
If $(K_n,R_{K_n},\mu_{K_n},\rho_{K_n})$ are random elements of $\check{\mathbb{F}}_c$ such that \eqref{kconv} holds in distribution, and moreover, for each $n$, either $K_n$ is at most a countable set or $X^{K_n}$ is $P^{K_n}$-a.s.\ continuous, $\mathbf{P}$-a.s.,\ then the law of $\tau_{\mathrm{cov}}(K_n)$ under $\mathbb{P}^{K_n}$ converges to that of $\tau_{\mathrm{cov}}(\mathcal{T})$ under its annealed law $\mathbb{P}^\mathcal{T}$ (defined similarly to \eqref{annealedlaw}, with the measure within the integrand being given by $P_\rho^\mathcal{T}$).
\end{corollary}

Whilst the statements of Theorem \ref{cor:convres} and Corollary \ref{anncor} are broad, they are maybe a little abstract for those not familiar with resistance forms. So, let us outline how they apply in the case of random walks on graphs. In particular, suppose $K$ is the vertex set of a connected, finite graph containing no loops or multiple edges. Equip the graph with conductances $(c(x,y))_{x,y\in K}$, i.e.\ $c(x,y)$ are non-negative numbers such that $c(x,y)=c(y,x)$ and $c(x,y)>0$ if and only if $\{x,y\}$ is an edge of the graph. Write $R_K=(R_K(x,y))_{x,y\in K}$ for the associated effective resistance (see, for example, \cite[Section 9.4]{LPW} for a definition). Moreover, let $\mu_K$ be a finite measure on $K$ placing a non-zero mass on each vertex. In this case, $(K,R_K,\mu_K,\rho_K)\in \check{\mathbb{F}}_c$ for any vertex $\rho_K\in K$, and the associated stochastic process $(X^K_t)_{t\geq 0}$ is the continuous time Markov chain with generator given by
\[\Delta_Kf(x):=\frac{1}{\mu_K(\{x\})}\sum_{y\in K}c(x,y)\left(f(y)-f(x)\right),\]
for $f:K\rightarrow \mathbb{R}$. One standard choice for $\mu_K$ is the counting measure, i.e.\ $\mu_K(\{x\})=1$ for all $x\in K$, in which case $X^K$ is called the \emph{variable-speed random walk (VSRW)}. Another is the conductance measure, i.e.\ $\mu_K(\{x\})=\sum_{y\in K}c(x,y)$ for all $x\in K$, in which case the process is called the \emph{constant-speed random walk (CSRW)} since it has unit mean holding times. The latter property means the CSRW closely resembles the discrete-time simple random walk with transition probabilities given by $P(x,y)=c(x,y)/\sum_{z\in K}c(x,z)$. Although we will not discuss the discrete-time process to any great extent later, where we have the conclusion of Theorem \ref{cor:convres} or Corollary \ref{anncor} for the CSRW, we will typically be able to transfer this to the discrete-time random walk with no problem. (For moments of the cover time, one should be slightly careful in moving between discrete and continuous time, see Section \ref{uisec}.) As for the local times of $X^K$, these are given by
\[L^K_t(x)=\frac{1}{\mu_K(\{x\})}\int_0^t\mathbf{1}_{\{X^K_s=x\}}ds.\]
In examples, we will generally discuss a sequence of graphs $(K_n)_{n\geq 1}$ for which the effective resistance and measure have been rescaled. That is, writing $P^{K_n}$ for the law of $X^{K_n}$ started from $\rho_{K_n}$, we are able to check
\begin{equation}\label{rescaledg}
\left(K_n,a_n^{-1}R_{K_n},b_n^{-1}\mu_{K_n},\rho_{K_n},P^{K_n}\left(\left(\left(X^{K_n}_{ta_nb_n}\right)_{t\geq 0},\left(a_n^{-1}L^{K_n}_{ta_nb_n}(x)\right)_{x\in K_n,\:t\geq 0}\right)\in \cdot\right)\right)
\end{equation}
converges to the right-hand side of \eqref{kconv}, either for a typical realisation of the Brownian CRT (as in Theorem \ref{cor:convres}) or in distribution (as in Corollary \ref{anncor}), where $(a_n)_{n\geq 1}$, $(b_n)_{n\geq 1}$ are some divergent sequences of non-negative numbers. We then obtain that
\begin{equation}\label{graphconv}
(a_nb_n)^{-1}\tau_{\mathrm{cov}}\left(K_n\right)\rightarrow\tau_{\mathrm{cov}}\left(\mathcal{T}\right)
\end{equation}
in the appropriate sense. Various examples for which the convergence of objects as at \eqref{rescaledg} can be checked are presented in \cite{Noda} (see also \cite{andblanket, Nodadisc}), and these include the following. We stress that, although these two examples are graph trees, our result should be applicable beyond such. See Problem 2 below for details of some non-tree models for which we expect similar results to hold.

\begin{description}
\item[\textbf{Example 1:} Critical, finite variance Galton-Watson trees.]
Let $T_n$ be the tree generated by a Galton-Watson process whose offspring distribution is non-trivial, critical (mean one) and has finite variance $\sigma^2$, conditioned to have $n$ vertices. (One may consider only those $n$ for which the probability of the latter event is strictly positive.) Let $c_{T_n}(x,y)=1$ if $\{x,y\}$ is a parent-child bond in $T_n$, and set $c_{T_n}(x,y)=0$ otherwise. Also, let $\mu_{T_n}^V$ be the counting measure on $T_n$, and $\mu_{T_n}^C$ be the conductance (or degree) measure. Since the spaces are trees, one has the associated effective resistance $R_n$ is the usual shortest path graph distance on $T_n$, and we have from \cite{andblanket, Noda} that the convergence described at \eqref{rescaledg} holds in distribution, with scaling factors given by
\[a_n:=\frac{\sqrt{n}}{2\sigma},\]
\[b_n^V:=n,\qquad b_n^C:=2n,\]
where the superscripts $V$ and $C$ distinguish between the cases of the variable- and constant-speed random walks. (The convergence of the random walks was initially established in \cite{croydonconv}.) Moreover, we highlight that in results concerning the scaling of graph trees to the Brownian CRT, the most common measure considered is the counting measure. However, for a tree with unit conductances, it is easy to check the discrepancy between this and half the conductance measure is small, and so there is no problem in checking the result for both the VSRW and the CSRW.) Thus we have that \eqref{graphconv} holds in distribution under the relevant annealed laws. In particular, for the CSRW (or discrete-time random walk) on $T_n$, the time scaling factor is given by $a_nb_n^C=\sigma^{-1}n^{3/2}$.

\item[\textbf{Example 2:} Uniform spanning trees on high-dimensional tori.]\
Consider the uniform spanning tree $U_n$ on a $d$-dimensional torus $(\mathbb{Z}/n\mathbb{Z})^d$ with $d\geq 5$. Equip edges of the graph with unit conductances, and suppose $\mu_{U_n}$ is either the counting measure or the conductance measure on the vertices of $U_n$. Based on the work of \cite{ANS}, it is then shown in \cite{Noda} that the convergence at \eqref{rescaledg} holds in distribution, with scaling factors given by
\[a_n:=\alpha(d)n^{d/2},\]
\[b_n^V:=n^d,\qquad b_n^C:=2n^d,\]
where $\alpha(d)$ is some dimension-dependent constant and again the superscripts $V$ and $C$ distinguish between the cases of the variable- and constant-speed random walks. (In \cite{ANS,Noda}, only the counting measure was considered, but, similarly to the comment of the previous example, it is straightforward to replace this with the conductance measure, and thereby obtain the result for both the VSRW and the CSRW.) In particular, for the CSRW (or discrete-time random walk) on $T_n$, the time scaling factor is given by $a_nb_n^C=2\alpha(d)n^{3d/2}$. We note that \cite{ANS,Noda} also consider a more general model of a `high-dimensional' uniform spanning tree and our results would cover this; we restrict to the case of tori simply for brevity.
\end{description}

Whilst the above discussion covers the main contribution of this paper, we continue this introduction by presenting a couple of further consequences of our results that we believe are also of interest. Firstly, for any $p\geq 1$, we check the uniform $L^p$-integrability of the cover times of random walks on critical Galton-Watson trees whose offspring distribution admits finite exponential moments. In conjunction with our earlier results, this allows us to conclude that the $p$th moments of the cover times of the discrete models converge, with again the limit being expressed in terms of the cover time of Brownian motion on the Brownian CRT. As part of the result we note that the latter quantity also admits finite moments of all orders. (Our argument could also be used to yield a corresponding result for certain stretched exponential moments, see Remark \ref{semom} below.) Note that we write $\mathbb{E}^K$ for the expectation under $\mathbb{P}^K$ (and similarly for other spaces).

\begin{corollary}\label{pocor}
As in Example 1 above, let $T_n$ be the tree generated by a Galton-Watson process whose offspring distribution is non-trivial and critical (mean one), conditioned to have size $n$. Additionally, suppose that the offspring distribution $(p_k)_{k\geq 0}$ has finite exponential moments of some order, i.e.\ $\sum_{k\geq 0}e^{\lambda k}p_k<\infty$ for some $\lambda >0$. Writing $\sigma^2$ for the variance of the offspring distribution and $\tau_{\rm cov}({T}_n)$ for the cover time of the discrete-time or constant-speed random walk on $T_n$ started from the root, it then holds that, for all $p\geq 1$,
\[\sigma^pn^{-3p/2}\mathbb{E}^{T_n}\left(\tau_{\rm cov}({T}_n)^p\right)\rightarrow \mathbb{E}^\mathcal{T}\left(\tau_{\rm cov}(\mathcal{T})^p\right),\]
where the right-hand side above is finite. In particular, the cover time of Brownian motion on the Brownian CRT admits finite polynomial moments of all orders.
\end{corollary}

Secondly, again motivated by the work of Aldous in \cite{Aldous2,Aldcov}, we consider the cover-and-return time. In particular, for Brownian motion on the Brownian CRT started from $\rho$, this is defined as
\[\tau_{\mathrm{cov}}^{+}:=\inf\left\{t\geq \tau_{\mathrm{cov}}:\:X^\mathcal{T}_t=\rho\right\},\]
i.e.\ the first time $X^\mathcal{T}$ returns to its starting point $\rho$ after the cover time $\tau_{\mathrm{cov}}$. We define the corresponding stopping time similarly for other processes. It is not difficult to incorporate the additional time from $\tau_{\mathrm{cov}}$ to $\tau_{\mathrm{cov}}^+$ into our arguments and thereby obtain the following result.

\begin{corollary}\label{crcor}
(a) Under the assumptions of Theorem \ref{cor:convres}, the law of $\tau_{\mathrm{cov}}^+(K_n)$ under $P^{K_n}$ converges to that of $\tau_{\mathrm{cov}}^+(\mathcal{T})$ under $P_\rho^\mathcal{T}$.\\
(b) Under the assumptions of Corollary \ref{anncor}, the law of $\tau_{\mathrm{cov}}^+(K_n)$ under $\mathbb{P}^{K_n}$ converges to that of $\tau_{\mathrm{cov}}^+(\mathcal{T})$ under its annealed law $\mathbb{P}^\mathcal{T}$.\\
(c) In the setting of Corollary \ref{pocor},
\[\sigma^pn^{-3p/2}\mathbb{E}^{T_n}\left(\tau_{\rm cov}^+({T}_n)^p\right)\rightarrow \mathbb{E}^\mathcal{T}\left(\tau_{\rm cov}^+(\mathcal{T})^p\right),\]
where the right-hand side above is finite. In particular, the cover-and-return time of Brownian motion on the Brownian CRT admits finite polynomial moments of all orders.
\end{corollary}

Finally, we set out some problems of varying difficulty that are left open by this work.
\begin{description}
  \item[\textbf{Problem 1.}] Whilst the argument we give for the Brownian CRT has some appealing aspects, we believe that it should be possible to generalise the conclusion of Theorem \ref{thm:mainres} significantly. Somewhat obvious modifications should allow us to cover tree-like spaces with similarly nice recursive decompositions, such as $\alpha$-stable trees (see \cite{AbDel}) or the scaling limit of critical Erd\"{o}s-R\'{e}nyi graphs (see \cite{ABG1,ABG2} for a recursive decomposition of this space and \cite{BDNP} for a study of the associated cover time asymptotics); indeed, essentially all our technical arguments for Brownian motion on real trees have been written in a general way so that they could easily be applied directly or adapted to such spaces. More ambitiously, though, it is not unreasonable to conjecture that the identity at \eqref{twodefs} holds for any of the stochastic processes associated with elements of $\check{\mathbb{F}}_c$, such as nested fractals (as introduced in \cite{Lind}) or the two-dimensional Sierpi\'{n}ski carpet. For all these spaces, the processes in question admit jointly continuous local times, and so it is harder to conceive that they admit exceptional points, as would be needed for the conclusion of Theorem \ref{thm:mainres} to fail. At the moment, however, we have no idea of a strategy to tackle this issue.
  \item[\textbf{Problem 2.}] Other examples for which a Brownian CRT scaling limit is known include the high-dimensional critical branching random walk \cite{BCF2,BCF1} and lattice trees \cite{CFHP}, as well as the critical percolation cluster on the random hyperbolic half-planar triangulation \cite{AC}. To apply Corollary \ref{anncor} in these cases, it will suffice to verify the metric entropy condition of \cite{Noda}, as described in Remark \ref{nodarem} below. At least for the branching random walk model, this should be an easy consequence of the corresponding result for Galton-Watson trees. Another model for which we would also expect the result to be applicable is the incipient infinite cluster of percolation on the integer lattice in high dimensions. However, deriving the scaling limit in this case is still an open question. (See \cite{HS} for rigourous work in this direction, and \cite{Croydonscalelim} for a detailed conjecture on the resistance scaling of this model.)
 \item[\textbf{Problem 3.}] We do not believe exponential moments are essential for the conclusion of Corollary \ref{pocor} to be applicable, rather we expect that a finite variance assumption is sufficient. The one part of the argument for which we apply the exponential moment assumption is in deducing a suitable quantitative bound for the metric cover size of the graphs in question (see the comment above Lemma \ref{cc3} for further detail). As with the bounds in \cite{ABDJ}, however, one might hope that second moments of the offspring distribution are enough to yield exponential tail estimates for the distribution of the metric cover size associated with the conditioned Galton-Watson tree.
  \item[\textbf{Problem 4.}] Concerning cover-and-return times, in \cite{Aldous2}, Aldous conjectured that
\[\mathbb{E}\left(\tau_{\mathrm{cov}}^{+}(\mathcal{T})\right)=6\sqrt{2\pi}.\]
This assertion was based on a consideration of the asymptotics of cover-and-return times for random walks on conditioned Poisson(1) trees, as studied in \cite{Aldcov}. (We note that such trees can also be interpreted as uniform random trees on $n$ labelled vertices $\{1,2,\dots,n\}$, with root vertex $1$, see \cite[Lemma 12]{Aldcov} or \cite[Section 2.2]{Kolchin}, for example.) In partial support of this conjecture, Corollaries \ref{pocor} and \ref{crcor} confirm that indeed the cover-and-return times of the relevant discrete models converge to the corresponding quantity for the Brownian CRT. To complete the proof of the above identity, one also needs to make rigourous various heuristic arguments in \cite{Aldcov}. We believe that the local time scaling limits now known to hold in this setting will contribute to this program, but it seems more detailed estimates will be needed to complete it.
\end{description}

The remainder of the article is organised as follows. In Section \ref{bmsec}, we derive basic results for Brownian motions on compact real trees, including a version of the first Ray-Knight theorem for trees. It is here that we reduce the problem of proving Theorem \ref{thm:mainres} to a technical condition (see Assumption \ref{assu:2}) that will later be checked using BESQ-Brownian snakes. Following this, our recursive height decomposition of the Brownian CRT is described in Section \ref{sec:rec}, and then the desired property of BESQ-Brownian snakes is checked in Section \ref{sec:bes}. Putting these pieces together, we prove Theorem \ref{thm:mainres} in Section \ref{sec:mainproof}. Finally, we derive the convergence results of Theorem \ref{cor:convres} and Corollary \ref{anncor} in Section \ref{sec:convsec}, the moment convergence statement of Corollary \ref{pocor} in Section \ref{uisec}, and our result on cover-and-return times, i.e.\ Corollary \ref{crcor}, in Section \ref{sec:cr}. We note that, in various estimates, we will write $C$ and $c$ for constants whose particular values are unimportant, and these might change from line to line.

\section{Brownian motion on compact real trees}\label{bmsec}

The aim of this section is to develop the fundamental results concerning Brownian motions on compact real trees that will provide the basis for the proof of Theorem \ref{thm:mainres}. In particular, after introducing the process of interest in Section \ref{sec:bmdefn} and setting out an assumption that ensures the existence of jointly continuous local times (see Assumption \ref{assu:1} and Lemma \ref{lem:contlt}), in Section \ref{sec:covertimeapprox}, we give a version of Theorem \ref{thm:mainres} (see Proposition \ref{propmain}) under an additional assumption (see Assumption \ref{assu:2}). The latter assumption is rather technical, involving a specific property of the local times at a certain approximation of the cover time. Towards checking this, in Section \ref{sec:rktheorems}, we give a Ray-Knight theorem for local times that holds at the relevant approximation of the cover time. Together with a key property of a BESQ-Brownian snake that we will check in Section \ref{sec:bes}, such a Ray-Knight theorem will enable us to check Assumption \ref{assu:2} in our setting, and thereby establish Theorem \ref{thm:mainres}.

\subsection{Definition and basic properties}\label{sec:bmdefn}

We start by introducing the setting of this section. Since the construction and basic properties of Brownian motions on real trees are by now rather well understood, we will be brief with the details, and refer the interested reader to \cite{AEW} for comprehensive background. (See also \cite{Kigamidendrite} for earlier work on this topic, and \cite{Aldous2, Croydoncrt, Krebs} for an introduction to Brownian motion on the Brownian CRT in particular.) Specifically, for now, we assume that $(\mathcal{T},d_{\mathcal{T}})$ is a fixed compact real tree with root (distinguished vertex) $\rho$. We also suppose that $\mu_\mathcal{T}$ is a Borel probability measure on $\mathcal{T}$ of full support. As is described in \cite{AEW,Kigamidendrite} (see \cite[Theorem 1]{AEW} in particular), there is a canonical local, regular Dirichlet form  $(\mathcal{E}^\mathcal{T},\mathcal{F}^\mathcal{T})$ on $L^2(\mathcal{T},\mu_\mathcal{T})$ that satisfies
\begin{equation}\label{reschar}
2d_\mathcal{T}(x,y)=\left(\inf\left\{\mathcal{E}^\mathcal{T}(f,f):\:f\in\mathcal{F}^\mathcal{T},\:f(x)=0,\:f(y)=1\right\}\right)^{-1},\qquad \forall x,y\in\mathcal{T},\:x\neq y,
\end{equation}
and the $\mu_\mathcal{T}$-symmetric strong Markov process associated with this through the standard correspondence (of \cite[Theorem 7.2.1]{FOT}, for example), which we will denote by $((X^\mathcal{T}_t)_{t\geq 0},(P^\mathcal{T}_x)_{x\in\mathcal{T}})$, has a law that is uniquely defined from every starting point. Moreover, $X^\mathcal{T}$ has continuous sample paths, $P^\mathcal{T}_x$-a.s.\ for every $x\in \mathcal{T}$. This process is called the Brownian motion associated with $(\mathcal{T},d_\mathcal{T},\mu_\mathcal{T})$. (In various works, the constant 2 is not included in the characterisation of the form at \eqref{reschar}, but we choose to do so here to align the definition better with standard one-dimensional Brownian motion, and also the convention in works such as \cite{Aldous2, AEW, ALW}.) Since points have positive capacity in this setting (see \cite[Lemma 3.4]{AEW}), it is known that $X^\mathcal{T}$ admits jointly measurable local times $(L^\mathcal{T}_t(x))_{x\in\mathcal{T},t\geq 0}$ that satisfy the occupation density formula, i.e.\
\begin{equation}\label{odf}
\int_{0}^tf(X^\mathcal{T}_s)ds=\int_\mathcal{T}f(y)L_t^\mathcal{T}(y)\mu_\mathcal{T}(dy),
\end{equation}
for all continuous $f:\mathcal{T}\rightarrow\mathbb{R}$ and $t\geq 0$, $P^\mathcal{T}_x$-a.s.\ for every $x\in \mathcal{T}$, see \cite[Section 2]{GK}. (Precisely, when we say that $X^\mathcal{T}$ or $L^\mathcal{T}$ admits a property $P^\mathcal{T}_x$-a.s.,\ we mean that there exists a version of the relevant process that admits the property in question.)

We finish this subsection by giving a simple sufficient condition for the local time process $L^\mathcal{T}$ to admit a jointly continuous version. In particular, we will generally suppose the following, which can be verified for $\mathbf{P}$-a.e.\ realisation of the Brownian CRT (equipped with its canonical measure) by applying the volume estimates of \cite[Theorem 1.2]{Croydoncrt}, for example.

\begin{assu}\label{assu:1} There exist constants $C$ and $\eta$ such that
\[\inf_{x\in \mathcal{T}}\mu_\mathcal{T}\left(B_\mathcal{T}(x,r)\right)\geq Cr^\eta,\qquad \forall r\in (0,1),\]
where $B_\mathcal{T}(x,r):=\{y\in\mathcal{T}:\:d_\mathcal{T}(x,y)<r\}$.
\end{assu}

\begin{lem}\label{lem:contlt} Suppose Assumption \ref{assu:1} holds. It is then the case that the local times $(L^\mathcal{T}_t(x))_{x\in\mathcal{T},t\geq 0}$ are continuous, $P^\mathcal{T}_y$-a.s.\ for every $y\in \mathcal{T}$.
\end{lem}
\begin{proof}
See \cite[Lemma 2.2]{croydonscale}.
\end{proof}

\begin{rem}
A weaker sufficient condition for the continuity of local times is the metric entropy condition of \cite{Noda}, as described at \eqref{entcond}. In fact, for the results of Section \ref{sec:covertimeapprox}, we only need local time continuity, and so it would be sufficient to make this itself the assumption. However, for our Ray-Knight theorem of Section \ref{sec:rktheorems}, we require continuity of squared Bessel processes on various subtrees, and Assumption \ref{assu:1} is convenient for checking this.
\end{rem}

\subsection{Approximations for the cover time}\label{sec:covertimeapprox}

In order to prove the main result of this section (Proposition \ref{propmain}), we will consider a sequence of approximations to $\mathcal{T}$ with finitely many branches. The precise choice of approximations that we will introduce for the Brownian CRT will be presented in Section \ref{sec:rec}. For now, we merely need to suppose $(x_i)_{i\geq 0}$ is a dense sequence of points in $\mathcal{T}$. It will be consistent with our decomposition for the Brownian CRT and result in no loss of generality to assume that $x_0=\rho$ and that the elements of the sequence $(x_i)_{i\geq 0}$ are distinct. We also let $M=M(\varepsilon)$, $\varepsilon>0$, be natural numbers that are non-increasing as a function of $\varepsilon$ and diverge as $\varepsilon\rightarrow 0$. For each $\varepsilon>0$, we define
\begin{equation}\label{teps}
\mathcal{T}^{(\varepsilon)}:=\bigcup_{i=1}^{M(\varepsilon)}[[\rho,x_i]]_\mathcal{T},
\end{equation}
where for points $x,y\in\mathcal{T}$, we write $[[x,y]]_\mathcal{T}$ for the geodesic arc between $x$ and $y$. We equip $\mathcal{T}^{(\varepsilon)}$ with the restriction of $d_\mathcal{T}$ to $\mathcal{T}^{(\varepsilon)}$ (we will not introduce special notation for this restriction), so that $(\mathcal{T}^{(\varepsilon)},d_\mathcal{T})$ is a compact real tree with finitely many branches. The natural projection from $\mathcal{T}$ to $\mathcal{T}^{(\varepsilon)}$ that maps $x\in\mathcal{T}$ to the (unique) closest point in $\mathcal{T}^{(\varepsilon)}$ will be written $\pi^{(\varepsilon)}$. It is straightforward to check that
\[
d_\mathcal{T}^H\left(\mathcal{T}^{(\varepsilon)},\mathcal{T}\right)=\sup_{x\in\mathcal{T}}d_\mathcal{T}\left(x,\pi^{(\varepsilon)}(x)\right)\rightarrow0,
\]
as $\varepsilon\rightarrow 0$, where $d_\mathcal{T}^H$ is the Hausdorff distance between compact subsets of $\mathcal{T}$.

\begin{figure}[t]
\begin{center}
\includegraphics[width=9cm]{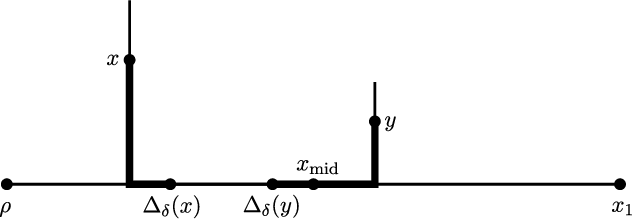}
\end{center}
\caption{An illustration of how the map $\Delta_\delta$ operates. The bold arcs represent $[[x,\Delta_\delta(x)]]_\mathcal{T}$ and  $[[y,\Delta_\delta(y)]]_\mathcal{T}$ for points $x$ and $y$ on opposite sides of the point $x_{\mathrm{mid}}$, which is the midpoint of $[[\rho,x_1]]_{\mathcal{T}}$.}\label{delta}
\end{figure}

We are now able to proceed with our discussion of cover times. The principal quantity of interest will be $\tau_{\mathrm{cov}}$, defined from $X^\mathcal{T}$ as at \eqref{taucovdef}. To show that $\tau_{\mathrm{cov}}$ can be written in terms of local times, as per the conclusion of Theorem \ref{thm:mainres}, we will introduce various approximations of $\tau_{\mathrm{cov}}$. Firstly, for $\varepsilon>0$, we set
\[\tau_{\mathrm{cov}}^{(\varepsilon)}:=\inf\left\{t\geq 0:\:X^{\mathcal{T}}_{[0,t]}\supseteq \mathcal{T}^{(\varepsilon)}\right\},\]
i.e.\ $\tau_{\mathrm{cov}}^{(\varepsilon)}$ is the first time that $X^\mathcal{T}$ has visited every point of $\mathcal{T}^{(\varepsilon)}$. We will actually run the process $X^\mathcal{T}$ for a slightly longer time than this, until it hits a specified point in $\mathcal{T}^{(\varepsilon)}$ at a distance $\delta$ away from $X^{\mathcal{T}}_{\tau_{\mathrm{cov}}^{(\varepsilon)}}$. To this end, for each $\delta\in (0,d_\mathcal{T}(\rho,x_1)/2)$, we define a map $\Delta_\delta:\mathcal{T}\rightarrow\mathcal{T}$ such that, writing $b^{\mathcal{T}}(x,y,z)$ for the branch point of $x,y,z\in\mathcal{T}$ (in particular, this is the unique point that lies on each of the arcs $[[x,y]]_{\mathcal{T}}$, $[[y,z]]_{\mathcal{T}}$ and $[[z,x]]_{\mathcal{T}}$):
\begin{itemize}
  \item if $d_\mathcal{T}(\rho,b^\mathcal{T}(\rho,x,x_1))\geq d_\mathcal{T}(\rho,x_1)/2$, then $\Delta_\delta(x)$ is the point on $[[\rho,x]]_{\mathcal{T}}$ at a distance $\delta$ from $x$;
  \item if $d_\mathcal{T}(\rho,b^\mathcal{T}(\rho,x,x_1))< d_\mathcal{T}(\rho,x_1)/2$, then $\Delta_\delta(x)$ is the point on $[[x_1,x]]_{\mathcal{T}}$ at a distance $\delta$ from $x$.
\end{itemize}
See Figure \ref{delta}. We observe that, for each $\delta$, the map $\Delta_\delta$ is continuous at every point of $\mathcal{T}\backslash \{x_{\mathrm{mid}}\}$, where $x_{\mathrm{mid}}$ is the mid-point of $[[\rho,x_1]]_{\mathcal{T}}$, i.e.\ the unique point of $\mathcal{T}$ such that
\[d_{\mathcal{T}}\left(\rho,x_{\mathrm{mid}}\right)=d_{\mathcal{T}}\left(x_{\mathrm{mid}},x_1\right)=d_\mathcal{T}\left(\rho,x_1\right)/2.\]
Moreover, we highlight that $\Delta_\delta(\mathcal{T}^{(\varepsilon)})\subseteq\mathcal{T}^{(\varepsilon)}$. With these preparations in place, we define, for $\varepsilon>0$ and $\delta\in (0,d_\mathcal{T}(\rho,x_1)/2)$,
\begin{equation}\label{tced}
\tau_{\mathrm{cov}}^{(\varepsilon)}(\delta):=\inf\left\{t\geq \tau_{\mathrm{cov}}^{(\varepsilon)}:\:X^{\mathcal{T}}_{t}=\Delta_\delta\left(X^\mathcal{T}_{\tau_{\mathrm{cov}}^{(\varepsilon)}}\right)\right\}.
\end{equation}
We similarly set
\begin{equation}\label{tcd}\tau_{\mathrm{cov}}(\delta):=\inf\left\{t\geq \tau_{\mathrm{cov}}:\:X^{\mathcal{T}}_{t}=\Delta_\delta\left(X^\mathcal{T}_{\tau_{\mathrm{cov}}}\right)\right\}.
\end{equation}
Before continuing, we note \cite[Lemma 2.3]{croydonscale} yields that, under Assumption \ref{assu:1}, the local times of $X^\mathcal{T}$ diverge uniformly (i.e.\ $\inf_{x\in\mathcal{T}}L^{\mathcal{T}}_t(x)\rightarrow \infty$ as $t\rightarrow\infty$, $P^\mathcal{T}_\rho$-a.s.), which implies in turn that $\tau_{\mathrm{cov}}$ and  $\tau_{\mathrm{cov}}^{(\varepsilon)}$ are finite, $P^\mathcal{T}_\rho$-a.s. Moreover, since $X^\mathcal{T}$ satisfies the commute time identity, i.e.\
\begin{equation}\label{commute}
E_x^\mathcal{T}(\tau_y)+E_y^\mathcal{T}(\tau_x)=2d_\mathcal{T}(x,y),\qquad \forall x,y\in\mathcal{T},
\end{equation}
where $E_x^\mathcal{T}$ is the expectation under $P_x^\mathcal{T}$ and $\tau_x:=\inf\{t\geq 0:\:X^\mathcal{T}_t=x\}$ is the hitting time of $x$ by $X^\mathcal{T}$ (see \cite[equation (2.17)]{CHK}, for example), it further holds that, under Assumption \ref{assu:1}, both $\tau_{\mathrm{cov}}^{(\varepsilon)}(\delta)$ and $\tau_{\mathrm{cov}}(\delta)$ are also finite, $P^\mathcal{T}_\rho$-a.s.

With $\tau_{\mathrm{cov}}^{(\varepsilon)}(\delta)$ and $\tau_{\mathrm{cov}}(\delta)$ defined, we are now able to introduce our second main assumption. This excludes the possibility that, at time $\tau_{\mathrm{cov}}^{(\varepsilon)}(\delta)$, the local time $L^\mathcal{T}$ takes the value zero on a small set.

\begin{assu}\label{assu:2} For each $\varepsilon\in (0,\infty)\cap\mathbb{Q}$,  $2\delta\in (0,d_\mathcal{T}(\rho,x_1))\cap\mathbb{Q}$,
\[P_\rho^\mathcal{T}\left(E_1(\varepsilon,\delta)\cup E_2(\varepsilon,\delta)\right)=1,\]
where
\begin{equation}\label{e1}
E_1(\varepsilon,\delta):=\left\{L^\mathcal{T}_{\tau_{\mathrm{cov}}^{(\varepsilon)}(\delta)}(x)>0,\:\forall x\in\mathcal{T}\right\},
\end{equation}
\begin{equation}\label{e2}
E_2(\varepsilon,\delta):=\left\{\exists x\in\mathcal{T},\:r>0:\:L^\mathcal{T}_{\tau_{\mathrm{cov}}^{(\varepsilon)}(\delta)}(y)=0,\:\forall y\in B_\mathcal{T}(x,r)\right\}.
\end{equation}
\end{assu}

In order to explain how we apply this assumption, which we will do in Proposition \ref{propmain}, we need several preparatory lemmas (none of which require Assumption \ref{assu:2} itself). The first of these lemmas establishes that $\tau_{\mathrm{cov}}(\delta)$ is indeed an approximation of $\tau_{\mathrm{cov}}$. This is a general result in the sense that we need not make use of either of our general assumptions.

\begin{lem}\label{taudelcon} 
$P_\rho^\mathcal{T}$-a.s.,\ as $\delta \to 0$, it holds that 
\[\tau_{\mathrm{cov}}(\delta)\rightarrow \tau_{\mathrm{cov}}.\]
\end{lem}
\begin{proof}
Clearly $\tau_{\mathrm{cov}}(\delta)$ decreases monotonically as $\delta\rightarrow 0$, $P_\rho^\mathcal{T}$-a.s.,\ and so it will suffice to show the desired convergence statement holds in $P_\rho^\mathcal{T}$-probability. Conditioning on the value of $X^\mathcal{T}_{\tau_{\mathrm{cov}}}$, from the strong Markov property we obtain: for any $\eta>0$,
\[P_\rho^\mathcal{T}\left(\left|\tau_{\mathrm{cov}}(\delta) -\tau_{\mathrm{cov}}\right|>\eta\right)\leq \sup_{x\in\mathcal{T}}P_x^\mathcal{T}\left(\tau_{\Delta_\delta(x)}>\eta\right).\]
Now, by applying Markov's inequality and the commute time identity \eqref{commute}, we have that the right-hand side here is bounded above by $2\delta\eta^{-1}$. Hence, it converges to zero as $\delta\rightarrow0$. This is enough to complete the proof.
\end{proof}

Next, we rule out the possibility that $X^\mathcal{T}$ is located at the point $x_{\mathrm{mid}}$, which we recall is the midpoint of $[[\rho,x_1]]_{\mathcal{T}}$, at time $\tau_{\mathrm{cov}}$.

\begin{lem}\label{midlem} If Assumption \ref{assu:1} holds, then
\[P_\rho^\mathcal{T}\left(X^\mathcal{T}_{\tau_{\mathrm{cov}}}=x_{\mathrm{mid}}\right)=0.\]
\end{lem}
\begin{proof}
We may assume that, for each fixed point $x\in\mathcal{T}$, the local time $(L_t^{\mathcal{T}}(x))_{t \geq0}$ is a continuous additive functional satisfying
$P_x^\mathcal{T}$-a.s.,\
\begin{equation}\label{ltinc}
\inf\left\{t\geq 0:\:L_t^\mathcal{T}(x)>0\right\}=0,
\end{equation}
see \cite[Section 3.6]{MR}, for example. It is also straightforward to check from \eqref{odf} and the continuity of the processes $X^\mathcal{T}$ and $L^\mathcal{T}$ that, $P_\rho^\mathcal{T}$-a.s.,\ $L^\mathcal{T}_{\tau_x}(x)=0$. Hence, by applying the strong Markov property at time $\tau_{x_{\mathrm{mid}}}$, we deduce that, $P_\rho^\mathcal{T}$-a.s.,\
\begin{equation}\label{ltinc2}
\inf\left\{t\geq 0:\:L_t^\mathcal{T}(x_{\mathrm{mid}})>0\right\}=\tau_{x_{\mathrm{mid}}}.
\end{equation}
Applying the continuity of $X^\mathcal{T}$, it moreover $P_\rho^\mathcal{T}$-a.s.\ holds that $\tau_{x_{\mathrm{mid}}}<\tau_{x_1}\leq \tau_{\mathrm{cov}}$. Hence, writing $\tau$ for $(\tau_{x_{\mathrm{mid}}}+\tau_{\mathrm{cov}})/2$, we can conclude that, $P_\rho^\mathcal{T}$-a.s.,\ $L_{\tau}^\mathcal{T}(x_{\mathrm{mid}})>0$. Since the local times are jointly continuous by Lemma~\ref{lem:contlt}, it follows that, $P_\rho^\mathcal{T}$-a.s.,\ for some $r>0$,
\[L^\mathcal{T}_{\tau}(x)>0,\qquad \forall x\in B_\mathcal{T}(x_{\mathrm{mid}},r).\]
From this, it is an elementary application of the occupation density formula \eqref{odf} to deduce that, $P_\rho^\mathcal{T}$-a.s.,\ for some $r>0$, $\tau_x\leq \tau$ for all $x\in B_\mathcal{T}(x_{\mathrm{mid}},r)$. The result is a straightforward consequence of this (and the continuity of $X^\mathcal{T}$).
\end{proof}

We now present a useful property concerning the uniformity of the distribution of the time to hit $\Delta_\delta(x)$ from $x$. Again, this requires neither Assumption \ref{assu:1} nor \ref{assu:2}.

\begin{lem}\label{lowerlem} For each $\delta\in (0,d_\mathcal{T}(\rho,x_1)/2)$, it holds that
\[\sup_{x\in\mathcal{T}}P_x^\mathcal{T}\left(\eta>\tau_{\Delta_\delta(x)}\right)\rightarrow 0,\]
as $\eta\rightarrow 0$.
\end{lem}
\begin{proof}
We first claim that if $x_n\rightarrow x$ and $y_n\rightarrow y$ in $(\mathcal{T},d_\mathcal{T})$, then
\begin{equation}\label{claim}
P^{\mathcal{T}}_{x_n}\left(\tau_{y_n}\in \cdot\right)\rightarrow P^{\mathcal{T}}_{x}\left(\tau_{y}\in \cdot\right)
\end{equation}
weakly as probability measures on $\mathbb{R}$. To prove this, note that, under $P^{\mathcal{T}}_{x_n}$, it holds that $\tau_{y_n}\leq \tau_1^n+\tau_2^n+\tau_3^n$, where $\tau_1^n$ is the time taken by $X^\mathcal{T}$ to hit $x$, $\tau_2^n$ is the time taken from then to hit $y$ and $\tau_3^n$ is the subsequent time to hit $y_n$. Clearly $\tau_2^n$ under $P^{\mathcal{T}}_{x_n}$ has the distribution of $\tau_{y}$ under $P^{\mathcal{T}}_{x}$. Moreover, applying \eqref{commute}, we find that, for any $\eta>0$,
\[P^{\mathcal{T}}_{x_n}\left(\tau_1^n+\tau_3^n\geq \eta\right)\leq P^{\mathcal{T}}_{x_n}\left(\tau_x\geq \eta/2\right)+P^{\mathcal{T}}_{y}\left(\tau_{y_n}\geq \eta/2\right)\leq\frac{4\left(d_{\mathcal{T}}(x_n,x)+d_{\mathcal{T}}(y_n,y)\right)}{\eta}\rightarrow0,\]
as $n\rightarrow\infty$. Hence we obtain that, for any $\eta>0$,
\[\limsup_{n\rightarrow\infty}P^{\mathcal{T}}_{x_n}\left(\tau_{y_n}\geq \eta\right)\leq P^{\mathcal{T}}_{x}\left(\tau_y\geq \eta\right).\]
By considering a similar upper bound for $\tau_y$ under $P^\mathcal{T}_x$, we are also able to deduce that, for any $\eta>0$,
\[\liminf_{n\rightarrow\infty}P^{\mathcal{T}}_{x_n}\left(\tau_{y_n}> \eta\right)\geq P^{\mathcal{T}}_{x}\left(\tau_y>\eta\right).\]
Combining the two previous displays yields the claim.

Now, for each $\eta>0$, define $\phi_\eta(u):=\max\{\min\{1,2-u/\eta\},0\}$ for $u\geq 0$; this is a continuous bounded function that strictly dominates $\mathbf{1}_{\{u<\eta\}}$. Hence to prove the lemma, it will suffice to show that
\[\sup_{x\in\mathcal{T}}E_x^\mathcal{T}\left(\phi_\eta\left(\tau_{\Delta_\delta(x)}\right)\right)\rightarrow 0,\]
as $\eta\rightarrow 0$. To simplify notation, we define $f_\eta(x):=E_x^\mathcal{T}(\phi_\eta(\tau_{\Delta_\delta(x)}))$. Since for any $x\in\mathcal{T}$, $\tau_{\Delta_\delta(x)}$ is strictly positive, $P_x^\mathcal{T}$-a.s.,\ we have from the monotone convergence theorem that $f_\eta$ decreases monotonically to zero pointwise as $\eta\rightarrow 0$. To extend this to uniform convergence, we will appeal to Dini's theorem, which requires continuity. Whilst we can not apply this on the whole space $\mathcal{T}$, the required continuity condition does hold on certain subsets. In particular, define
\[A:=\left\{x\in \mathcal{T}:\:d_\mathcal{T}(\rho,b^\mathcal{T}(\rho,x,x_1))\geq d_\mathcal{T}(\rho,x_1)/2\right\}.\]
By construction, the function $\Delta_\delta$ is continuous on $A$. Thus it follows from \eqref{claim} that, for each $\eta$, $f_\eta$ is also continuous on $A$. Since $A$ is clearly a compact subset of $\mathcal{T}$, we obtain as a consequence that
\[\sup_{x\in A}f_\eta(x)\rightarrow 0,\]
as $\eta\rightarrow 0$. Similarly, set
\[\tilde{A}:=\left\{x\in \mathcal{T}:\:d_\mathcal{T}(\rho,b^\mathcal{T}(\rho,x,x_1))\leq d_\mathcal{T}(\rho,x_1)/2\right\},\]
and consider the function $\tilde{\Delta}_\delta$ given by changing $\geq$ and $<$ to $>$ and $\leq$, respectively, in the definition of $\Delta_\delta$. Note that $\tilde{\Delta}_\delta$ is continuous on $\tilde{A}$ and $\tilde{\Delta}_\delta={\Delta}_\delta$ on $A^c$. Thus, proceeding as in the argument for the set $A$, we find that
\[\sup_{x\in A^c}f_\eta(x)\leq \sup_{x\in \tilde{A}}E_x^\mathcal{T}\left(\phi_\eta\left(\tau_{\tilde{\Delta}_\delta(x)}\right)\right)\rightarrow 0,\]
as $\eta\rightarrow0$. This completes the proof.
\end{proof}

Applying the previous two lemmas, we are able to prove the following result, which, in conjunction with Lemma \ref{taudelcon}, gives that $\tau_{\mathrm{cov}}^{(\varepsilon)}(\delta)$ is also a good approximation for $\tau_{\mathrm{cov}}$ when both $\varepsilon$ and $\delta$ are small.

\begin{lem}\label{tauepsconv} If Assumption \ref{assu:1} holds, then, for each $\delta\in (0,d_\mathcal{T}(\rho,x_1)/2)$,
\[\tau_{\mathrm{cov}}^{(\varepsilon)}(\delta)\rightarrow \tau_{\mathrm{cov}}(\delta)\]
in $P_\rho^\mathcal{T}$-probability as $\varepsilon\rightarrow 0$ .
\end{lem}
\begin{proof}
Clearly $\tau_{\mathrm{cov}}^{(\varepsilon)}\leq \tau_{\mathrm{cov}}$ for every $\varepsilon>0$, $P_\rho^\mathcal{T}$-a.s. Moreover, $\tau_{\mathrm{cov}}^{(\varepsilon)}$ monotonically increases as $\varepsilon$ decreases, and so has a $P_\rho^\mathcal{T}$-a.s.\ limit as $\varepsilon\rightarrow0$, $\tau_{\mathrm{cov}}^{(0)}$, say. Suppose $\tau_{\mathrm{cov}}^{(0)}<\tau_{\mathrm{cov}}$. If $t\in [\tau_{\mathrm{cov}}^{(0)},\tau_{\mathrm{cov}})$, then $X_{[0,t]}^\mathcal{T}\supseteq \cup_{\varepsilon>0}\mathcal{T}^{(\varepsilon)}$, and so, by the denseness of $(x_i)_{i\geq 0}$ and the continuity of $X^\mathcal{T}$, $X_{[0,t]}^\mathcal{T}=\mathcal{T}$, but this contradicts that $t<\tau_{\mathrm{cov}}$. Hence, it must $P_\rho^\mathcal{T}$-a.s.\ hold that
\begin{equation}\label{star1}
\tau_{\mathrm{cov}}^{(\varepsilon)}\rightarrow \tau_{\mathrm{cov}},
\end{equation}
as $\varepsilon\rightarrow0$. Applying the continuity of $X^\mathcal{T}$, it is thus also the case that $X^{\mathcal{T}}_{\tau_{\mathrm{cov}}^{(\varepsilon)}}\rightarrow X^{\mathcal{T}}_{\tau_{\mathrm{cov}}}$, and moreover, from Lemma \ref{midlem} and the continuity of $\Delta_\delta$ on $\mathcal{T}\backslash \{x_{\mathrm{mid}}\}$, we deduce that, for every $\delta\in (0,d_\mathcal{T}(\rho,x_1)/2)$, $P_\rho^\mathcal{T}$-a.s.,\
\begin{equation}\label{star2}
\Delta_\delta\left(X^{\mathcal{T}}_{\tau_{\mathrm{cov}}^{(\varepsilon)}}\right)\rightarrow\Delta_\delta\left(X^{\mathcal{T}}_{\tau_{\mathrm{cov}}}\right),
\end{equation}
as $\varepsilon\rightarrow0$.

We will next apply Lemma \ref{lowerlem} and \eqref{star1} to check that, for every $\delta\in (0,d_\mathcal{T}(\rho,x_1)/2)$,
\begin{equation}\label{to1}
P_\rho^\mathcal{T}\left(\tau_{\mathrm{cov}}\leq \tau_{\mathrm{cov}}^{(\varepsilon)}(\delta)\right)\rightarrow 1,
\end{equation}
as $\varepsilon\rightarrow0$. In order to prove this, we start by observing that, for every $\eta>0$,
\begin{eqnarray*}
  P_\rho^\mathcal{T}\left(\tau_{\mathrm{cov}}>\tau_{\mathrm{cov}}^{(\varepsilon)}(\delta)\right) &= &
    P_\rho^\mathcal{T}\left(\tau_{\mathrm{cov}}-\tau_{\mathrm{cov}}^{(\varepsilon)}> \tau_{\mathrm{cov}}^{(\varepsilon)}(\delta)-\tau_{\mathrm{cov}}^{(\varepsilon)}\right) \\
   &\leq& P_\rho^\mathcal{T}\left(\tau_{\mathrm{cov}}-\tau_{\mathrm{cov}}^{(\varepsilon)}>\eta\right)
   +P_\rho^\mathcal{T}\left(\tau_{\mathrm{cov}}^{(\varepsilon)}(\delta)-\tau_{\mathrm{cov}}^{(\varepsilon)}<\eta\right).
\end{eqnarray*}
The first term here converges to zero as $\varepsilon\rightarrow0$ by \eqref{star1}. As for the second term, this is bounded above by $\sup_{x\in\mathcal{T}}P_x^\mathcal{T}(\eta>\tau_{\Delta_\delta(x)})$, which, by Lemma \ref{lowerlem}, can be made arbitrarily small by choosing $\eta$ appropriately. This establishes \eqref{to1}, as desired.

Now, consider the following two stopping times:
\[\tau_1 =\inf\left\{t\geq \tau_{\mathrm{cov}}:\:X^\mathcal{T}_t\in \left\{\Delta_\delta\left(X^{\mathcal{T}}_{\tau_{\mathrm{cov}}^{(\varepsilon)}}\right),\Delta_\delta\left(X^{\mathcal{T}}_{\tau_{\mathrm{cov}}}\right)
\right\}\right\},\]
\[\tau_2 =\inf\left\{t\geq \tau_{\mathrm{cov}}:\:X^\mathcal{T}_{[\tau_{\mathrm{cov}},t]}\supseteq \left\{\Delta_\delta\left(X^{\mathcal{T}}_{\tau_{\mathrm{cov}}^{(\varepsilon)}}\right),\Delta_\delta\left(X^{\mathcal{T}}_{\tau_{\mathrm{cov}}}\right)
\right\}\right\}.\]
Recall the definitions of $\tau_{\mathrm{cov}}^{(\varepsilon)}(\delta)$ and $\tau_{\mathrm{cov}}(\delta)$ from \eqref{tced}
 and \eqref{tcd}, respectively. Clearly, on the event $\{\tau_{\mathrm{cov}}\leq\tau_{\mathrm{cov}}^{(\varepsilon)}(\delta)\}$, it holds that
\[\tau_2-\tau_1=\left|\tau_{\mathrm{cov}}^{(\varepsilon)}(\delta)- \tau_{\mathrm{cov}}(\delta)\right|.\]

From \eqref{to1} and the conclusion of the previous paragraph, to complete the proof, it will consequently suffice to show that $\tau_2-\tau_1\rightarrow 0$ as $\varepsilon\rightarrow 0$ in $P_\rho^\mathcal{T}$-probability. To this end, we define
\[d(\varepsilon,\delta)=d_\mathcal{T}\left(\Delta_\delta\left(X^{\mathcal{T}}_{\tau_{\mathrm{cov}}^{(\varepsilon)}}\right),\Delta_\delta\left(X^{\mathcal{T}}_{\tau_{\mathrm{cov}}}\right)\right),\]
and then apply the strong Markov property at $\tau_1$ (noting that both $\Delta_\delta(X^{\mathcal{T}}_{\tau_{\mathrm{cov}}^{(\varepsilon)}})$ and $\Delta_\delta(X^{\mathcal{T}}_{\tau_{\mathrm{cov}}})$ are $(X^\mathcal{T}_t)_{t=0}^{\tau_{\mathrm{cov}}}$-measurable) to deduce that, for every $\eta>0$,
\[P_\rho^\mathcal{T}\left(\tau_2-\tau_1>\eta\right)  \leq  E_\rho^\mathcal{T}\left(\sup_{\substack{x,y\in\mathcal{T}:\\d_\mathcal{T}(x,y)\leq d(\varepsilon,\delta)}} P_x^\mathcal{T}\left(\tau_y>\eta\right)\right)\leq E_\rho^\mathcal{T}\left( \min\{1,2\eta^{-1}d(\varepsilon,\delta)\}\right),\]
where the second inequality follows from \eqref{commute}. In conjunction with \eqref{star2}, this implies the result.
\end{proof}

Putting the above pieces together, we arrive at the main result of this subsection.

\begin{propn}\label{propmain}
If Assumptions \ref{assu:1} and \ref{assu:2} hold, then, $P^\mathcal{T}_\rho$-a.s.,\
\[\tau_{\mathrm{cov}}=\inf\left\{t\geq0:\: L^\mathcal{T}_t(x)>0,\:\forall x\in\mathcal{T}\right\}.\]
\end{propn}
\begin{proof} Fix rational $\delta\in (0,d_\mathcal{T}(\rho,x_1)/2)$ and suppose there exists a rational sequence $(\varepsilon_i)$ with $\varepsilon_i\rightarrow 0$ such that
\[P_\rho^\mathcal{T}\left(E_1(\varepsilon_i,\delta)\mbox{ holds for finitely many $i$}\right)>0.\]
Then, from Assumption \ref{assu:2}, it follows that
\begin{equation}\label{e2hold}
P_\rho^\mathcal{T}\left(E_2(\varepsilon_i,\delta)\mbox{ holds for all large $i$}\right)>0.
\end{equation}
Clearly, on $E_2(\varepsilon_i,\delta)$, the occupation density formula \eqref{odf} and the continuity of $X^\mathcal{T}$ imply that
\[\tau_{\mathrm{cov}}\geq \tau_{\mathrm{cov}}^{(\varepsilon_i)}(\delta).\]
Moreover, by Lemma \ref{tauepsconv} and \cite[Lemma 4.2]{kall}, it is possible to find a (deterministic) subsequence $(\varepsilon_{i_j})$ such that $\tau_{\mathrm{cov}}^{(\varepsilon_{i_j})}(\delta)$ converges to $\tau_{\mathrm{cov}}(\delta)$ as $j\rightarrow\infty$, $P_\rho^\mathcal{T}$-a.s. Consequently, from \eqref{e2hold}, we obtain that
\[P_\rho^\mathcal{T}\left(\tau_{\mathrm{cov}}\geq \tau_{\mathrm{cov}}(\delta)\right)>0.\]
However, from the continuity of $X^\mathcal{T}$, it is clear that $\tau_{\mathrm{cov}}(\delta)>\tau_{\mathrm{cov}}$, $P_\rho^\mathcal{T}$-a.s.,\ and so the above inequality is a contradiction. Thus we have established that, for any rational sequence $(\varepsilon_i)$ with $\varepsilon_i\rightarrow 0$,
\[P_\rho^\mathcal{T}\left(E_1(\varepsilon_i,\delta)\mbox{ holds for infinitely many $i$}\right)=1.\]
Now, again appealing to Lemma \ref{tauepsconv} and \cite[Lemma 4.2]{kall}, we can find a (deterministic) sequence of rationals $(\varepsilon_{i})$ such that $\tau_{\mathrm{cov}}^{(\varepsilon_{i})}(\delta)$ converges to $\tau_{\mathrm{cov}}(\delta)$ as $i\rightarrow\infty$, $P_\rho^\mathcal{T}$-a.s. On the event that $E_1(\varepsilon_i,\delta)$ holds, we clearly have that
\[\inf\left\{t\geq0:\: L^\mathcal{T}_t(x)>0,\:\forall x\in\mathcal{T}\right\}\leq \tau_{\mathrm{cov}}^{(\varepsilon_{i})}(\delta).\]
Since this holds infinitely often along the sequence with probability one, it follows that, $P_\rho^\mathcal{T}$-a.s.,\
\[\inf\left\{t\geq0:\: L^\mathcal{T}_t(x)>0,\:\forall x\in\mathcal{T}\right\}\leq \limsup_{i\rightarrow\infty}\tau_{\mathrm{cov}}^{(\varepsilon_{i})}(\delta)=\tau_{\mathrm{cov}}(\delta).\]
As $\delta\in (0,d_\mathcal{T}(\rho,x_1)/2)$ can be chosen as an arbitrarily small rational number, it follows from this and Lemma \ref{taudelcon} that, $P_\rho^\mathcal{T}$-a.s.,\
\begin{equation}\label{bb1}
\inf\left\{t\geq0:\: L^\mathcal{T}_t(x)>0,\:\forall x\in\mathcal{T}\right\}\leq\tau_{\mathrm{cov}}.
\end{equation}

On the other hand, if $t$ is such that $L^\mathcal{T}_t(x)>0$ for all $x\in\mathcal{T}$, then it readily follows from the occupation density formula \eqref{odf} and the continuity of $X^\mathcal{T}$ that $t\geq \tau_{\mathrm{cov}}$. Hence, it $P_\rho^\mathcal{T}$-a.s.\ holds that
\begin{equation}\label{bb2}
\inf\left\{t\geq0:\: L^\mathcal{T}_t(x)>0,\:\forall x\in\mathcal{T}\right\}\geq\tau_{\mathrm{cov}}.
\end{equation}

Combining \eqref{bb1} and \eqref{bb2} yields the desired result.
\end{proof}

\subsection{A Ray-Knight theorem for trees}\label{sec:rktheorems}

The aim of this section is to give a probabilistic characterisation of the local time process $L^\mathcal{T}$ at time ${\tau_{\mathrm{cov}}^{(\varepsilon)}(\delta)}$ that will facilitate the proof that Assumption \ref{assu:2} holds in our setting. The main result, see Proposition \ref{rk} below, is in the spirit of the classical Ray-Knight theorem, which relates the local times of one-dimensional Brownian motion at a hitting time and a BESQ process. (For background on Ray-Knight theorems, see \cite[Section XI.2]{RevuzYor} or \cite[Chapter 2]{MR}, for example; of the results that appear there, the statement closest to ours is probably \cite[Theorem 2.6.3]{MR}. We also note that a related version of the second Ray-Knight theorem on tree-like spaces is studied in \cite{LZ}.) To state our conclusion, we first need to introduce several additional objects: a certain Brownian motion on $\mathcal{T}^{(\varepsilon)}$, a description of the connected components of $\mathcal{T}\backslash{\mathcal{T}}^{(\varepsilon)}$, and the notion of a tree-indexed Bessel process.

Recall the definition of $\mathcal{T}^{(\varepsilon)}$ from \eqref{teps}. Clearly, for each $\varepsilon>0$, this is a real tree of finite and strictly positive length. Hence the associated one-dimensional Hausdorff measure $\lambda^{(\varepsilon)}$ is a finite Borel measure of full support on $\mathcal{T}^{(\varepsilon)}$. Consequently, similarly to the discussion at the start of Section \ref{sec:bmdefn}, it is possible to define a canonical Brownian motion $((X^{(\varepsilon)}_t)_{t\geq 0},(P^{(\varepsilon)}_x)_{x\in{\mathcal{T}^{(\varepsilon)}}})$ from $(\mathcal{T}^{(\varepsilon)},d_\mathcal{T},\lambda^{(\varepsilon)})$. Since $\lambda^{(\varepsilon)}$ is readily checked to satisfy Assumption \ref{assu:1}, we may assume that $X^{(\varepsilon)}$ admits jointly continuous local times $(L^{(\varepsilon)}_t(x))_{x\in {\mathcal{T}^{(\varepsilon)}},t\geq 0}$. As one further piece of notation concerning the process $X^{(\varepsilon)}$, we define $\tau(\varepsilon,\delta)$ in analogous way to how $\tau_{\mathrm{cov}}^{(\varepsilon)}(\delta)$ was defined from $X^\mathcal{T}$. In particular, if $\tau_{\mathrm{cov}}(X^{(\varepsilon)})$ is the cover time of $X^{(\varepsilon)}$ (defined as at \eqref{taucovdef}), then
\begin{equation}\label{teddef}
\tau(\varepsilon,\delta):=\inf\left\{t\geq \tau_{\mathrm{cov}}(X^{(\varepsilon)}):\:X^{(\varepsilon)}_t=\Delta_\delta\left(X^{(\varepsilon)}_{\tau_{\mathrm{cov}}(X^{(\varepsilon)})}\right)\right\}.    
\end{equation}
Note we require $\delta\in (0,d_\mathcal{T}(\rho,x_1)/2)$ to ensure this is well-defined.

Next, we observe that because $(\mathcal{T},d_\mathcal{T})$ is a compact real tree, it has at most a countable number of branch points (see \cite[Lemma 3.1]{DW}, for example). This means that, for each $\varepsilon>0$, the set $\mathcal{T}\backslash{\mathcal{T}}^{(\varepsilon)}$ consists of at most a countable number of connected components, $(\mathcal{T}_i^o)_{i\in \mathcal{I}^{(\varepsilon)}}$ say. We write $(\mathcal{T}_i)_{i\in \mathcal{I}^{(\varepsilon)}}$ for the closures of these components in $\mathcal{T}$, which means adding one additional point to each of the sets. Namely, for each $i$, there exists a $\rho_i\in \mathcal{T}^{(\varepsilon)}$ such that $\mathcal{T}_i\backslash \mathcal{T}_i^o=\{\rho_i\}$. We then have that each $(\mathcal{T}_i,d_\mathcal{T})$ is a compact real tree (again, we do not introduce specific notation for the restriction of $d_\mathcal{T}$), and we will consider the root of each to be the relevant $\rho_i$.

As the final ingredient for stating Proposition \ref{rk}, we need the notion of a tree-indexed (squared) Bessel process. To introduce this, we start by recalling that a BESQ process of dimension $d\geq 0$, started from $z_0\geq 0$, can be characterised as the unique positive strong solution $(\xi_t)_{t\geq 0}$ of the stochastic differential equation
\[\xi_t:=z_0+2\int_0^t\sqrt{\xi_s}dB_s+d \cdot t,\]
where $(B_t)_{t\geq 0}$ is a standard one-dimensional Brownian motion. (See \cite[Section 14.2]{MR} for a brief introduction to such processes with a focus on aspects that are useful for the study of local times, and also \cite[Chapter XI]{RevuzYor} for a more in-depth discussion.) We will denote the law of such a process by $\mathrm{BESQ}^d(z_0)$. Heuristically, a $\mathcal{T}$-indexed $\mathrm{BESQ}^d(z_0)$ process is a random function $(\xi_x^\mathcal{T})_{x\in\mathcal{T}}$ such that, along each arc $[[\rho,x]]_\mathcal{T}$, $(\xi_y^\mathcal{T})_{y\in [[\rho,x]]_\mathcal{T}}$ is a one-dimensional $\mathrm{BESQ}^d(z_0)$ process (where the arc is parameterised using the distance from the root), and, if $x$ and $x'$ are two distinct vertices, then conditional on $\xi^\mathcal{T}_{b^\mathcal{T}(\rho,x,x')}$, the values of $(\xi_y^\mathcal{T})_{y\in [[b^\mathcal{T}(\rho,x,x'),x]]_\mathcal{T}}$ and $(\xi_y^\mathcal{T})_{y\in [[b^\mathcal{T}(\rho,x,x'),x']]_\mathcal{T}}$ are independent. When the number of leaves in the tree is finite, such as in each $\mathcal{T}^{(\varepsilon)}$, there is no problem in making this definition rigorous by applying the Markov property for one-dimensional Bessel processes (see \cite[Chapter XI]{RevuzYor}). However, for a general compact real tree, some care is needed. More generally, we will say that $(\xi_x^\mathcal{T})_{x\in\mathcal{T}}$ is a $\mathcal{T}$-indexed $\mathrm{BESQ}^d(z_0)$ process if it is a continuous stochastic process such that, for each subtree $\mathcal{T}'\subseteq\mathcal{T}$ that spans a finite set of vertices including the root, $(\xi_x^\mathcal{T})_{x\in\mathcal{T}'}$ satisfies the properties just described. Following the work of Le Gall \cite{LG}, we will later give a construction of such a process in our setting (i.e.\ where the underlying tree is the Brownian CRT) using the notion of a Brownian snake. For now, though, it is convenient to give a preliminary result that ensures certain 0-dimensional tree-indexed BESQ processes are well defined. (We expect that the proof could be modified to address BESQ processes of other dimensions, but we restrict to the 0-dimensional case since that is slightly simpler to handle and is all we need here.)

\begin{lem}\label{besqlem} Suppose Assumption \ref{assu:1} holds. It is then possible to construct, for any $z_0\geq 0$, a $\mathcal{T}$-indexed $\mathrm{BESQ}^0(z_0)$ process. The law of such a process is unique.
\end{lem}
\begin{proof}
Since for $z_0=0$, the one-dimensional $\mathrm{BESQ}^0(z_0)$ process remains at zero (see \cite[Theorem XI.1.5]{RevuzYor}, the generalisation to trees is trivial in this case. So, we fix $z_0>0$. The key input to the proof in this case will be the following estimate for a one-dimensional $\mathrm{BESQ}^0(z_0)$ process, $(\xi_t)_{t\geq 0}$: for any $T\geq 0$,
\begin{equation}\label{bes0bound}
\mathbf{P}\left(\left|\xi_s-\xi_t\right|>\varepsilon\right)\leq C\left(1+\frac{T}{t-s}\right)\exp\left(- \frac{c\varepsilon}{\sqrt{t-s}}\right),\qquad \forall \varepsilon>0,\:0\leq s<t\leq T,
\end{equation}
where the constants $C$ and $c$ depend only upon $z_0$ and $T$, and here we write $\mathbf{P}$ for the probability measure on the space upon which $\xi$ is defined. To prove this, we use the fact that the $\mathrm{BESQ}^0(z_0)$ process can, up to the hitting time of 0, be written as a time-change of one-dimensional Brownian motion $(B_t)_{t\geq 0}$ started from $z_0$, with speed measure given by $\frac{1}{4x}dx$, see \cite[Section V.47.27 and V.48]{RW2}. In particular, if
\[A_t:=\left\{
         \begin{array}{ll}
           \int_0^t\frac{1}{4B_s}ds, & \hbox{for $t<\tau_0^B$}, \\
           \infty, & \hbox{for $t\geq\tau_0^B$},
         \end{array}
       \right.\]
where $\tau_0^B$ is the hitting time of 0 by $B$, and $\alpha(t):=\inf\{s:\:A_s>t\}$ is its right-continuous inverse, then it is possible to construct a version of the $\mathrm{BESQ}^0(z_0)$ process $\xi$ by setting  $\xi_t=B_{\alpha(t)}$. Now, on the event that $\sup_{s\leq T_0}B_s \leq \lambda$, we clearly have that $A_t-A_s\geq (t-s)/4\lambda$ for any $0\leq s< t \leq \min\{T_0,\tau_0^B\}$. Hence $\alpha(t)-\alpha(s)\leq 4\lambda(t-s)$ for any $0\leq s<t\leq T_0/4\lambda$. Consequently, taking $T_0:=4\lambda T$, we deduce that: for any $0\leq s<t\leq T$ and $\lambda>z_0$,
\begin{eqnarray*}
\lefteqn{\mathbf{P}\left(\left|\xi_s-\xi_t\right|>\varepsilon\right)}\\
&\leq& \mathbf{P}\left(\sup_{s\leq T_0}B_s > \lambda\right)+\mathbf{P}\left(\left|B_{\alpha(s)}-B_{\alpha(t)}\right|>\varepsilon,\:\sup_{s\leq T_0}B_s \leq \lambda\right)\\
&\leq &2\mathbf{P}\left(B_{T_0} > \lambda\right)+\mathbf{P}\left(\sup_{q\leq 4\lambda s}\sup_{r\leq 4\lambda(t-s)}|B_{q+r}-B_{q}|>\varepsilon\right)\\
&\leq &2\mathbf{P}\left(B_{1}-B_0 >\frac{\lambda-z_0}{2\sqrt{\lambda T}}\right)+\sum_{i=0}^{\lfloor\frac{s}{t-s}\rfloor}\mathbf{P}\left(\sup_{r\leq 8\lambda(t-s)}|B_{4\lambda(t-s)i+r}-B_{4\lambda(t-s)i}|>\varepsilon/2\right)\\
&\leq &2\mathbf{P}\left(B_{1}-B_0 >\frac{\lambda-z_0}{2\sqrt{\lambda T}}\right)+\left(1+\frac{T}{t-s}\right)\mathbf{P}\left(\sup_{r\leq 8\lambda(t-s)}|B_{r}-B_0|>\varepsilon/2\right)\\
&\leq &2\mathbf{P}\left(B_{1}-B_0 > \frac{\lambda-z_0}{2\sqrt{\lambda T}}\right)+4\left(1+\frac{T}{t-s}\right)\mathbf{P}\left(B_{1}-B_0 >\frac{\varepsilon}{2\sqrt{8\lambda(t-s)}}\right),
\end{eqnarray*}
where we apply a union bound to deduce the third inequality, and the reflection principle to bound the terms involving the supremum of Brownian motion. Now, since $B_1-B_0$ is simply a standard normal random variable, we have that $\mathbf{P}(B_1-B_0>x)\leq C e^{-x^2/2}$, and so applying the above bound with $\lambda=\varepsilon/\sqrt{t-s}$ gives the result at \eqref{bes0bound} when $\varepsilon/\sqrt{t-s}>z_0$. We can extend this to $\varepsilon/\sqrt{t-s}\leq z_0$ by adjusting $C$ as appropriate.

To apply \eqref{bes0bound} to deduce the existence of a $\mathcal{T}$-indexed $\mathrm{BESQ}^0(z_0)$ process, we will apply a version of the Kolmogorov-Chentsov continuity theorem, as appears in \cite[Section 3]{Noda}. To this end, for each $n\geq 0$, let $D_n''$ be a minimal $2^{-n}$ covering of $\mathcal{T}$ (i.e.,\ a collection $(x_i)_{i=1}^{N_n}$ of minimal size such that $\cup_{i=1}^{N_n}B_\mathcal{T}(x_i,2^{-n})=\mathcal{T})$. From Assumption \ref{assu:1}, it is an elementary exercise to check that $N_n\leq C2^{\eta n}$, where $\eta$ is the constant of that assumption. We also let $D_n'=\cup_{m\leq n}D_m''$, and then set $D_n$ to be the union of $D_n'$ and all those points at a distance that is an integer multiple of $2^{-n}$ from the root on one of the arcs $[[\rho,x]]_{\mathcal{T}}$ for some $x\in D_n'$. Note that $D_n$ forms an increasing sequence and each has size bounded by $C2^{(1+\eta)n}$. Finally, we set $D=\cup_{n\geq 0}D_n$, which is countable. If $\mathcal{T}_n:=\cup_{x\in D_n}[[\rho,x]]_\mathcal{T}$, then it is straightforward to construct a $\mathcal{T}_n$-indexed $\mathrm{BESQ}^0(z_0)$ process directly, by piecing together one-dimensional BESQ processes on each of the arcs. Moreover, by Kolmogorov's extension theorem, we can readily extend this construction to give a $\mathrm{BESQ}^0(z_0)$ on $\mathcal{T}_D:=\cup_{x\in D}[[\rho,x]]_\mathcal{T}$. Let us denote the latter process by $(Z_x^{\mathcal{T}_D})_{x\in \mathcal{T}_D}$. Now, since the diameter of $\mathcal{T}$ is finite, we obtain from \eqref{bes0bound} with $\varepsilon =(t-s)^{\frac{1}{2}-\delta}$ for some $\delta>0$ that:
\[\mathbf{P}\left(\left|\xi^{\mathcal{T}_D}_x-\xi^{\mathcal{T}_D}_y\right|>d_\mathcal{T}(x,y)^{\frac{1}{2}-\delta}\right)\leq C\exp\left(- cd_\mathcal{T}(x,y)^{-\delta}\right),\qquad \forall x,y\in\mathcal{T}_D,\]
where the constants here depend only upon the diameter of $\mathcal{T}$ and $z_0$. Hence we can apply \cite[Corollary 3.2]{Noda} to deduce that
\[\mathbf{P}\left(\sup_{\substack{x,y\in D:\\d_\mathcal{T}(x,y)\leq 2^{-(n-1)}}}\left|\xi^{\mathcal{T}_D}_x-\xi^{\mathcal{T}_D}_y\right|>C 2^{-(\frac{1}{2}-\delta)n}\right)\leq  Ce^{-2^{\delta n}}.\]
Since the upper bound here is summable in $n$, we deduce from a Borel-Cantelli argument that $(\xi_x^{\mathcal{T}_D})_{x\in D}$ is (almost-surely) uniformly continuous on $D$, and thus can be extended uniquely to a continuous function $(\xi_x^\mathcal{T})_{x\in \mathcal{T}}$. We further note that the inclusion of the additional points in $D_n$ as compared to $D_n'$ means that $D\cap[[\rho,x_i]]_\mathcal{T}$ is dense in $[[\rho,x_i]]_\mathcal{T}$ for each $x_i\in D$, and from this it is straightforward to check that $\xi_x^\mathcal{T}=\xi_x^{\mathcal{T}_D}$ on the whole tree $\mathcal{T}_D$.

To complete the proof and obtain uniqueness of the law of $(\xi_x^\mathcal{T})_{x\in \mathcal{T}}$, it will suffice to show that for any subtree $\mathcal{T}'\subseteq \mathcal{T}$ spanned by a finite set of vertices including the root, the process $(\xi_x^\mathcal{T})_{x\in \mathcal{T}'}$ is a $\mathcal{T}'$-indexed $\mathrm{BESQ}^0(z_0)$ process. (Since $\xi^\mathcal{T}$ is continuous, its distribution is specified by the distributions of such restrictions.) Suppose $\mathcal{T}'=\cup_{i=1}^M[[\rho,x_i]]_\mathcal{T}$. By the denseness of $D$, for each $i$, there exist sequences $(x_i^n)_{n\geq 0}$ in $D$ such that $d_\mathcal{T}(x_i^n,x_i)\rightarrow 0$ as $n\rightarrow\infty$. It follows that, $b^\mathcal{T}(\rho,x_i,x_i^n)\rightarrow x_i$. Now, from the construction of $(\xi_x^{\mathcal{T}_D})_{x\in \mathcal{T}_D}$, we obtain that the law of $(\xi_x^\mathcal{T})_{x\in \mathcal{T}_n''}$, where $\mathcal{T}_n'':=\cup_{i=1}^M[[\rho,b^\mathcal{T}(\rho,x_i,x_i^n)]]_\mathcal{T}$ is a $\mathcal{T}_n''$-indexed $\mathrm{BESQ}^0(z_0)$ process, for each $n$. From this and the continuity of $(\xi_x^\mathcal{T})_{x\in \mathcal{T}}$, one can deduce the corresponding result holds for $\mathcal{T}'$, as required.
\end{proof}

Often we will apply Lemma \ref{besqlem} to subsets of $\mathcal{T}$. Specifically, for each $x,y\in\mathcal{T}$, $x\neq y$, write $(\mathcal{T}_i)_{i\in \mathcal{I}_{x,y}}$ for the closures of the connected components of $\mathcal{T}\backslash[[x,y]]_\mathcal{T}$, defined similarly to $(\mathcal{T}_i)_{i\in \mathcal{I}^{(\varepsilon)}}$. If Assumption \ref{assu:1} holds for $\mathcal{T}$, then it also holds for the restriction of $\mu_\mathcal{T}$ to each component $\mathcal{T}_i$, $i\in \mathcal{I}_{x,y}$ (with possibly different constants for each). Hence, for any $z_0\geq 0$, it is possible to construct a $\mathcal{T}_i$-indexed $\mathrm{BESQ}^0(z_0)$ process for any $i\in \mathcal{I}_{x,y}$. Obviously, the same result also holds for each component of $(\mathcal{T}_i)_{i\in \mathcal{I}^{(\varepsilon)}}$. For all $i\in \mathcal{I}^{(\varepsilon)}$, the set $ \mathcal{T}_i \cap \mathcal{T}^{(\varepsilon)}$ consists of a single element, which we denote by $\rho_i$.

We are now ready to state the main result of this section, which relates local times and certain BESQ processes.

\begin{propn}\label{rk} Suppose Assumption \ref{assu:1} holds, $\varepsilon>0$, and $2\delta\in (0,d_\mathcal{T}(\rho,x_1))$. Under $P_\rho^\mathcal{T}$, the law of the local time process $L^\mathcal{T}$ at time ${\tau_{\mathrm{cov}}^{(\varepsilon)}(\delta)}$ is characterised as follows.
\begin{enumerate}
  \item[(a)] The law of $(L^\mathcal{T}_{{\tau_{\mathrm{cov}}^{(\varepsilon)}(\delta)}}(x))_{x\in \mathcal{T}^{(\varepsilon)}}$ is equal to that of $(L^{(\varepsilon)}_{{\tau(\varepsilon,\delta)}}(x))_{x\in \mathcal{T}^{(\varepsilon)}}$ under $P^{(\varepsilon)}_\rho$.
  \item[(b)] Conditional on $(L^\mathcal{T}_{{\tau_{\mathrm{cov}}^{(\varepsilon)}(\delta)}}(x))_{x\in \mathcal{T}^{(\varepsilon)}}$, the local times on the various components of $\mathcal{T}\backslash\mathcal{T}^{(\varepsilon)}$, that is, $(L^\mathcal{T}_{{\tau_{\mathrm{cov}}^{(\varepsilon)}(\delta)}}(x))_{x\in \mathcal{T}_i}$, $i\in \mathcal{I}^{(\varepsilon)}$, are independent, with each distributed as a $\mathcal{T}_i$-indexed 0-dimensional BESQ process started from $L^\mathcal{T}_{{\tau_{\mathrm{cov}}^{(\varepsilon)}(\delta)}}(\rho_i)$.
\end{enumerate}
\end{propn}

Towards checking part (a) of the above result, we will use the fact that $X^{(\varepsilon)}$ can be represented as the time change of $X^\mathcal{T}$. In particular, we introduce a continuous additive functional $A^{(\varepsilon)}$ by setting
\[A^{(\varepsilon)}_t:=\int_{\mathcal{T}^{(\varepsilon)}}L_t^\mathcal{T}(x)\lambda^{(\varepsilon)}(dx),\]
and write
\[\alpha^{(\varepsilon)}(t):=\inf\left\{s\geq 0:\:A^{(\varepsilon)}_s>t\right\}\]
for its right-continuous inverse. We then have the following connection.

\begin{lem}\label{ltcoupling} Suppose Assumption \ref{assu:1} holds. For any $\varepsilon>0$ and $x\in \mathcal{T}^{(\varepsilon)}$, it holds that the joint law of $(X^{(\varepsilon)},L^{(\varepsilon)})$ under $P^{(\varepsilon)}_x$ is equal to that of
\[\left(\left(X^{\mathcal{T}}_{\alpha^{(\varepsilon)}(t)}\right)_{t\geq 0},\left(L^{\mathcal{T}}_{\alpha^{(\varepsilon)}(t)}(x)\right)_{x\in {\mathcal{T}^{(\varepsilon)}},t\geq 0}\right)\]
under $P_x^\mathcal{T}$.
\end{lem}
\begin{proof}
The proof is identical to that of \cite[Lemmas 2.6 and 3.4]{Croydoncrt}, which dealt with realisations of the Brownian CRT, but applies to more general compact real trees. The one point that should be expanded upon is that the starting points of the two processes are the same, i.e.,\ that $X^{\mathcal{T}}_{\alpha^{(\varepsilon)}(0)}=x$ almost-surely under $P_x^\mathcal{T}$. For this, it is sufficient to show that $\alpha^{(\varepsilon)}(0)=0$ almost-surely under $P_x^\mathcal{T}$. However, it is an elementary exercise to check  this fact by applying that $\lambda^{(\varepsilon)}$ has full support, \eqref{ltinc} holds and the local times $L^\mathcal{T}$ are almost-surely jointly continuous.
\end{proof}

In what follows, we will often suppose that $X^{(\varepsilon)}$ (and its local times) are coupled with $X^{\mathcal{T}}$ (and its local times) in the way given by the previous result. Namely, we will suppose that both processes are built on the same probability space in such a way that we have equality between $X^{(\varepsilon)}_t$ and $X^{\mathcal{T}}_{\alpha^{(\varepsilon)}(t)}$ for all $t\geq 0$ (and similarly for the associated local times), almost-surely. When this is the case, we will simply say the processes are coupled, and write the joint law as $P_x^\mathcal{T}$ if the starting point is $x\in \mathcal{T}^{(\varepsilon)}$. Under such a coupling, we also have that various stopping times of interest for the processes align with respect to the time change. Recall $\tau_x$ is the hitting time of $x$ by $X^\mathcal{T}$. We will write $\tau_x^{(\varepsilon)}$ for the corresponding quantity for $X^{(\varepsilon)}$.

\begin{lem}\label{htident} Suppose Assumption \ref{assu:1} holds, $\varepsilon>0$,  $2\delta\in (0,d_\mathcal{T}(\rho,x_1))$, $x,y\in \mathcal{T}^{(\varepsilon)}$, and that $X^{(\varepsilon)}$ and $X^{\mathcal{T}}$ are coupled, each with starting point $x$. It then $P_x^\mathcal{T}$-a.s.\ holds that
\[\tau_y=\alpha^{(\varepsilon)}\left(\tau_y^{(\varepsilon)}\right),\qquad \tau_{\mathrm{cov}}^{(\varepsilon)}(\delta)=\alpha^{(\varepsilon)}\left(\tau(\varepsilon,\delta)\right),\]
where we recall the definition of $\tau(\varepsilon,\delta)$ from \eqref{teddef}.
\end{lem}
\begin{proof}
Under the coupling, we $P_x^\mathcal{T}$-a.s.\ have that
\[X^\mathcal{T}_{\alpha^{(\varepsilon)}\left(\tau_y^{(\varepsilon)}\right)}=X^{(\varepsilon)}_{\tau_y^{(\varepsilon)}}=y.\]
Hence $\alpha^{(\varepsilon)}(\tau_y^{(\varepsilon)})\geq \tau_y$. On the other hand, from the version of \eqref{ltinc2} with ${x_{\mathrm{mid}}}$ replaced by $y$ (and the continuity of the local times, as well as the full support of $\lambda^{(\varepsilon)}$), we have that $A^{(\varepsilon)}_{\tau_y+t}>A^{(\varepsilon)}_{\tau_y}=0$ for all $t>0$, $P_x^\mathcal{T}$-a.s. Hence, it $P_x^\mathcal{T}$-a.s.\ holds that $\alpha^{(\varepsilon)}\left(A^{(\varepsilon)}_{\tau_y}\right)=\tau_y$, and so
\[X^{(\varepsilon)}_{A^{(\varepsilon)}_{\tau_y}}=X^\mathcal{T}_{\alpha^{(\varepsilon)}\left(A^{(\varepsilon)}_{\tau_y}\right)}=X^{\mathcal{T}}_{\tau_y}=y,\]
which implies $A^{(\varepsilon)}_{\tau_y}\geq \tau_y^{(\varepsilon)}$, and in turn that $\tau_y\geq \alpha^{(\varepsilon)}(\tau_y^{(\varepsilon)})$. Thus we have completed the proof of the first equality.

To verify the second equality, we start by applying the first one to deduce that, $P_x^\mathcal{T}$-a.s.,\
\begin{align*}
\tau_{\mathrm{cov}}^{(\varepsilon)}&=\max_{i=1,\dots ,M(\varepsilon)}\tau_{x_i}=\max_{i=1,\dots ,M(\varepsilon)}\alpha^{(\varepsilon)}\left(\tau_{x_i}^{(\varepsilon)}\right)\\
&=\alpha^{(\varepsilon)}\left(\max_{i=1,\dots ,M(\varepsilon)}\tau_{x_i}^{(\varepsilon)}\right)=\alpha^{(\varepsilon)}\left(\tau_{\mathrm{cov}}\left(X^{(\varepsilon)}\right)\right),
\end{align*}
where we recall from \eqref{teps} that $\mathcal T^{(\varepsilon)} = \cup_{i=0}^{M(\varepsilon)}[[\rho,x_i]]_{\mathcal{T}}$. We then apply the strong Markov property at this time, together with another application of the first equality to deduce the result.
\end{proof}

With these preparations, we can prove the first part of Proposition \ref{rk}.

\begin{proof}[Proof of Proposition \ref{rk}(a)]
Supposing that $X^{(\varepsilon)}$ and $X^{\mathcal{T}}$ are coupled, each with starting point $\rho$, we have from Lemma \ref{htident} that, $P_\rho^\mathcal{T}$-a.s.,\
\[L^{(\varepsilon)}_{\tau(\varepsilon,\delta)}(x)=L^\mathcal{T}_{\alpha^{(\varepsilon)}\left(\tau(\varepsilon,\delta)\right)}(x)=L^\mathcal{T}_{\tau_{\mathrm{cov}}^{(\varepsilon)}(\delta)}(x),\qquad \forall x\in\mathcal{T}^{(\varepsilon)}.\]
Hence the claim follows from Lemma \ref{ltcoupling}.
\end{proof}

For the proof of part (b) of Proposition \ref{rk}, the additional ingredients we need are the following Ray-Knight-type result at a single hitting time (Lemma \ref{rkhit}) and certain additivity and Markov properties for tree-indexed BESQ processes (Lemma \ref{add}). Although we will not need the greater generality here, it should be straightforward to extend this result to locally-compact real trees (at least when $X^\mathcal{T}$ is recurrent).

\begin{lem}\label{rkhit} Suppose Assumption \ref{assu:1} holds, and that $x,y\in\mathcal{T}$, $x\neq y$. Under $P_x^\mathcal{T}$, the law of the local time process $L^\mathcal{T}$ at time $\tau_y$ is characterised as follows.
\begin{enumerate}
  \item[(a)] The law of $(L^\mathcal{T}_{\tau_y}(z))_{z\in [[x,y]]_\mathcal{T}}$ is given by a $\mathrm{BESQ}^2(0)$ distribution, with time parameterised by distance from $y$.
  \item[(b)] Conditional on $(L^\mathcal{T}_{\tau_y}(z))_{z\in [[x,y]]_\mathcal{T}}$, the local times on the components of $\mathcal{T}\backslash[[x,y]]_\mathcal{T}$, that is, $(L^\mathcal{T}_{\tau_y}(z))_{z\in \mathcal{T}_i}$, $i\in \mathcal{I}_{x,y}$, are independent, with each distributed as a $\mathcal{T}_i$-indexed 0-dimensional BESQ process started from $L^\mathcal{T}_{\tau_y}(\rho_i)$.
\end{enumerate}
\end{lem}
\begin{proof}
In this proof, we consider the Gaussian process $(G_z)_{z\in \mathcal{T}}$ whose distribution is characterised by
\[\mathbf{E}\left(G_z\right)=0,\]
\[\mathrm{Cov}\left(G_w,G_z\right)=d_\mathcal{T}\left(y,b^\mathcal{T}(y,w,z)\right).\]
As per the discussion in \cite[Section 6]{DLeG}, it is possible to check that, under Assumption \ref{assu:1}, this process has a continuous version. (Alternatively, one could follow the proof of Lemma \ref{besqlem}, using a Gaussian estimate in place of \eqref{bes0bound}.) We will show that
\begin{equation}\label{e0}
\left(L^\mathcal{T}_{\tau_y}(z)+\left(G_z-G_{b^\mathcal{T}(x,y,z)}\right)^2+\left(\bar{G}_z-\bar{G}_{b^\mathcal{T}(x,y,z)}\right)^2\right)_{z\in\mathcal{T}}\overset{(d)}=
\left(G_z^2+\bar{G}_z^2\right)_{z\in \mathcal{T}},
\end{equation}
where $(\bar{G}_z)_{z\in \mathcal{T}}$ is an independent copy of $(G_z)_{z\in \mathcal{T}}$, and both $G$ and $\bar{G}$ are independent of $L^\mathcal{T}_{\tau_y}$. From this, the result follows exactly as in the one-dimensional case; see the proof of \cite[Theorem 2.6.3]{MR}.

\begin{figure}[t]
\begin{center}
\includegraphics[width=8.5cm]{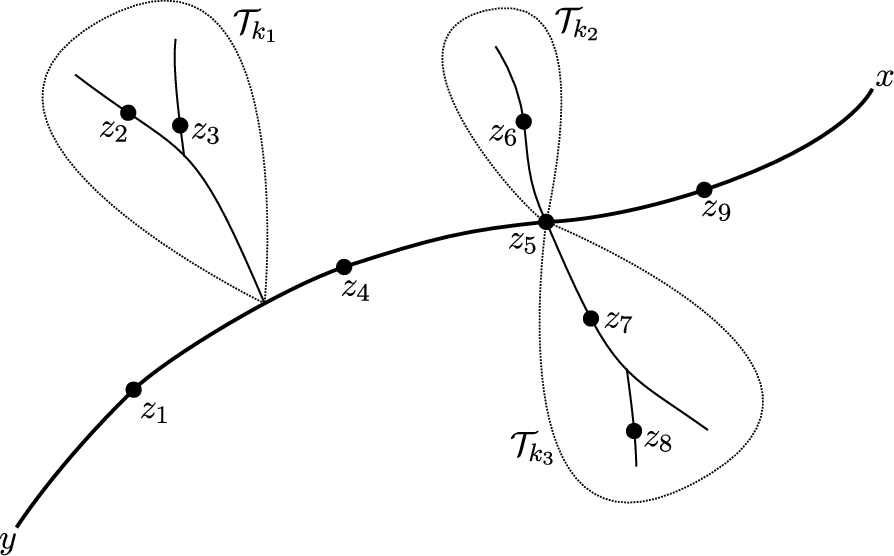}
\end{center}
\caption{An example of the labelling of the points $(z_i)_{i=1}^M$ and subtrees $(\mathcal{T}_{k_n})_{n=1}^N$ from the proof of Lemma \ref{rkhit}. The arc $[[x,y]]_\mathcal{T}$ is shown in bold.}\label{f2}
\end{figure}

Since all the processes in question are continuous, it will suffice to check that the finite-dimensional distributions of the various objects align. In particular, for the remainder of the proof, we fix vertices $(z_i)_{i=1}^M$ in $\mathcal{T}$. Without loss of generality, we will suppose these are distinct. Moreover, it will be convenient to suppose these are ordered so that $d_\mathcal{T}(y,b^\mathcal{T}(x,y,z_i))$ is non-decreasing, and if $d_\mathcal{T}(y,b^\mathcal{T}(x,y,z_i))=d_\mathcal{T}(y,b^\mathcal{T}(x,y,z_j))$ with $z_i\in[[x,y]]_\mathcal{T}$, then $i\leq j$ (i.e.,\ vertices on the arc $[[x,y]]_\mathcal{T}$ appear in the order before those in the subtrees that emanate from them). Moreover, vertices within the same subtree off $[[x,y]]_\mathcal{T}$ should be consecutive in the sequence, though the order within these is unimportant. See Figure \ref{f2} for an example of this labelling. To understand the distribution of $(L^\mathcal{T}_{\tau_y}(z_i))_{i=1}^M$, we consider its moment generating function. In particular, from \cite[Lemma 3.10.2]{MR}, we have that, for $\lambda_1,\dots,\lambda_M$ sufficiently small,
\begin{equation}\label{lt}
E_x^\mathcal{T}\left(\exp\left(\sum_{i=1}^M \lambda_i L^\mathcal{T}_{\tau_y}(z_i)\right)\right)=\frac{\det\left(I-\hat{\Sigma}\Lambda\right)}{\det\left(I-{\Sigma}\Lambda\right)},
\end{equation}
where $I=(I_{i,j})_{i,j=1}^M$ is the $M\times M$ identity matrix, $\Lambda$ is the $M\times M$ diagonal matrix with entries given by $\Lambda_{i,j}=\lambda_iI_{i,j}$, $\Sigma=(\Sigma_{i,j})_{i,j=1}^M$ is given by
\[\Sigma_{i,j}=2d_\mathcal{T}\left(y,b^\mathcal{T}(y,z_i,z_j)\right),\]
and $\hat{\Sigma}=(\hat{\Sigma}_{i,j})_{i,j=1}^M$ is given by
\[\hat{\Sigma}_{i,j}=2d_\mathcal{T}\left(y,b^\mathcal{T}(y,z_i,z_j)\right)-2d_\mathcal{T}\left(y,b^\mathcal{T}(x,y,z_j)\right).\]

We highlight that, in \cite[Lemma 3.10.2]{MR}, $\Sigma$ and $\hat{\Sigma}$ are expressed in terms of a certain potential density, namely $\Sigma_{i,j}=u_{\tau_y}(z_i,z_j)$ and $\hat{\Sigma}_{i,j}=u_{\tau_y}(z_i,z_j)-u_{\tau_y}(x,z_j)$, where $u_{\tau_y}(z,\cdot)$ is the occupation density for $X^\mathcal{T}$, started from $z$ and run until hitting $y$. In our case, it is known that $u_{\tau_y}(z,w)=2d_\mathcal{T}(y,b^\mathcal{T}(y,z,w))$, see \cite[Corollary 3.10]{AEW} or \cite[Proposition 2.12]{ALW}, for example. Moreover, \cite[Lemma 3.10.2]{MR} requires a continuous $\alpha$-potential density, but this readily follows from the existence of a jointly continuous transition density for $X^\mathcal{T}$, $(p^\mathcal{T}_t(w,z))_{w,z\in\mathcal{T},t>0}$ say (see \cite[Theorem 10.4]{Kigq}, for example), which, under Assumption \ref{assu:1}, satisfies $p^\mathcal{T}_t(z,z)\leq Ct^{-\eta/(1+\eta)}$ for $t\in (0,1)$ (as follows from the argument of \cite[Proposition 4.1]{Kum}, for example; applying such ideas, for the Brownian CRT in particular, detailed transition density estimates were presented in \cite{Croydoncrt}).

Thus it remains to show that the right-hand side of \eqref{lt} can be expressed in terms of the Laplace transforms of the relevant Gaussian processes. Firstly, we observe that $\mathrm{Cov}(G_{z_i},G_{z_j})=\frac12\Sigma_{i,j}$ by definition. Hence, writing $g^T$ for the transpose of a vector, we can directly compute that
\begin{eqnarray*}
\mathbf{E}\left(\exp\left(\sum_{i=1}^M \lambda_i G_{z_i}^2\right)\right)&=&\frac{1}{(2\pi)^{M/2}\det\left(\frac12\Sigma\right)^{1/2}}\int_{\mathbb{R}^M}\exp\left( g^T\Lambda g -\frac{1}{2}g^T\left(\frac12\Sigma\right)^{-1}g\right)dg\\
&=&\frac{1}{(2\pi)^{M/2}\det\left(\frac12\Sigma\right)^{1/2}}\int_{\mathbb{R}^M}\exp\left(-\frac{1}{2}g^T\left(2\Sigma^{-1}-2\Lambda\right)g\right)dg\\
&=&\frac{\det\left(\left(2\Sigma^{-1}-2\Lambda\right)^{-1}\right)^{1/2}}{\det\left(\frac12\Sigma\right)^{1/2}}\\
&=&\det\left(1-\Sigma \Lambda\right)^{-1/2},
\end{eqnarray*}
which of course implies
\begin{equation}\label{gausscalc}
\mathbf{E}\left(\exp\left(\sum_{i=1}^M \lambda_i \left(G_{z_i}^2+\bar{G}_{z_i}^2\right)\right)\right)=\det\left(1-\Sigma \Lambda\right)^{-1}.
\end{equation}
To handle the numerator of \eqref{lt}, we will consider the form of $\hat{\Sigma}$ carefully. To do this, we first note that every $z_i$ is either on the arc $[[x,y]]_\mathcal{T}$ or in a set of the form $\mathcal{T}_k\backslash\{\rho_k\}$, where $\mathcal{T}_k$ is one of the trees in the collection $(\mathcal{T}_l)_{l\in\mathcal{I}_{x,y}}$. Write $(\mathcal{T}_{k_n})_{n=1}^N$ for the set of trees that contain such a vertex, ordered in a way that is consistent with the ordering of the points from $(z_i)_{i=1}^M$ that they contain. Again, see Figure \ref{f2} for an example. Now, if $d_\mathcal{T}(y,b^\mathcal{T}(x,y,z_i))>d_\mathcal{T}(y,b^\mathcal{T}(x,y,z_j))$, then
\begin{eqnarray*}
\hat{\Sigma}_{i,j}&=&2d_\mathcal{T}\left(y,b^\mathcal{T}(y,z_i,z_j)\right)-2d_\mathcal{T}\left(y,b^\mathcal{T}(x,y,z_j)\right)\\
&=&2d_\mathcal{T}\left(y,b^\mathcal{T}(x,y,z_j)\right)-2d_\mathcal{T}\left(y,b^\mathcal{T}(x,y,z_j)\right)\\
&=&0.
\end{eqnarray*}
Similarly, if $d_\mathcal{T}(y,b^\mathcal{T}(x,y,z_i))=d_\mathcal{T}(y,b^\mathcal{T}(x,y,z_j))$ and $z_i$ lies on the arc $[[x,y]]_{\mathcal{T}}$, or if $z_i$ and $z_j$ lie in different subtrees emanating from the same point on the arc $[[x,y]]_{\mathcal{T}}$, then $\hat{\Sigma}_{i,j}=0$. Furthermore, if $z_i$ and $z_j$ lie in the same tree $\mathcal{T}_k\backslash\{\rho_k\}$, then
\begin{eqnarray*}
\hat{\Sigma}_{i,j}&=&2d_\mathcal{T}\left(y,b^\mathcal{T}(y,z_i,z_j)\right)-2d_\mathcal{T}\left(y,b^\mathcal{T}(x,y,z_j)\right)\\
&=&2d_\mathcal{T}\left(y,b^\mathcal{T}(y,z_i,z_j)\right)-2d_\mathcal{T}\left(y,b^\mathcal{T}(x,y,z_i)\right)\\
&=&\hat{\Sigma}_{j,i}.
\end{eqnarray*}
Consequently, writing $\hat{\Sigma}^{(k_n)}$ for the symmetric matrix $(\hat{\Sigma}_{i,j})_{z_i,z_j\in \mathcal{T}_{k_n}}$, we have that $\hat{\Sigma}$ is block upper triangular of the form
\[\begin{pmatrix}
\begin{array}{|@{\,}c@{\,}|}
     \hline \hphantom{\hspace{1.85pt}}0\hphantom{\hspace{1.85pt}} \\\hline
   \end{array} &        &           &   &        &  \\
          & \hskip-2.5pt \begin{array}{|@{\,}c@{\,}|}
     \hline  \hphantom{\hspace{1.85pt}}0\hphantom{\hspace{1.85pt}} \\ \hline
   \end{array} &           &   &        &   \\
          &        & \hskip-2.5pt \begin{array}{|@{\,}ccc@{\,}|}
     \hline &&\\
&\hphantom{\hspace{-2.8pt}}\hat{\Sigma}^{(k_1)}\hphantom{\hspace{-2.8pt}}&\\
&& \\\hline
   \end{array} &   &        &   \\
          &        &           &  \hskip-2.5pt\begin{array}{|@{\,}c@{\,}|}
     \hline  \hphantom{\hspace{1.85pt}}0\hphantom{\hspace{1.85pt}}  \\\hline
   \end{array}  &        &   \\
          &        &           &   &  \hskip-2.5pt\begin{array}{|@{\,}ccc@{\,}|}
     \hline &&\\
&\hphantom{\hspace{-2.8pt}}\hat{\Sigma}^{(k_2)}\hphantom{\hspace{-2.8pt}}&\\
&& \\\hline
   \end{array} &   \\
0         &        &           &   &        & \ddots \\
\end{pmatrix}.\]
In particular, the matrix has 0 entries below the diagonal block, the 0 entries on the diagonal correspond to those $z_i$ on the arc $[[x,y]]_\mathcal{T}$ and the matrices $\hat{\Sigma}^{(k_n)}$ to those vertices in subtrees emanating from that. From this observation, we are readily able to deduce that
\begin{equation}\label{e3}
\det\left(I-\hat{\Sigma}\Lambda\right)=\prod_{n=1}^N\det\left(I^{(k_n)}-\hat{\Sigma}^{(k_n)}\Lambda^{(k_n)}\right)=\det\left(I-\tilde{\Sigma}\Lambda\right),
\end{equation}
where $I^{(k_n)}$ and $\Lambda^{(k_n)}$ are the restrictions to vertices in $\mathcal{T}_{k_n}$ of $I$ and $\Lambda$, respectively, and $\tilde{\Sigma}$ is the symmetric matrix given by replacing in $\hat{\Sigma}$ all the entries above the block diagonal by 0. Finally, we observe that
\[\mathbf{E}\left(G_{z_i}-G_{b^\mathcal{T}(x,y,{z_i})}\right)=0,\]
\begin{eqnarray*}
\lefteqn{\mathrm{Cov}\left(G_{z_i}-G_{b^\mathcal{T}(x,y,{z_i})},G_{z_j}-G_{b^\mathcal{T}(x,y,{z_j})}\right)}\\
&=&d_\mathcal{T}\left(y,b^\mathcal{T}(y,{z_i},{z_j})\right)+d_\mathcal{T}\left(y,b^\mathcal{T}(y,b^\mathcal{T}(x,y,{z_i}),b^\mathcal{T}(x,y,{z_j}))\right)\\
&&-d_\mathcal{T}\left(y,b^\mathcal{T}(y,b^\mathcal{T}(x,y,{z_i}),{z_j})\right)-d_\mathcal{T}\left(y,b^\mathcal{T}(y,{z_i},b^\mathcal{T}(x,y,{z_j})\right)\\
&=&\left\{
     \begin{array}{ll}
       d_\mathcal{T}(y,b^\mathcal{T}(y,z_i,z_j))-d_\mathcal{T}(y,b^\mathcal{T}(x,y,z_j)), & \begin{array}{l}\hbox{if $z_i,z_j\in \mathcal{T}_{k_n}\backslash\{\rho_{k_n}\}$}\\\hbox{for some $n=1,\dots,N$,}\end{array}\\
       0, & \hbox{otherwise,}
     \end{array}
   \right.\\
&=&\frac{1}{2}\tilde{\Sigma}_{i,j}.
\end{eqnarray*}
Thus, by following the calculation used to deduce \eqref{gausscalc}, we find that
\begin{equation}\label{e4}
\mathbf{E}\left(\exp\left(\sum_{i=1}^M \lambda_i \left(\left(G_{z_i}-G_{b^\mathcal{T}(x,y,{z_i})}\right)^2+\left(\bar{G}_{z_i}-\bar{G}_{b^\mathcal{T}(x,y,{z_i})}\right)^2\right)\right)\right)=\det\left(1-\tilde{\Sigma} \Lambda\right)^{-1}.
\end{equation}
Combining \eqref{lt}, \eqref{gausscalc}, \eqref{e3} and \eqref{e4} yields \eqref{e0}, as desired.
\end{proof}

\begin{lem}\label{add}
(a) Suppose $(\xi^i_x)_{x\in\mathcal{T}}$, $i=1,2$, are independent $\mathcal{T}$-indexed $\mathrm{BESQ}^{d_i}(z_i)$ processes. It is then the case that $(\xi^1_x+\xi^2_x)_{x\in\mathcal{T}}$ is a $\mathcal{T}$-indexed $\mathrm{BESQ}^{d_1+d_2}(z_1+z_2)$ process.\\
(b) Suppose $(\xi_x)_{x\in\mathcal{T}}$ is a $\mathcal{T}$-indexed $\mathrm{BESQ}^{d}(z_0)$ process and $\mathcal{T}'\subseteq \mathcal{T}$ is a subtree spanning a finite set of vertices including the root. Write $(\mathcal{T}_i)_{i\in\mathcal{I}'}$ for the closures of the components of $\mathcal{T}\backslash\mathcal{T}'$. It is then the case that, conditional on $(\xi_x)_{x\in\mathcal{T}'}$, $(\xi_x)_{x\in\mathcal{T}_i}$, $i\in \mathcal{I}'$, are independent, with each distributed as a $\mathcal{T}_i$-indexed $\mathrm{BESQ}^{d}(\xi_{\rho_i})$ process.
\end{lem}

\begin{proof}
Given the continuity of the processes in question, for both parts (a) and (b), it suffices to check the claims for the finite-dimensional distributions of the processes in question. This is straightforward from the construction of the tree-indexed BESQ processes and the analogous claims that hold for the one-dimensional process. In particular, for the additivity property, see \cite[Theorem XI.1.2]{RevuzYor}. The Markovianity of the processes is also discussed in \cite[Chapter XI]{RevuzYor}.
\end{proof}

\begin{proof}[Proof of Proposition \ref{rk}(b)]
In this proof, we write
\[\tau_{(i)}:=\max_{j=1,\dots,i}\tau_{x_j},\]
so that $\tau_{\mathrm{cov}}^{(\varepsilon)}=\tau_{(M(\varepsilon))}$.

By Lemma \ref{rkhit}, we have that, conditional on $(L^\mathcal{T}_{\tau_{(1)}}(z))_{z\in [[\rho,x_1]]_\mathcal{T}}$, $(L^\mathcal{T}_{\tau_{(1)}}(z))_{z\in \mathcal{T}_i}$, $i\in \mathcal{I}_{\rho,x_1}$, are independent, with each distributed as a $\mathcal{T}_i$-indexed 0-dimensional BESQ process started from $L^\mathcal{T}_{\tau_{(1)}}(\rho_i)$. Using the Markov property for tree-indexed Bessel processes of Lemma \ref{add}(b), we obtain from this that, conditional on $(L^\mathcal{T}_{\tau_{(1)}}(z))_{z\in \mathcal{T}^{(\varepsilon)}}$, $(L^\mathcal{T}_{\tau_{(1)}}(z))_{z\in \mathcal{T}_i}$, $i\in \mathcal{I}^{(\varepsilon)}$, are independent, with each distributed as a $\mathcal{T}_i$-indexed 0-dimensional BESQ process started from $L^\mathcal{T}_{\tau_{(1)}}(\rho_i)$.

We proceed from the previous observation inductively. Suppose for some $i<M(\varepsilon)$ we have that, conditional on $(L^\mathcal{T}_{\tau_{(i)}}(z))_{z\in \mathcal{T}^{(\varepsilon)}}$, $(L^\mathcal{T}_{\tau_{(i)}}(z))_{z\in \mathcal{T}_j}$, $j\in \mathcal{I}^{(\varepsilon)}$, are independent, with each distributed as a $\mathcal{T}_j$-indexed 0-dimensional BESQ process started from $L^\mathcal{T}_{\tau_{(i)}}(\rho_j)$. We have two situations to consider: either $\tau_{x_{i+1}}<\tau_{(i)}$ or $\tau_{x_{i+1}}>\tau_{(i)}$. (The times $\tau_{x_{i+1}}$ and $\tau_{(i)}$ can not be equal, because we assumed the vertices $(x_j)_{j\geq 0}$ to be distinct.) Moreover, we observe that, since it $P_\rho^\mathcal{T}$-a.s.\ holds that $\tau_{x_{i+1}}=\inf\{t\geq 0:\:L^\mathcal{T}_t(x_{i+1})>0\}$ (cf.\ \eqref{ltinc}), we have that $L^\mathcal{T}_{\tau_{(i)}}(x_{i+1})>0$ implies $\tau_{x_{i+1}}<\tau_{(i)}$ and $L^\mathcal{T}_{\tau_{(i)}}(x_{i+1})=0$ implies $\tau_{x_{i+1}}>\tau_{(i)}$. In particular, which of the events $\{\tau_{x_{i+1}}<\tau_{(i)}\}$ or $\{\tau_{x_{i+1}}>\tau_{(i)}\}$ occurs is determined by $(L^\mathcal{T}_{\tau_{(i)}}(z))_{z\in \mathcal{T}^{(\varepsilon)}}$. Similarly, $X^{\mathcal{T}}_{\tau_{(i)}}$ is determined by $(L^\mathcal{T}_{\tau_{(i)}}(z))_{z\in \mathcal{T}^{(\varepsilon)}}$, being the unique one of the vertices $x_1,\dots,x_i$ with zero local time at time $\tau_{(i)}$.

Now, conditional on $(L^\mathcal{T}_{\tau_{(i)}}(z))_{z\in \mathcal{T}^{(\varepsilon)}}$, if $\tau_{x_{i+1}}>\tau_{(i)}$, then
\[\left(L^\mathcal{T}_{\tau_{(i+1)}}(z)-L^\mathcal{T}_{\tau_{(i)}}(z)\right)_{z\in\mathcal{T}},\]
is distributed as the local time $(L_{\tau_{x_{i+1}}}^\mathcal{T}(z))_{z\in\mathcal{T}}$ under $P^\mathcal{T}_{X_{\tau_{(i)}}^\mathcal{T}}$, independently of $(L^\mathcal{T}_{\tau_{(i)}}(z))_{z\in \mathcal{T}_j}$, $j\in \mathcal{I}^{(\varepsilon)}$. In particular, applying Lemma \ref{rkhit} and the Markov property for tree-indexed Bessel processes (i.e.\ Lemma \ref{add}(b)), conditional on $(L^\mathcal{T}_{\tau_{(i)}}(z))_{z\in \mathcal{T}^{(\varepsilon)}}$ and  $(L^\mathcal{T}_{\tau_{(i+1)}}(z)-L^\mathcal{T}_{\tau_{(i)}}(z))_{z\in \mathcal{T}^{(\varepsilon)}}$, if $\tau_{x_{i+1}}>\tau_{(i)}$, then
\[\left(L^\mathcal{T}_{\tau_{(i+1)}}(z)-L^\mathcal{T}_{\tau_{(i)}}(z)\right)_{z\in\mathcal{T}_j},\qquad j\in\mathcal{I}^{(\varepsilon)},\]
are independent, with each distributed as a $\mathcal{T}_j$-indexed 0-dimensional BESQ process started from $L^\mathcal{T}_{\tau_{(i+1)}}(\rho_j)-L^\mathcal{T}_{\tau_{(i)}}(\rho_j)$ (again, independently of $(L^\mathcal{T}_{\tau_{(i)}}(z))_{z\in \mathcal{T}_j}$, $j\in \mathcal{I}^{(\varepsilon)}$). On the other hand, conditional on $(L^\mathcal{T}_{\tau_{(i)}}(z))_{z\in \mathcal{T}^{(\varepsilon)}}$, if $\tau_{x_{i+1}}<\tau_{(i)}$, then $\tau_{(i+1)}=\tau_{(i)}$, and so
\[L^\mathcal{T}_{\tau_{(i+1)}}(z)-L^\mathcal{T}_{\tau_{(i)}}(z)=0,\qquad\forall z \in\mathcal{T}.\]
Since the 0-dimensional BESQ process started from 0 remains at 0, then we can express this in an analogous way to the conclusion in the first case. Namely, conditional on $(L^\mathcal{T}_{\tau_{(i)}}(z))_{z\in \mathcal{T}^{(\varepsilon)}}$ and  $(L^\mathcal{T}_{\tau_{(i+1)}}(z)-L^\mathcal{T}_{\tau_{(i)}}(z))_{z\in \mathcal{T}^{(\varepsilon)}}$,
\[\left(L^\mathcal{T}_{\tau_{(i+1)}}(z)-L^\mathcal{T}_{\tau_{(i)}}(z)\right)_{z\in\mathcal{T}_j},\qquad j\in\mathcal{I}^{(\varepsilon)},\]
are independent, with each distributed as a $\mathcal{T}_j$-indexed 0-dimensional BESQ process started from $L^\mathcal{T}_{\tau_{(i+1)}}(\rho_j)-L^\mathcal{T}_{\tau_{(i)}}(\rho_j)$ (again, independently of $(L^\mathcal{T}_{\tau_{(i)}}(z))_{z\in \mathcal{T}_j}$, $j\in \mathcal{I}^{(\varepsilon)}$). Consequently, applying the inductive hypothesis and the additivity property of Lemma \ref{add}(a), we conclude that, conditional on $(L^\mathcal{T}_{\tau_{(i)}}(z))_{z\in \mathcal{T}^{(\varepsilon)}}$ and  $(L^\mathcal{T}_{\tau_{(i+1)}}(z)-L^\mathcal{T}_{\tau_{(i)}}(z))_{z\in \mathcal{T}^{(\varepsilon)}}$,
\[\left(L^\mathcal{T}_{\tau_{(i+1)}}(z)\right)_{z\in\mathcal{T}_j},\qquad j\in\mathcal{I}^{(\varepsilon)},\]
are independent, with each distributed as a $\mathcal{T}_j$-indexed 0-dimensional BESQ process started from $L^\mathcal{T}_{\tau_{(i+1)}}(\rho_j)$. Since the law only depends on the pointwise sum of $(L^\mathcal{T}_{\tau_{(i)}}(z))_{z\in \mathcal{T}^{(\varepsilon)}}$ and  $(L^\mathcal{T}_{\tau_{(i+1)}}(z)-L^\mathcal{T}_{\tau_{(i)}}(z))_{z\in \mathcal{T}^{(\varepsilon)}}$, it follows that we have the same result when we condition simply upon $(L^\mathcal{T}_{\tau_{(i+1)}}(z))_{z\in \mathcal{T}^{(\varepsilon)}}$.

The above argument gives the result corresponding to Proposition \ref{rk}(b) at the time $\tau_{\mathrm{cov}}^{(\varepsilon)}$. To get from here to the time $\tau_{\mathrm{cov}}^{(\varepsilon)}(\delta)$, we need to add the local time accumulated between these two times, but this can be dealt with in an identical fashion to the inductive step of the proof, and so we omit the details. (In fact, the argument is easier, because we certainly have that $\tau_{\mathrm{cov}}^{(\varepsilon)}(\delta)>\tau_{\mathrm{cov}}^{(\varepsilon)}$, and so we only have one case to deal with.)
\end{proof}

Before concluding the section, we give one other simple consequence of the Ray-Knight theorem for hitting times.

\begin{cor}\label{poscor}
Suppose Assumption \ref{assu:1} holds, $\varepsilon>0$, and $2\delta\in (0,d_\mathcal{T}(\rho,x_1))$. It then $P_\rho^\mathcal{T}$-a.s.\ holds that
\[\inf_{z\in\mathcal{T}^{(\varepsilon)}}L^\mathcal{T}_{\tau_{\mathrm{cov}}^{(\varepsilon)}(\delta)}(z)>0.\]
\end{cor}
\begin{proof}
Since a $\mathrm{BESQ}^2(0)$ process is simply the square of the modulus of a two-dimensional Brownian motion started from 0 (see \cite[Section 14.2]{MR}, for example), which is a process that almost-surely never returns to 0, it follows from Lemma \ref{rkhit}(a) that, for each $i$, $P_\rho^\mathcal{T}$-a.s.,\
\[L^\mathcal{T}_{\tau_{x_i}}(z)>0,\qquad \forall z\in [[\rho,x_i]]_\mathcal{T}\backslash\{x_i\}.\]
As we also know that $\tau_{x_i}=\inf\{t\geq 0:\: L^\mathcal{T}_t(x_i)>0\}$ (cf.\ \eqref{ltinc2}), it follows that
\[\tau_{x_i}=\inf\left\{t\geq 0:\: L^\mathcal{T}_t(z)>0,\:\forall z\in [[\rho,x_i]]_\mathcal{T}\right\}.\]
Putting the results together for each of the arcs forming $\mathcal{T}^{(\varepsilon)}$, we consequently obtain, $P_\rho^\mathcal{T}$-a.s.,\
\[\tau_{\mathrm{cov}}^{(\varepsilon)}=\inf\left\{t\geq 0:\: L^\mathcal{T}_t(z)>0,\:\forall z\in \mathcal{T}^{(\varepsilon)}\right\}.\]
Since $\tau_{\mathrm{cov}}^{(\varepsilon)}(\delta)>\tau_{\mathrm{cov}}^{(\varepsilon)}$ and the local times are continuous and non-decreasing, the result follows.
\end{proof}

\section{Height decomposition for the Brownian CRT}\label{sec:rec}

The aim of this section is to give an approximation procedure for the Brownian CRT by recursively adding branches of maximal height. This procedure relies on William's decomposition of Brownian excursions under the It\^o measure \cite{Williams}. The following description of this decomposition is based on \cite[Chapter XII.4]{RevuzYor}. We note that there are other recursive decompositions of the Brownian CRT, such as the stick-breaking construction of \cite[Process 3]{Aldous1} or the self-similarity seen when splitting at the branch point of the root and two uniformly chosen vertices, as described in \cite[Theorem 2]{Aldousrec} (see also \cite{AG,CHrec}), but we adopt the current approach due to the convenience of the description of the remainder of the tree in terms of a Poisson process of smaller copies of the Brownian CRT. See Proposition \ref{prop:poissonTepsilon} below for details.

We denote by $\mathbf n$ the It\^o measure on excursions $\zeta=(\zeta_s)_{s\geq 0}$. For an excursion $\zeta$, we define its (finite by definition) lifetime by setting
\[\sigma (\zeta):=\inf \left\{ s \geq 0 : 
\zeta_t = 0,\: \forall t >s  \right\},\]
and write $\sigma$ when the context is clear. We normalise the It\^o measure such that
\begin{equation} \label{eq:Ito}
\mathbf n \left( \sup_{s \geq 0} \zeta_s > \varepsilon \right) = \frac{1}{2 \varepsilon} \quad \text{and} \quad \mathbf n \left( \sigma > \varepsilon \right) = \frac{1}{\sqrt{ 2 \pi \varepsilon}}.
\end{equation}
(See \cite[Chapter XII]{RevuzYor} for background on the excursion measure of Brownian motion.) Each excursion in the support of $\mathbf n$ corresponds to a compact real tree via contour functions. In particular, to define $\mathcal{T}=\mathcal{T}_\zeta$, we first introduce
\[d_\zeta(s,t):=\zeta(s)+\zeta(t)-2\min_{r\in[s\wedge t,s\vee t]}\zeta(r),\qquad s,t\geq 0,\]
and then set $\mathcal{T}:=[0,\sigma]/\sim$, where we suppose $s\sim t$ if and only if $d_\zeta(s,t)=0$. (Here and below, we use the notation $s\wedge t:=\min\{s,t\}$ and $s\vee t:=\max\{s,t\}$.) We equip $\mathcal{T}$ with the quotient metric induced by $2d_\zeta$ (the scaling factor of 2 is included to align our definition with that of Aldous, see \cite[Corollary 22]{Aldous3}), the measure $\mu_\mathcal{T}$ obtained as the push-forward of Lebesgue measure on $[0,\sigma]$ by the quotient map, and a root $\rho$, as given by the equivalence class of $0$ with respect to $\sim$. (See \cite[Section 2]{rrt} for an introduction to this coding procedure.) It is known that the push-forward map from $(\zeta,\sigma)$ to $(\mathcal{T},d_\mathcal{T},\mu_\mathcal{T},\rho)$ is continuous when the image space is equipped with the marked Gromov-Hausdorff-Prokhorov topology (cf.\ \cite[Proposition 3.3]{ADH}), and thus the map $\zeta\mapsto(\mathcal{T},d_\mathcal{T},\mu_\mathcal{T},\rho)$ is measurable. Hence, $(\mathcal{T},d_\mathcal{T},\mu_\mathcal{T},\rho)$ is a well-defined random object, and we denote its law by $\mathbf N$. With a slight abuse of notation, the law $\mathbf P$ of the Brownian CRT is given by
\[\mathbf P = \mathbf N \left( \cdot \middle| \sigma = 1 \right).\]
We remark that, from the scaling property of Brownian motion, it is possible to check that
\begin{equation}\label{scaling}
    \mathbf N \left(\left(\mathcal{T},d_\mathcal{T},\mu_\mathcal{T},\rho\right)\in  \cdot \middle| \sigma = s \right)=\mathbf{P}\left(\left(\mathcal{T},s^{1/2}d_\mathcal{T},s\mu_\mathcal{T},\rho\right)\in  \cdot\right).
\end{equation}
Since, by \eqref{eq:Ito}, it holds that $\mathbf{N}=\int_0^\infty \frac{1}{2\sqrt{2\pi s^3}}\mathbf{N}(\cdot\:|\:\sigma=s)ds $, if a property of $(\mathcal{T},d_\mathcal{T},\mu_\mathcal{T},\rho)$ that does not depend on the particular scaling of the tree holds $\mathbf{N}$-a.s.,\ then it also holds $\mathbf{P}$-a.s.\ (and vice versa).

The Brownian CRT $\mathcal T$ has a unique vertex $x_\emptyset$ of maximal height (i.e.\ distance from the root), $\mathbf P$-.a.s. Indeed, this vertex corresponds to the maximum of a Brownian excursion, which is unique by William's decomposition, for example. (See the proof of Lemma \ref{lem:supHTi} below for further details.) In the same spirit as in previous sections, we denote by $(\mathcal T_i)_{i \geq 1}$ the subtrees of $\mathcal T$ corresponding to the closures of the connected components of $\mathcal T \setminus [[\rho, x_\emptyset]]_{\mathcal{T}}$ ordered by decreasing heights. (Our subsequent decomposition will make clear that, almost-surely, these heights are distinct.) For every $i$, we denote by $\rho_i$ the root of $\mathcal T_i$ (i.e.\ the point of its intersection with $[[\rho, x_\emptyset]]_{\mathcal{T}}$), and by $x_i$ its $\mathbf P$-a.s.\ unique vertex of maximal height (again, our decomposition will make clear such points are almost-surely well-defined). We will also denote the height of $\mathcal T_i$, which is the length of the segment $[[\rho_i,x_i]]_{\mathcal{T}}$, by $H(\mathcal T_i)$. We highlight that, unlike the assumption of the previous section, the points $(x_i)_{i\geq 1}$ will not be dense in $(\mathcal{T},d_\mathcal{T})$.

We can extend this procedure by induction; the following makes sense $\mathbf{P}$-a.s. As an index set, we introduce $\Sigma_* = \bigcup_{n \geq 0} \Sigma_n$, where $\Sigma_0:=\{\emptyset\}$ and, for every $n \geq 1$, $\Sigma_n := \mathbb N^n$. Suppose that subtrees $\mathcal T_i$ of $\mathcal T$ with respective roots $\rho_i$ and vertices of maximal height $x_i$ have been constructed for $i \in \Sigma_*$ such that $|i| \leq n$. Fix $i \in \Sigma_n$. We define the trees $(\mathcal T_{ij})_{j\geq1}$ as the subtrees of $\mathcal T_i$ corresponding to the closures of the connected components of $\mathcal T_i \setminus [[\rho_i , x_i]]_{\mathcal{T}}$ ordered by decreasing height. Each $\mathcal T_{ij}$ has root $\rho_{ij} \in [[\rho_i , x_i]]_{\mathcal{T}}$ and we denote by $x_{ij}$ its $\mathbf P$-a.s.\ unique vertex of maximal height. As before, we also denote the height of $\mathcal T_{ij}$ by $H(\mathcal T_{ij})$. See Figure~\ref{fig:CRTdec} for an illustration. The next lemma will ensure that, for any $\varepsilon>0$, there will only be a finite number of arcs $[[\rho_i,x_i]]_\mathcal{T}$, $i\in\Sigma_*$, with a length that exceeds $\varepsilon$.

\begin{figure}[t]
\begin{center}
\includegraphics[width=6cm]{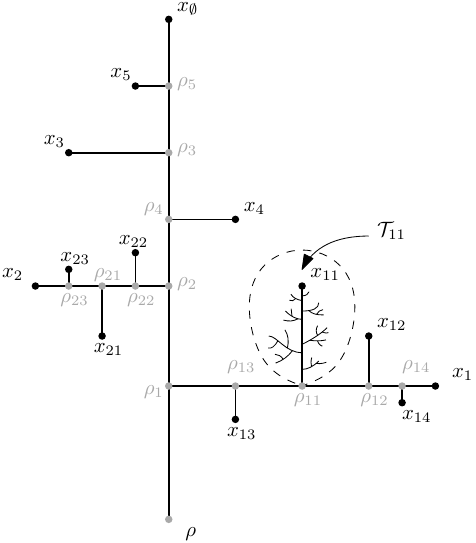}
\end{center}
\caption{Height decomposition of the Brownian CRT. A subset of leaves $x_i$ with $|i| \leq 2$ is represented in black; the corresponding roots $\rho_{i}$ are represented in gray.}\label{fig:CRTdec}
\end{figure}

\begin{lemma} \label{lem:supHTi}
$\mathbf P$-a.s.,\ as $n \to \infty$, it holds that 
\[\sup_{i \in \Sigma_n} H(\mathcal T_i) \underset{n \to \infty}{\rightarrow} 0.\]
\end{lemma}

\begin{proof}
We start by presenting William's decomposition of the Brownian excursion \cite[Chapter XII.4]{RevuzYor}. In the context of random trees, this description can also be found in \cite[Theorem 3.3]{AbDel}. Denoting $H(\mathcal{T})$ the height of a rooted real tree $\mathcal{T}$, our normalisation of the Brownian excursion gives
\[\mathbf N \left( H > \varepsilon \right) = \mathbf n \left( \sup_{s \geq 0} \zeta_s > \varepsilon/2 \right) = \frac{1}{\varepsilon},\]
and, as follows from \cite[Theorem 3.3]{AbDel}, $\mathbf N$-a.e.\ realisation of $\mathcal{T}$ has a unique vertex of height $H$. For $h >0$, we set $\mathbf N^{H=h}( \cdot )  := \mathbf N( \cdot \:|\:H=h )$, so that
\[\mathbf N = \int_0^\infty \frac{{d} h}{h^2} \, \mathbf N^{H=h} .\]
Fix $h >0$. Under $\mathbf N^{H=h}$, a tree has a unique vertex $x_h$ at height $h$ almost-surely, and we denote the (possibly empty) subtrees grafted on the left-hand side and right-hand side of the segment $[[\rho, x_h]]_{\mathcal{T}}$ at distance $x$ from $x_h$ and $h-x$ from $\rho$ by $\overleftarrow {\mathcal{T}}_x^h$ and $\overrightarrow{\mathcal{T}}_x^h$. (The left- and right- sides of the tree are determined by the whether the part of the excursion $\zeta$ that describes them appears before or after the unique index that corresponds to the vertex $x_h$.) Williams' decomposition states that, under $\mathbf N^{H=h}$, the collection $(x,\overleftarrow{\mathcal{T}}_x^h, \overrightarrow{\mathcal{T}}_x^h)$ with $x \in [0,h]$ and at least one of $\overleftarrow {\mathcal{T}}_x^h$ and $\overrightarrow{\mathcal{T}}_x^h$ being non-empty is countable, and that the point measure over such elements
\[\sum_{x\in [0,h]:\:\overleftarrow {\mathcal{T}}_x^h\cup\overrightarrow {\mathcal{T}}_x^h\neq\emptyset} \delta_{(x,\overleftarrow {\mathcal{T}}_x^h, \overrightarrow {\mathcal{T}}_x^h)}\]
is a Poisson random measure with intensity
\begin{equation} \label{eq:intensityHdec}
\frac{1}{4}\mathbf 1_{[0,h]}(x) \, {d} x \, \left(\delta_\emptyset ( d \overleftarrow{\mathcal{T}}_x^h) \, \mathbf N \left( \mathbf 1_{\{H \leq x\}} \, d \overrightarrow{\mathcal{T}}_x^h \right)+ \delta_\emptyset(d \overrightarrow{\mathcal{T}}_x^h) \, \mathbf N \left( \mathbf 1_{\{H \leq x\}} \,  d \overleftarrow{\mathcal{T}}_x^h \right)\right).
\end{equation}
(The factor $\frac14$ differs from that in \cite{AbDel} due to our choice of normalisation of the Brownian CRT.) As a consequence, conditional on their branching points $x \in [0,h]$ and on their heights $(h_x)_x$, the atoms of the point process, the trees $\overrightarrow{\mathcal{T}}_x^h$ or $\overleftarrow{\mathcal{T}}_x^h$ are independent and have respective laws $\mathbf N ^{H=h_x}$. (In addition, the event that there exists an $x$ such that both $\overrightarrow{\mathcal{T}}_x^h$ or $\overleftarrow{\mathcal{T}}_x^h$ are nonempty has probability $0$.) It is straightforward to track this decomposition in our recursive construction of $(\mathcal T_i)_{i \in \Sigma_*}$. Indeed, for every $n \geq 1$, conditional on $(h_i,\rho_i,x_i)_{|i| \leq n}$, the random trees $(\mathcal T_i)_{|i| \leq n}$ are independent and have respective laws $\mathbf N^{H=h_i}$ (and are on the left hand side or right hand side of $[[\rho, x_h]]_{\mathcal{T}}$ with probability $1/2$).

Now fix $\theta > 1$. From the intensity measure \eqref{eq:intensityHdec}, we calculate
\begin{align*}
\mathbf N^{H=h} \left( \sup_{x\in [0,h]} \left\{ H(\overrightarrow{\mathcal{T}}_x^h)^\theta \vee  H(\overleftarrow{\mathcal{T}}_x^h)^\theta \right\} \right)
&\leq\mathbf N^{H=h} \left( \sum_{x\in [0,h]} H(\overrightarrow{\mathcal{T}}_x^h)^\theta +  H(\overleftarrow{\mathcal{T}}_x^h)^\theta \right)\\
&= \frac{1}{2}\int_0^h {d}x \, \mathbf N \left(H^\theta \, \mathbf 1_{H \leq x} \right)\\
&= \frac{1}{2}\int_0^h {d}x \, \int_0^x \frac{u^\theta}{u^2} {d} u\\
&= \frac{1}{2\theta (\theta -1)} \, h^\theta.
\end{align*}
This immediately yields, for every $i \in \Sigma_*$:
\[\mathbf N^{H=h} \left( \sup_{j \geq 1} H(\mathcal T_{ij})^\theta  \middle| H(\mathcal T_i) \right)
\leq\frac{1}{2 \theta (\theta -1)} \, H(\mathcal T_i)^\theta.\]
By induction, for any $k \geq 1$, we therefore have
\[\mathbf N^{H=h} \left( \sup_{i \in \Sigma_k} H (\mathcal T_{i})^\theta  \right)\leq \left( \frac{1}{2 \theta (\theta -1)}\right)^k \, h^\theta.\]
Applying Borel-Cantelli with $\theta$ suitably large gives the result.
\end{proof}

We can now define our approximations $\mathcal T^{(\varepsilon)}$ for $\mathcal T$ that fit with the setting of Section \ref{bmsec}. Supposing $\mathcal{T}$ is chosen according to $\mathbf N^{H=h}$, for every $\varepsilon >0$, we set
\begin{equation} \label{eq:defTepsilon}
\mathcal T^{(\varepsilon)} := [[\rho,x_\emptyset]]_{\mathcal{T}} \cup \bigcup_{i \in \Sigma_* \, : \, h_i > \varepsilon} \mathcal [[\rho_i,x_i]]_{\mathcal{T}}.
\end{equation}
Lemma~\ref{lem:supHTi} ensures that, almost-surely, the trees $\mathcal T^{(\varepsilon)}$ are finite unions of lines and that $d_\mathcal{T}^H\left(\mathcal{T}^{(\varepsilon)},\mathcal{T}\right)\to 0$ as $\varepsilon \to 0$. In addition, from our definition of the $\mathcal T_i$'s, the $\mathcal T^{(\varepsilon)}$'s are subtrees of $\mathcal T$ and are nested when $\varepsilon$ decreases. In particular, by relabelling, it is possible to suppose that the $x_i$'s that appear in the above description (together with $\rho$) can be written as $(x_i)_{i=0}^{M(\varepsilon)}$ for some sequences $(x_i)_{i=0}^\infty$ and $(M(\varepsilon))_{\varepsilon>0}$ that satisfy the assumptions set out at the start of Section \ref{sec:covertimeapprox}. (Namely, the sequence $(x_i)_{i=0}^\infty$ is dense and composed of distinct points, and $M(\varepsilon) \rightarrow\infty$ as $\varepsilon \to 0$. Strictly speaking, these assumptions are only met when $\varepsilon\in(0,d_\mathcal{T}(\rho,x_1))$, but a straightforward reparameterisation from this range to $(0,\infty)$ shows that the results of Section \ref{bmsec} hold even when Assumption \ref{assu:2} is given in terms of this smaller range of $\varepsilon$.) Finally, we have the following Poisson description.

\begin{proposition} \label{prop:poissonTepsilon}
Fix $\varepsilon >0$ and denote by $\lambda^{(\varepsilon)}$ the one-dimensional Hausdorff measure on $\mathcal T^{(\varepsilon)}$. Denote by $(\mathcal T_x)_{x \in \mathcal T^{(\varepsilon)}}$ the countable collection of closures of the connected components of $\mathcal T \setminus \mathcal T^{(\varepsilon)}$. For every $x \in \mathcal T^{(\varepsilon)}$, we denote by $h_x$ the distance between $x$ and the  leaf $x_i$ such that $x$ belongs to the branch $[[\rho_i,x_i]]_{\mathcal{T}}$ (which is unique apart from when $x$ is a branch point of $\mathcal{T}^{(\varepsilon)}$).  Conditional on $\mathcal T^{(\varepsilon)}$, the point process
\[\sum_{x \in \mathcal T^{(\varepsilon)}} \delta_{(x,\mathcal T_x)}\]
is an inhomogeneous Poisson point process with intensity measure
\[ \frac{1}{2}\lambda^{(\varepsilon)} ({d} x) \, \mathbf N ^{H \leq h_x\wedge \varepsilon} (d \mathcal T_x),\]
where $\mathbf{N}^{H\leq h}:=\int_0^h \frac{dh'}{(h')^2} N^{H=h'}=\mathbf{N}\left(\cdot\cap \{H\leq h\}\right)$.
\end{proposition}
\begin{rem}\label{remrem}
Note that, for a given $\varepsilon > 0$, the intensity measure of Proposition \ref{prop:poissonTepsilon} is dominated by the homogeneous measure
\[ \frac{1}{2}\lambda^{(\varepsilon)} ({d} x) \, \mathbf N ^{H \leq \varepsilon} ({d} \mathcal T_x). \]
Moreover, although in the statement of the above result we only explicitly keep track of the subsets of the form $\mathcal{T}_x$, we have implicitly included the restriction of the metric $d_\mathcal{T}$ on these and the marked point that represents their root. Additionally, we could have included the restriction of the measure $\mu_\mathcal{T}$, and thus considered each $\mathcal{T}_x$ as a marked compact metric space equipped with a finite Borel measure. We will later need this more refined viewpoint.
\end{rem}
\begin{proof}
Fix $\varepsilon >0$ and set $I_\varepsilon = \{ i \in \Sigma_*: \, h_i > \varepsilon \}$. Under $\mathbf P$, Lemma~\ref{lem:supHTi} ensures that the random set $I_\varepsilon$ is finite almost-surely. In addition, $\mathcal T^{(\varepsilon)} = [[\rho , x_\emptyset ]] \cup \left( \bigcup_{i \in I_\varepsilon} [[\rho_i,x_i]] \right)$ by construction. Therefore, the collection $(\mathcal T_x)_{x \in \mathcal T^{(\varepsilon)}}$ of Proposition \ref{prop:poissonTepsilon} consists of all the trees $\mathcal T_j$ with $j\notin I_\varepsilon$. Alternatively,
\begin{align*}
(\mathcal T_x)_{x \in \mathcal T^{(\varepsilon)}}
&= \bigcup_{i \in I_\varepsilon} \left( \bigcup_{j \geq 1 \, : \, ij \notin I_\varepsilon} \mathcal T_{ij} \right) = \bigcup_{i \in I_\varepsilon} \left( \bigcup_{j \geq 1 \, : \, h_{ij} \leq \varepsilon} \mathcal T_{ij} \right),
\end{align*}
where the unions are disjoint. Given $\mathcal T^{(\varepsilon)}$ and thanks to~\eqref{eq:intensityHdec}, for every $i \in I_{\varepsilon}$, the point measure
\[\mathcal P_i : = \sum_{j \geq 1} \delta_{(\rho_{ij},\mathcal T_{ij})}\]
is a Poisson point process with intensity $\frac12\mathbf 1_{[[\rho_i, x_i]]_{\mathcal{T}}}(x) {d} x \, \mathbf N \left( \mathbf 1_{H \leq h_x} \, {d} \mathcal{T}_x\right)$. Therefore the point measure
\[\mathcal P_i^{(\varepsilon)} : = \sum_{j \geq 1 \, : \, h_{ij} \leq \varepsilon} \delta_{(\rho_{ij},\mathcal T_{ij})}\]
is a Poisson point process with intensity
\[\frac12\mathbf 1_{[[\rho_i, x_i]]_{\mathcal{T}}}(x) {d} x \, \mathbf N \left( \mathbf 1_{H \leq h_x \wedge \varepsilon} \, {d} \mathcal{T}_x\right) = \frac12\mathbf 1_{[[\rho_i, x_i]]_{\mathcal{T}}}(x) {d} x \, \mathbf N^{H \leq h_x \wedge \varepsilon} \left( {d} \mathcal{T}_x\right).\]
Given $\mathcal{T}^{(\varepsilon)}$, the point measures $\mathcal P_i^{(\varepsilon)}$ have disjoint supports and are independent, and the result follows.
\end{proof}

\section{A path property of BESQ-Brownian snakes}\label{sec:bes}

In this section we to prove a path property of snakes driven by a Brownian motion with spatial displacements following a 0-dimensional BESQ process. In Section~\ref{sec:snakedef}, we first give a quick overview of snakes and then formally state our path property as Theorem~\ref{th:snake}. The proof of this result involves the special Markov property of snakes that we state in Section~\ref{sec:specialMarkov}. These two sections cover well-known properties of Brownian snakes and we refer the reader to Le Gall's book \cite[Chapter IV]{LG} for more details. (An interested reader might also want to refer to \cite{DLeG2}, where more general branching mechanisms are treated.) Finally, in Section~\ref{sec:proofsnake}, we prove Theorem~\ref{th:snake}.

\subsection{Definition and statement of the path property} \label{sec:snakedef}

We start by giving a description of the snakes we will work with, which are path-valued processes. A finite path is a continuous mapping $\mathrm w : [0,\zeta] \to [0, \infty)$, and $\zeta$ is called the lifetime of $\mathrm w$. The restriction to nonnegative values comes from the fact that we will only deal with Bessel squared processes.
Let $\mathcal W$ be the set of all finite paths in $[0, \infty)$. For $\mathrm w \in \mathcal W$, we denote by $\zeta_{(\mathrm w)}$ the lifetime of $\mathrm w$. The tip of the path $\mathrm w$ is denoted by $\widehat{\mathrm w} = \mathrm w ( \zeta_{(\mathrm w)} )$. For every $x \in \mathbb R$, we set $\mathcal W_x = \left\{ \mathrm w \in \mathcal W : \mathrm w (0) = x \right\}$, and we identify the trivial path of $\mathcal W_x$ with lifetime $0$ with the point $x$. The space $\mathcal W$ is equipped with the following distance:
\begin{equation}\label{distd}
   d(\mathrm w_1, \mathrm w_2) := \left| \zeta_{(\mathrm w_1)} - \zeta_{(\mathrm w_2)} \right| + \sup_{t \geq 0} \left| \mathrm w_1 ( t \wedge \zeta_{(\mathrm w_1)} ) - \mathrm w_2 ( t \wedge \zeta_{(\mathrm w_2)} ) \right|.
\end{equation}
A path-valued process is a random element of the space $\mathcal C (\mathbb R_+, \mathcal W)$ of continuous functions from $\mathbb R_+$ to $\mathcal W$ equipped with the topology of uniform convergence on every compact subset of $\mathbb R_+$.

For $\mathrm w \in \mathcal W$, the law of the \emph{zero-dimensional BESQ-Brownian snake} started from $\mathrm{w}$, which we will denote by $\mathbb P_{\mathrm w}$, is characterised as follows. Note that we will write $(W_s)_{s\geq 0}$ for a random element of $\mathcal C (\mathbb R_+, \mathcal W)$ sampled according to $\mathbb P_{\mathrm w}$. First, under $\mathbb P_{\mathrm w}$, the lifetime process $\left( \zeta_{(W_s)} \right)_{s \geq 0}$ is a reflected Brownian motion started at $\zeta_{(\mathrm{w})}$. To simplify notation we write $\zeta_s := \zeta_{(W_s)}$ for every $s \geq 0$. Second, conditionally on the lifetime process $(\zeta_s)_{s\geq 0}$, the law of process $(W_s)_{s \geq 0}$, denoted by $\Theta_{\mathrm w}^{\zeta}$, satisfies
\begin{enumerate}
\item The initial value $W_0$ is $\Theta_{\mathrm w}^{\zeta}$-almost-surely equal to $\mathrm w$.
\item The process $(W_s)_{s \geq 0}$ is time-homogeneous Markov under $\Theta_{\mathrm w}^{\zeta}$. Its transition kernel is specified as follows. For $0 \leq s \leq s'$:
\begin{itemize}
\item $W_{s'} (t) = W_s(t)$ for all $t \leq m_\zeta(s,s') := \min \{ \zeta_r : s \leq r \leq s' \}$;
\item The random path $\left( W_{s'} (m_\zeta(s,s') + t) \right)_{0 \leq t \leq \zeta_{s'} - m_\zeta(s,s')}$ only depends on $W_s$ through $W_{s} (m_\zeta(s,s'))$ and is distributed as a zero-dimensional squared Bessel process started at $W_{s} (m_\zeta(s,s'))$ and stopped at time $\zeta_{s'} - m_\zeta(s,s')$.
\end{itemize}
\end{enumerate}
The full construction of this process is covered in \cite[Chapter IV]{LG} when the zero-dimensional BESQ process is replaced by a generic Markov process with uniformly H\"older continuous sample paths (see Assumption (C) in Section IV.4 of \cite{LG}). In particular, \cite[Chapter IV]{LG} proves that the process has continuous sample paths (Proposition 5) and is strong Markov (Theorem 6). Inconveniently for our purposes, the zero-dimensional BESQ process does not have uniformly H\"older continuous sample paths and a (very minor) tweak is needed to prove that our snake process is still continuous. We include this proof (which replaces the proof of Proposition 5 in \cite[Chapter IV]{LG}) for the sake of completeness. The proof of the strong Markov property (Theorem 6 in \cite[Chapter IV]{LG}) is identical as it does not use the H\"older property.

\begin{proposition}[{cf.\ \cite[Proposition IV.5]{LG}}]
For every $\mathrm w \in \mathcal W$, the process $(W_s)_{s \geq 0}$ has a continuous modification under $\mathbb P_{\mathrm w}$.
\end{proposition}
\begin{proof}
The proof is almost identical to that of Le Gall and we only sketch the main difference. In particular, we only prove the statement for $\mathrm w \equiv x$ for some $x>0$. The extension to any starting path is then identical to Le Gall's proof. We denote by $\xi = (\xi_t)_{t \geq 0}$ a zero-dimensional squared Bessel process and by $\Pi_x$ the law of $\xi$ started from $x$. Fix a deterministic function $f \in \mathcal C(\mathbb R_+ , \mathbb R_+)$ such that $f(0) =0$, as well as constants $p > 2$ and $T >0$. By the definition of the distance at \eqref{distd}, for every $0 \leq s \leq s' \leq T$, one has
\begin{align*}
\lefteqn{\Theta_x^f \left( d (W_s , W_{s'})^p\right)}\\
&\leq c_p \left( |f(s) - f(s')|^p + \Pi_x \left( \sup_{t \in [m_f(s,s'), f(s)]} |\xi_t - \xi_{m_f(s,s')} |^p \right) \vphantom{+ \Pi_x \left( \sup_{t \in [m_f(s,s'), f(s')]} |\xi_t - \xi_{m_f(s,s')} |^p \right)} \right.\\
&\qquad \qquad\qquad \qquad\qquad \qquad\left.\vphantom{ |f(s) - f(s')|^p + \Pi_x \left( \sup_{t \in [m_f(s,s'), f(s)]} |\xi_t - \xi_{m_f(s,s')} |^p \right)}+ \Pi_x \left( \sup_{t \in [m_f(s,s'), f(s')]} |\xi_t - \xi_{m_f(s,s')} |^p \right) \right).
\end{align*}
Applying the Markov property of $\xi$, we get
\begin{align} 
\lefteqn{\Theta_x^f \left( d (W_s , W_{s'})^p \right)}\nonumber\\
&\leq& c_p \left( |f(s) - f(s')|^p + 2 \Pi_x \left( \Pi_{\xi_{m_f(s,s')}} \left( \sup_{t \in [0, f(s)\vee f(s') - m_f(s,s')]} |\xi_t - \xi_0 |^p \right) \right) \right).\label{eq:markovsnakecontinuity}
\end{align}
Now, since we will be supposing that $f$ is a typical realisation of a Brownian path, and such paths are $(2-\eta)$-H\"older continuous, for $\eta$ arbitrarily small, there exists a constant $C$ such that $f$ satisfies
\[|f(s) - f(s')| \leq C |s - s'|^{\frac{1}{2} - \eta}.\]
for $0 \leq s \leq s' \leq T$. This is where we need to tweak Le Gall's proof, since the H\"older continuity of $\xi$ is not uniform with respect to its starting point. However, by the scaling property of Bessel processes, one can check that, for every $\varepsilon >0$, $x \geq 0$ and $T\geq0$, there exists $q>0$ and a constant $C$ that does not depend on $x$ or $T$ such that
\[\Pi_x \left( \sup_{0\leq t \leq T } | \xi_t - x|^p \right) \leq C \, x^q T^{p/2}.\]
Plugging this into \eqref{eq:markovsnakecontinuity} yields
\begin{align*}
\Theta_x^f \left( d (W_s , W_{s'})^p \right) &\leq C' \left( |s - s'|^{p(\frac{1}{2} - \eta)} +  \left(f(s)\vee f(s')-m_f(s,s')\right)^{p/2} \, \Pi_x \left( \xi_{m_f(s,s')}^q \right) \right),\\
&\leq C' \left( |s - s'|^{p(\frac{1}{2} - \eta)} + |s - s'|^{\frac{p}{2}(\frac{1}{2} - \eta)} \right),
\end{align*}
where the constant $C'$, which can depend on $p,q,x,f,T$, is finite, since Bessel processes have finite polynomial moments. Choosing $p>4$ and then $\eta$ accordingly, we get
\[\Theta_x^f \left( d (W_s , W_{s'})^p \right) \leq C  \, |s - s'|^{p'},\]
with $p'  >1$, and $C$ a finite positive constant that depends on $x$, $f$ and $T$ (and $p$ and $p'$). Kolmogorov's lemma ensures that there is a continuous modification of $(W_s)_{s \geq 0}$ under $\Theta_x^f$, and the proposition follows.
\end{proof}

We will denote by $\mathbb P_x$ the law of the snake started at $x > 0$. Since $0$ is an absorbing point for the zero-dimensional Bessel process, we can extend the definition to $x = 0$ by simply setting $W_s(t) = 0$ for all $s \geq 0$ and $0\leq t \leq \zeta_s$. Every point $x \geq0$ is then regular for $W$ in the sense that $\mathbb P_x \left( \inf \{s >0 : W_s = x \} = 0 \right) = 1$. The general excursion theory for Markov processes can then be used to construct the excursion measure of $W$ away from $x$, which we denote by $\mathbb N_x$.

Under the (infinite) measure $\mathbb N_x$, the lifetime process $(\zeta_s)_{s \geq 0}$ is distributed according to the It\^o excursion measure $\mathbf n$, normalised as in \eqref{eq:Ito}. As before, we will denote by $\sigma$ the duration of an excursion under $\mathbb N_x$, that is $\sigma = \sup \{ s >0 : \, \zeta_s >0 \}$. To connect with the discussion of Section \ref{bmsec}, we highlight that if $\mathcal{T}$ is the real tree associated with $\zeta$ through the procedure described at the start of Section \ref{sec:rec}, then the map $s\mapsto \widehat W_s = W_s(\zeta_s)$ to the tip of the path induces a well-defined map $y\mapsto\xi^\mathcal{T}_y$ on $\mathcal{T}$ given by setting $\xi^\mathcal{T}_y= \widehat W_s $ for any $s$ in the equivalence class that corresponds to $y\in\mathcal{T}$. It is straightforward to check from the construction that, conditional on $\mathcal{T}$, $(\xi^\mathcal{T}_y)_{y\in\mathcal{T}}$ is a $\mathcal{T}$-indexed 0-dimensional BESQ process started from $x$. (Cf.\ the discussion of \cite{DLeG}, which presented a similar discussion for tree-indexed Brownian motion.)

The main conclusion of this section is the following. For its statement, we again recall that $\widehat W_s = W_s(\zeta_s)$ is the tip of the path $W_s$. Specifically, the result gives that either $\widehat W$ does not hit zero, or it takes the value zero on an open set; this will be a key input into checking Assumption \ref{assu:2} for the Brownian CRT.

\begin{theorem}\label{th:snake}
 We have, for every $x >0$:
\[\mathbb N_x \left( \left(\left\{ \widehat W_s = 0 \, \text{on an open interval} \right\} \cup \left\{ \forall s \in [0,\sigma] , \, \widehat W _s \neq 0\right\} \right)^c\right) = 0.\]
\end{theorem}

This result is a consequence of the special Markov property of the Brownian snake, first stated in Le Gall's book \cite{LG}. A more general version is available in the Appendix of \cite{LGsubordinationoftrees}. In the next section, we state the relevant property in the setting that we will use to prove Theorem~\ref{th:snake}.

\subsection{The special Markov property} \label{sec:specialMarkov}

In this section, $x >0$ is fixed. For any path $\mathrm w \in \mathcal W_x$, we write
\[\tau (\mathrm w) = \inf \{ t \geq 0 : \mathrm w (t) = 0\}\]
for the hitting time of $0$ by $\mathrm w$ (or equivalently, since $x >0$, the first exit time of the domain $(0, + \infty)$), with the convention $\inf \emptyset = + \infty$. We define the first hitting time of zero by the snake to be
\[T := \inf \left\{s \geq 0 : \widehat W_s = 0\right\} = \inf \left\{s \geq 0 : \tau ( W_s ) < +\infty \right\}.\]
Since the spatial displacement of our Brownian snake is a zero-dimensional squared Bessel process, we have
\begin{equation*} 
\mathbb P_x \left( T = + \infty \right) = 0\quad \text{and} \quad 0 < \mathbb N_x \left( T < + \infty \right) < \infty.
\end{equation*}
The first equality is straightforward and if $N_x \left( T < + \infty \right) = 0$, by excursion theory, under $\mathbb P_x$, no excursion would hit $0$ contradicting the first inequality. The last inequality reads
\begin{equation} \label{eq: N_x T finite is finite}
\mathbb N_x \left( T < + \infty \right) = \mathbb N_x \left(\exists s \in [0,\sigma] : \widehat W_s = 0 \right) < \infty.
\end{equation}
If \eqref{eq: N_x T finite is finite} was not true, then, again by excursion theory, under $\mathbb{P}_x$, infinitely many excursions of the snake would exit $[x/2, \infty)$ before any given positive time. This contradicts the fact that the snake is continuous under $\mathbb{P}_x$.

We will work under the probability measure $\mathbb N_x$ conditioned on the event $\{ T < +\infty \}$:
\[\mathbb N_x^T := \mathbb N_x \left( \, \cdot \, \middle| T < +\infty \right).\]
This probability measure can also be interpreted as the law under $\mathbb P_x$ of the first Brownian snake excursion away from $x$ that hits $0$. The exit local time process from $(0,+\infty)$ under $\mathbb N_x^T$ is the process $(\Lambda_s)_{0 \leq s \leq \sigma}$ defined by
\[\Lambda_s : = \lim_{\varepsilon \to 0} \int_0^s \mathbf{1}_{\{\tau(W_r) < \zeta_r < \tau(W_r) + \varepsilon \}} {d}r,\]
and the measure ${d} \Lambda_s$ is supported on $\{s \in [0,\sigma] : \tau(W_s) = \zeta_s \}$ (see \cite[Section V.1]{LG} and \cite[Appendix]{LGsubordinationoftrees}).

In our setting, the special Markov property describes the law of the excursions of the zero-dimensional BESQ-Brownian snake at $0$, conditionally on the information provided by the paths of the snake before they hit $0$. Let us next describe the filtration $\mathcal F$ containing this information. For every $s \geq 0$, we set
\[\eta_s := \inf \left\{ t \geq 0 : \int_0^t \mathbf{1}_{\{\zeta_r \leq \tau(W_r) \}} {d}r > s \right\}.\]
This definition makes sense $\mathbb P_x$ almost-surely, and we define $\mathcal F$ to be the $\sigma$-field generated by the process $\left( W_{\eta_s}\right)_{s \geq 0}$ and the $\mathbb P_x$-negligible sets. We note that $(\Lambda_{\eta_s})_{s=0}^\sigma$ is $\mathcal{F}$-measurable (see \cite[Lemma 19]{LGsubordinationoftrees}).

Now let us define the excursions at $0$ of the snake. For every $s \geq 0$, we set $\gamma_s := \left( \zeta_s - \tau(W_s) \right)^+$. Under $\mathbb P_x$, the set $\{s \geq 0 : \gamma_s >0\}$ is almost-surely a countable union of disjoint open intervals:
\[\{s \geq 0 : \gamma_s >0\} = \bigcup_{i \in \mathbb N} (a_i,b_i),\]
and the enumeration can be fixed in such a way that the variables $a_i$ and $b_i$ are measurable with respect to the $\sigma$-field generated by $(\gamma_s)_{s \geq 0}$.

From the properties of the Brownian snake and of the zero-dimensional squared Bessel process, we have that $\mathbb P_x$-a.s.,\ for every $i$ and every $s \in [a_i,b_i]$,
\[\tau(W_s) = \tau(W_{a_i}) = \zeta_{a_i},\]
and that all the paths $W_s$ for $s\in [a_i,b_i]$ coincide up to their hitting time of $0$, from which time they remain at $0$ up to their lifetime. For every $i$, we define $W^{(i)}$ by setting
\[W^{(i)}_s(t) := W_{(a_i +s) \wedge b_i} \left( \zeta_{a_i} + t \right), \quad \text{for $s \geq 0$ and $0 \leq t \leq \zeta_s^{(i)}:= \zeta_{(a_i +s) \wedge b_i} - \zeta_{a_i}$}.\]
These variables are the excursions of the zero-dimensional BESQ-Brownian snake at $0$, and their structure is fully described by the special Markov property. The variant of this property that we will use is the following.

\begin{theorem}[{Special Markov property, \cite[Corollary 22]{LGsubordinationoftrees}}] \label{th:specialMarkov}
Under $\mathbb N_x^T$, conditionally on the $\sigma$-field $\mathcal F$, the point measure
\[\sum_{i\in \mathbb N} \delta_{W^{(i)}}({d} \omega)\]
is Poisson with intensity
\[ \int \Lambda_\sigma \, \mathbb N_0({d} \omega ) = \int \Lambda_\sigma \, \mathbf n ({d} \zeta),\]
with a slight abuse of notation in the last equality (the associated spatial process is understood to be the constant 0).
\end{theorem}

\subsection{Proof of Theorem~\ref{th:snake}} \label{sec:proofsnake}

Theorem~\ref{th:snake} states that $\mathbb N_x$-almost-surely, on the event when the snake hits $0$, it has at least one excursion of positive length at $0$. A direct reformulation of the statement we want to prove is: $\mathbb N_x$-a.s.,\
\[\left\{ \exists s_1, s_2 :\:  0 < s_1 < s_2 < \sigma \text{ and } \widehat W_s = 0,\:\forall s \in (s_1,s_2)\right\} = \left\{ \exists s \in [0,\sigma]:\: \widehat W _s = 0\right\}.\]
For $\varepsilon \geq 0$, set
\[A_\varepsilon := \left\{ \widehat W_s = 0 \text{ on an open interval of length $>\varepsilon$} \right\}.\]
The special Markov property stated in Theorem~\ref{th:specialMarkov} and our normalisation for It\^o excursion measure (see \eqref{eq:Ito}) gives, for $\varepsilon >0$:
\[\mathbb N_x^T (A_\varepsilon) = \mathbb N_x^T \left( 1 - \exp \left( - \frac{\Lambda_\sigma}{\sqrt{2 \pi \varepsilon}} \right) \right).\]
Taking the limit when $\varepsilon >0$ we obtain
\[\lim_{\varepsilon \to 0} \mathbb N_x^T (A_\varepsilon) = \mathbb N_x^T (\Lambda_\sigma >0),\]
or equivalently, $\mathbb N_x$-a.s.,\
\begin{equation}\label{firstineq}
    \left\{ \Lambda_\sigma >0 \right\} =\left\{ \widehat W_s = 0 \quad \text{on an open interval} \right\}.
\end{equation}
We will now prove, $\mathbb N_x$-a.s.,\
\[\left\{ \Lambda_\sigma >0 \right\} = \left\{ \exists s \in [0,\sigma] , \, \widehat W _s = 0\right\},\]
which will then immediately give Theorem~\ref{th:snake}.

For every $x>0$, set
\[u(x) : = \mathbb N_x \left( \exists s \in [0,\sigma] , \, \widehat W _s = 0\right).\]
Writing $\mathbb N_x^{(u)}$ for $\mathbb{N}_x(\cdot|\:\sigma=u)$, the scaling properties of Brownian excursions (cf.\ \eqref{scaling}) and squared Bessel processes yield, for $\lambda >0$:
\begin{align*}
u(x) &= \mathbb N_x \left( \exists s \in [0,\sigma] , \, \widehat W _s = 0\right)\\
& = \int_0^\infty \frac{d u}{2 \sqrt{2 \pi u^3}} \mathbb N_x^{(u)} \left( \exists s \in [0,\sigma] , \, \widehat W _s = 0\right)\\
& = \frac{1}{\lambda} \int_0^\infty \frac{d u}{2 \sqrt{2 \pi u^3}} \mathbb N_x^{(\lambda^2 u)} \left( \exists s \in [0,\sigma/\lambda^2] , \, \frac{1}{\lambda} \widehat W _{\lambda^2 s} = 0\right)\\
&=\frac{1}{\lambda} \int_0^\infty \frac{d u}{2 \sqrt{2 \pi u^3}} \mathbb N_{x/\lambda}^{(u)} \left( \exists s \in [0,\sigma] , \, \widehat W _{s} = 0\right)\\
& = \lambda^{-1} \, \mathbb N_{x/\lambda} \left( \exists s \in [0,\sigma] , \, \widehat W _s = 0\right) \\
&=\lambda^{-1} \, u(x/\lambda).
\end{align*}
Since we know from \eqref{eq: N_x T finite is finite} that the function $u$ is not infinite, there is a finite constant $C = u(1) \geq 0$ such that, for every $x > 0$,
\[u(x) = C \, x^{-1}.\]
In addition, Proposition VI.2 of Le Gall's book~\cite{LG} states that the function $u$ is the maximal non-negative solution of the equation
\begin{equation}\label{ade}
\mathcal A u (x) = 2 u^2 (x),\qquad \forall x >0,
\end{equation} 
where $\mathcal A$ is the Markov generator of the zero-dimensional BESQ process:
\[\mathcal A f(x) = 2xf''(x),\]
see \cite[Section XI.1]{RevuzYor}. Precisely, \cite[Proposition VI.2]{LG} treats a Brownian snake whose spatial displacements follow a Brownian motion. However, it is possible to check that, with \eqref{eq: N_x T finite is finite} in place of \cite[Proposition V.9(i)]{LG}, the same argument applies. Applying \eqref{ade}, a quick calculation gives that either $C = 0$ or $C = 2$, and since $u$ is maximal, we have $u(x) = 2 x^{-1}$ for every $x>0$. (Note that, apart from the difference in the Markov process describing the spatial displacements, this argument is the same as that given by Le Gall in the example appearing after \cite[Proposition VI.2]{LG}.)

\bigskip

Similarly, we can calculate $v(x) := \mathbb N_x \left( \Lambda_\sigma > 0\right)$. Since $v(x) \leq \mathbb N_x ( \exists s \in [0,\sigma] , \, \widehat W _s = 0)$, we obtain from \eqref{eq: N_x T finite is finite} that $v(x)<\infty$ for any $x>0$. Moreover, the same scaling arguments as for $u$ give
\[v(x) = \mathbb N_x \left( \Lambda_\sigma >0 \right) = \lambda^{-1} \, \mathbb N_{x/\lambda} \left( \Lambda_\sigma > 0\right) =\lambda^{-1} \, v(x/\lambda).\]
Therefore, there is another finite constant $C' \geq 0$ such that, for every $x > 0$,
\[v(x) = C' \, x^{-1}.\]
For $n >0$, set
\[v_n(x) : = \mathbb N_x \left( 1 - \exp \left( - n \, {\Lambda_\sigma} \right) \right).\]
A straightforward adaptation to the case of a 0-dimensional BESQ-Brownian snake of \cite[Theorem V.6]{LG} gives that $v_n$ satisfies
\[\begin{cases}
\mathcal A v_n = 2 v_n^2 \quad \text{on} \quad (0,+\infty),\\
\lim_{x \to 0^+} v_n(x) = n.
\end{cases}\]
We highlight that, as for the result of \cite{LG} that was cited in the previous part of the proof, \cite[Theorem V.6]{LG} is written for a Brownian snake whose spatial displacements follow a Brownian motion. However, the proof is readily adapted by using the fact that a 0-dimensional BESQ process can be written as the time-change of a Brownian motion. (See the proof of Lemma \ref{besqlem} for details.) Indeed, the only sensitive part of Le Gall's argument is an appeal to \cite[Theorem 6.6]{PS} to check that, for a suitable function $f$, it holds that $\frac{1}{2}\Delta \int G^{\mathrm{BM}}_B(x,y)f(y)dy=-f(x)$, where $G^{\mathrm{BM}}_B(x,y)$ is the Green's function of Brownian motion killed on exiting a ball $B$ contained in $(0,\infty)$. In our case, we are required to check $\mathcal{A}\int G^{\mathrm{BESQ^0}}_B(x,y)f(y)dy=-f(x)$ for suitable $f$, where $G^{\mathrm{BESQ^0}}_B(x,y)$ is the corresponding Green's function of the 0-dimensional BESQ process. However, since $G^{\mathrm{BESQ^0}}_B(x,y)=\frac{1}{4y}G^{\mathrm{BM}}_B(x,y)$ and the density of the speed measure, $\frac{1}{4y}$, is bounded away from 0 and $\infty$ on compact subsets of $(0,\infty)$, there is no problem in applying the result of \cite{PS} in our case. We next note that, since $v$ is the pointwise limit of the increasing sequence of functions $(v_n)$ and the set of non-negative solutions on $(0,+\infty)$ of the equation $\mathcal A f = 2f^2$ is closed under pointwise convergence (to check this, one may follow the argument of \cite[Proposition V.9(iii)]{LG}, for example), the function $v$ also satisfies $\mathcal A v = 2v^2$ on $(0,+\infty)$ and is not identically equal to $0$. Therefore we have $C'=2$, that is, for every $x>0$:
\begin{equation*} 
\mathbb N_x \left( \Lambda_\sigma >0 \right) = \mathbb N_x \left( \exists s \in [0,\sigma] , \, \widehat W _s = 0\right) = \frac{2}{x}.
\end{equation*}
Since it clearly follows from \eqref{firstineq} that $\{ \Lambda_\sigma >0 \} \subseteq \{  \exists s \in [0,\sigma] , \, \widehat W _s = 0 \}$, $\mathbb N_x$-a.s.,\ the two events are therefore $\mathbb N_x$-a.s.\ equal, and Theorem~\ref{th:snake} follows.

\section{Proof of Theorem \ref{thm:mainres}}\label{sec:mainproof}

The aim of this section is to prove Theorem \ref{thm:mainres}. Clearly, by the comment below the scaling property at \eqref{scaling}, establishing that the desired property of the Brownian CRT holds $\mathbf{P}$-a.s.\ is equivalent to showing the result almost-surely for trees under the push-forward of the It\^{o} excursion measure (i.e.,\ $\mathbf{N}$), as introduced in Section \ref{sec:rec}. Our strategy is to first prove that Assumptions \ref{assu:1} and \ref{assu:2} hold for $\mathbf{N}$-typical trees, which readily yields the conclusion of Theorem \ref{thm:mainres} when $X^\mathcal{T}$ is started from the root (see Section \ref{52}). To complete the proof, we give a small additional argument involving random re-rooting and a continuity property to verify that the result can be extended to arbitrary starting points (see Section \ref{53}). Before this, in Section \ref{51}, we formalise the setting so as to ensure all the objects of the discussion are well-defined.

\subsection{An extended Gromov-Hausdorff-type topology}\label{51}

In this section, we introduce (a simplification of) the extended Gromov-Hausdorff-type topology used in \cite{Noda} for the purposes of studying random metric-measure spaces equipped with stochastic processes and corresponding local times. (A similar viewpoint was also set out in \cite{andblanket}.) In order to keep the presentation concise, we will be brief with the details; for a more in-depth presentation, we refer the reader to \cite{Nodametric} (cf.\ \cite{KA}), which also provides a discussion of earlier works in this area.

We will consider two basic spaces: $\check{\mathbb{F}}_c$ and $\mathbb{M}_{L,c}$. For the first of these, we consider objects of the form
\begin{equation}\label{quad}
\left(K,R_K,\mu_K,\rho_K\right),
\end{equation}
where, as in the introduction, $(K,R_K)$ is a (non-empty) compact resistance metric space, $\mu_K$ is a finite Borel measure on $K$ of full support, and $\rho_K$ is a marked point (or root) of $K$. Here, we recall that a resistance metric, which was introduced by Kigami in the context of analysis on fractals \cite{KigAOF,Kigq}, is defined as follows.

\begin{defn}[{\cite[Definition 2.3.2]{KigAOF}}]
Let $K$ be a set. A function $R:K\times K\rightarrow \mathbb{R}$ is a resistance metric if, for every finite $K'\subseteq K$, one can find a weighted (i.e.\ equipped with conductances) graph with vertex set $K'$ for which $R|_{K'\times K'}$ is the associated effective resistance.
\end{defn}

Similarly to \cite{Croydonscalelim,CHK,Noda}, the collection of spaces of the form at \eqref{quad}, identified if they are measure- and root-preserving isometric, will be written $\mathbb{F}_c$. (The subscript $c$ here refers to the restriction to compact spaces.) It is known that such spaces can be associated with a canonical stochastic process $((X^K_t)_{t\geq 0},(P^K_x)_{x\in K})$ that admits jointly measurable local times $(L^K_t(x))_{x\in K,\:t\geq 0}$ satisfying an occupation density formula as at \eqref{odf}. (See the discussion in \cite{Croydonscalelim,CHK} for the existence of the process, and \cite[Lemma 2.4]{CHK}, in particular, for the existence of local times. In these papers, a certain regularity condition is imposed on the resistance metric, but this is immediate in the compact case, see \cite[Corollary 6.4]{Kigq}.) To ensure that the local times are jointly continuous, we assume a further condition involving the metric entropy, as introduced by Noda in \cite{Noda}, which we now describe. In particular, for a compact metric space $(K,d_K)$ and $\varepsilon>0$, $K'\subseteq K$ is called an $\varepsilon$-covering of $(K, d_K)$ if for each point $x\in K$, there exists an $x'\in K'$ such that $d_K(x,x')\leq \varepsilon$. Let
\begin{equation}\label{coveringdef}
N\left(K,d_K,\varepsilon\right):= \min\left\{|K'|:\:K'\text{\ is an }\varepsilon\text{-covering of }(K, d_K)\right\}.
\end{equation}
It is standard in the literature to call the family (over $\varepsilon$) of the logarithm of $N(K, d_K, \varepsilon)$, the metric entropy of $(K, d_K)$, though in \cite{Noda} (and some other places), this term is used to denote $(N(K,d_K, \varepsilon))_{\varepsilon>0}$ instead. It is shown in \cite[Corollary 4.12]{Noda} that if $(K,R_K,\mu_K,\rho_K)\in \mathbb{F}_c $ and there exists an $\alpha\in (0,1/2)$ such that
\begin{equation}\label{entcond}
\sum_{k\geq 1}^\infty N\left(K,R_K, 2^{-k}\right)^2 \exp{\left(-2^{\alpha k}\right)}=0,
\end{equation}
then the local times $L^K$ can be constructed to be jointly continuous. We denote the subset of $\mathbb{F}_c$ containing elements for which \eqref{entcond} is satisfied by $\check{\mathbb{F}}_c$; this is the space appearing in the statements of Theorem \ref{cor:convres} and Corollary \ref{anncor}.

The space ${\mathbb{F}}_c$ may be equipped with a metric yielding the marked Gromov-Hausdorff-Prokhorov metric, as given by setting
\begin{align}
\lefteqn{d_{GHP}\left(\left(K_1,R_{K_1},\mu_{K_1},\rho_{K_1}\right),\left(K_2,R_{K_2},\mu_{K_2},\rho_{K_2}\right)\right)}\nonumber\\
&=\inf_{(M,d_M),\:\phi_1,\:\phi_2}\max\left\{d_M^H\left(\phi_1(K_1),\phi_2(K_2)\right),d_M^P\left(\mu_{K_1}\circ\phi_1^{-1}, \mu_{K_2}\circ\phi_2^{-1}\right)\right\},\label{ghpmet}
\end{align}
where the infimum is taken over compact rooted metric spaces $(M,d_M,\rho_M)$ and root-preserving isometric embeddings $\phi_i:K_i\rightarrow M$, $d_M^H$ is the Hausdorff metric between compact subsets of $(M,d_M)$, and $d_M^P$ is the Prokhorov metric between Borel measures on $(M,d_M)$. Note that there are various slightly different versions of the above metric used in the literature, but, as discussed in \cite[Section 2.1]{Noda}, they all give rise to the same topology. On the larger collection of marked compact metric spaces equipped with a finite Borel measure, it is known that this topology is complete and separable, as is convenient for various probabilistic applications. To say $(K,R_K,\mu_K,\rho_K)$ is a random element of $\check{\mathbb{F}}_c$, as we do in Corollary \ref{anncor}, we mean that it is chosen according to a probability measure on this larger space whose support is contained in $\check{\mathbb{F}}_c$.

Secondly, we consider objects of the form \eqref{quintuple}, i.e.\ to each element $(K,R_K,\mu_K,\rho_K)\in \check{\mathbb{F}}_c$, we add $P^K$, which is the law of $((X^K_t)_{t\geq 0},(L^K_t(x))_{x\in K,\:t\geq 0})$ under $P^{K}_{\rho_K}$. Now, for such spaces, we have that $X^K$ is $P^K$-a.s.\ an element of $D(\mathbb{R}_+,K)$, the space of c\`{a}dl\`{a}g paths in $K$, which we equip with the usual Skorokhod $J_1$ topology, and $L^K$ takes values in $C(K\times \mathbb{R}_+,\mathbb{R}_+)$. The larger space, where we drop the constraints that $R_K$ is a resistance metric satisfying \eqref{entcond} and $\mu_K$ has full support, and allow $P^K$ to be any probability measure on $D(\mathbb{R}_+,K)\times C(K\times \mathbb{R}_+,\mathbb{R}_+)$, will be denoted $\mathbb{M}_{L,c}$. (This is the restriction of the space $\mathbb{M}_L$ of \cite{Noda} to those elements for which the metric space is compact; we do not need the generality of non-compact spaces considered in that article. The subscript $L$ refers to the fact elements of the space are equipped with local time-type functions, and $c$ again refers to the restriction to compact spaces.) To measure distances between two objects $(K_i,R_{K_i},\mu_{K_i},\rho_{K_i},P^{K_i})$, $i=1,2$, we use a similar metric to \eqref{ghpmet}, but incorporate an additional distance
\[d_M^L\left(P^{K_1}\circ\phi_1^{-1},P^{K_1}\circ\phi_2^{-1}\right)\]
which is the Prokhorov distance between the laws of 
\[\left((\phi_i(X^{K_i}_t))_{t\geq 0},(L^{K_i}_t(\phi_i^{-1}(x)))_{x\in \phi_i(K_i),\:t\geq 0}\right),\qquad i=1,2.\]
Of course, to give full definition of this Prokhorov metric, one needs to define a distance between objects of the latter form. To this end, in \cite{Noda}, such a distance is defined in terms of a standard metrisation of the Skorokhod $J_1$ topology for the components $(\phi_i(X^{K_i}_t))_{t\geq 0}$, $i=1,2$, and a notion of uniform convergence for functions with variable domains for the components $(L^{K_i}_t(\phi_i^{-1}(x)))_{x\in \phi_i(K_i),\:t\geq 0}$, $i=1,2$. See \cite[Section 2.2]{Noda} for details. Whilst we omit the precise definition, the most important consequence of it for our purposes is set out in Lemma \ref{embedding}, which describes how convergence in $\mathbb{M}_{L,c}$ implies the existence of isometric embeddings of the spaces in question with various nice properties.

Before proceeding, let us finish this section with a few remarks about measurability issues that are relevant here. Firstly, \cite[Proposition 6.1]{Noda} ensures that the map $(K,R_K,\mu_K,\rho_K)\mapsto (K,R_K,\mu_K,\rho_K,P^K)$ is a measurable mapping from $\check{\mathbb{F}}_c$ to $\mathbb{M}_{L,c}$. This ensures the integrals defining annealed measures, cf.\ \eqref{annealedlaw}, are well-defined. We will also need in the proof of Theorem \ref{cor:convres} measurability of the map
\begin{equation}\label{aimain}
\left(\mathcal{T},d_{\mathcal{T}},\mu_\mathcal{T},\rho_\mathcal{T},\ell\right)\mapsto\left(\mathcal{T},d_{\mathcal{T}},\mu_\mathcal{T},\rho_\mathcal{T},\ell,P^{\mathcal{T}\hbox{-}\mathrm{BESQ}^0(\ell)}\right),
\end{equation}
which takes a compact, measured (with a probability measure of full support), rooted real tree $(\mathcal{T},d_{\mathcal{T}},\mu_\mathcal{T},\rho_\mathcal{T})$ satisfying Assumption \ref{assu:1} (or, more generally, \eqref{entcond}) and an value $\ell\geq 0$, and returns in addition the law of the associated tree-indexed 0-dimensional BESQ process started at the root from $\ell$, $P^{\mathcal{T}\hbox{-}\mathrm{BESQ}^0(\ell)}$; we recall from Lemma \ref{besqlem} that, under Assumption \ref{assu:1}, a continuous version of such a process can be constructed. (The topology on the domain above can be assumed to be given by the product of the marked Gromov-Hausdorff-Prokhorov topology on $(\mathcal{T},d_{\mathcal{T}},\mu_\mathcal{T},\rho_\mathcal{T})$ and the usual Euclidean topology for the final component. As for the codomain, a topology here can be constructed as for local-time-type functions, using the notion of uniform convergence of functions with variable domains of \cite{Noda}.) To check the measurability in this case, one can follow the approach of \cite[Lemma 8.1]{croydonconv}, which gave a similar result for the laws of Brownian motions on real trees. In particular, it is easy to check that convergence of the form
\[\left(\mathcal{T}_n,d_{\mathcal{T}_n},\mu_{\mathcal{T}_n},\rho_{\mathcal{T}_n},\ell_n\right)\rightarrow\left(\mathcal{T},d_{\mathcal{T}},\mu_\mathcal{T},\rho_\mathcal{T},\ell\right)\]
implies, for each $k\geq 1$, convergence in distribution of
\[\left(\mathcal{T}_n,d_{\mathcal{T}_n},\mu_{\mathcal{T}_n},\rho_{\mathcal{T}_n},\ell_n,x_1^n,\dots,x_k^n\right)\rightarrow\left(\mathcal{T},d_{\mathcal{T}},\mu_\mathcal{T},\rho_\mathcal{T},\ell,x_1,\dots,x_k\right),\]
where we add additional marks to the spaces by selecting $k$ points according to the measure on the space in question. (For a suitable topology for such convergence, see \cite[Section 8.3]{Croydonscalelim} or \cite[Section 3.7.2]{KA} or \cite[Section 4.2]{Nodametric}.) Using the description of the moment generating functions of the associated BESQ process at the points $(x_1,\dots,x_k)$, cf.\ the proof of Lemma \ref{rkhit}, it is clear that the laws of the finite-dimensional marginals of these random variables simultaneously converge in distribution (considered as probability measures on the space of continuous functions with variable domains). Summarising so far, this gives that, for each $k$, the map from $(\mathcal{T},d_{\mathcal{T}},\mu_\mathcal{T},\rho_\mathcal{T},\ell)$ to the law of \[\left(\mathcal{T},d_{\mathcal{T}},\mu_\mathcal{T},\rho_\mathcal{T},\ell,P^{\mathcal{T}\hbox{-}\mathrm{BESQ}^0(\ell)}|_{x_1,\dots,x_k}\right),\]
is continuous. Finally, since the uniformly selected points become increasingly dense in $\mathcal{T}$, and the BESQ processes are continuous, it readily follows that
\[\left(\mathcal{T},d_{\mathcal{T}},\mu_\mathcal{T},\rho_\mathcal{T},\ell,P^{\mathcal{T}\hbox{-}\mathrm{BESQ}^0(\ell)}|_{x_1,\dots,x_k}\right)\rightarrow \left(\mathcal{T},d_{\mathcal{T}},\mu_\mathcal{T},\rho_\mathcal{T},\ell,P^{\mathcal{T}\hbox{-}\mathrm{BESQ}^0(\ell)}\right),\]
almost-surely, which yields that the map from the element $(\mathcal{T},d_{\mathcal{T}},\mu_\mathcal{T},\rho_\mathcal{T},\ell)$ to the law of the right-hand side above is measurable. However, the latter is deterministic, and so we get the measurability of the map at \eqref{aimain}, as desired.

\subsection{Checking the assumptions and the result with $X^\mathcal{T}$ started from $\rho$}\label{52}

We are now ready to start putting the pieces together and check the main result when $X^\mathcal{T}$ is started from the root $\rho$ (see Lemma \ref{rootres}). We first make the following observation, which readily follows from the measure estimates of \cite{Croydoncrt} (and the scaling property of $\mathbf{N}$, as described at \eqref{scaling}).

\begin{lem}[{\cite[Theorem 1.2]{Croydoncrt}}]\label{a1check}
For $\mathbf{N}$-a.e.\ realisation of $\mathcal{T}$, Assumption \ref{assu:1} is satisfied.
\end{lem}

More challenging to check is Assumption \ref{assu:2}, but we are able to do this by combining the Ray-Knight theorem of Section \ref{bmsec}, the recursive decomposition of Section \ref{sec:rec} and the path property of snakes proved in Section \ref{sec:bes}. In what follows in this section, we suppose that $\mathcal{T}^{(\varepsilon)}$ is the subtree defined in Section \ref{sec:rec}, see \eqref{eq:defTepsilon} in particular. We note that the map $\mathcal{T}\mapsto (\mathcal{T}^{(\varepsilon)},(\rho_i,\mathcal{T}_i)_{i\in\mathcal{I}^{(\varepsilon)}})$ can be defined measurably with respect to a suitable product Gromov-Hausdorff-type topology on the image space (and, indeed, continuously at typical realisations of the Brownian CRT); since it is straightforward to check this, we omit the details. (Similarly to the comment in Section \ref{sec:rec}, we highlight that it will be sufficient to check that Assumption \ref{assu:2} holds with $\varepsilon$ taking a countable selection of values in the restricted range $(0,d_\mathcal{T}(\rho,x_1)$.)

\begin{lem}\label{a2check}
For $\mathbf{N}$-a.e.\ realisation of $\mathcal{T}$, Assumption \ref{assu:2} is satisfied.
\end{lem}
\begin{proof}
Since Assumption \ref{assu:2} only requires us to check the desired property holds for a countable number of $\varepsilon$ and $\delta$, it will suffice to check it for a single pair. In particular, in this proof, we fix $\varepsilon,\delta>0$ and write $E:=E_1(\varepsilon,\delta)\cup E_2(\varepsilon,\delta)$, where the latter events were defined at \eqref{e1} and \eqref{e2}. We will show that ${\mathcal{N}(\varepsilon,\delta)}=0$, where
\[\mathcal{N}(\varepsilon,\delta):=\int\mathbf{N}\left(d\mathcal{T}\right) \mathbf{1}_{\{2\delta,\varepsilon\in (0, d_\mathcal{T}(\rho,x_1))\}}P_\rho^\mathcal{T}\left(E^c\right).\]
Now, by Corollary \ref{poscor} and Lemma \ref{a1check}, we readily deduce that
\begin{eqnarray*}
  \mathcal{N}(\varepsilon,\delta)&=& \int\mathbf{N}\left(d\mathcal{T}\right) \mathbf{1}_{\{2\delta,\varepsilon\in (0, d_\mathcal{T}(\rho,x_1))\}} P_\rho^\mathcal{T}\left(E^c,\:\inf_{z\in\mathcal{T}^{(\varepsilon)}}L^\mathcal{T}_{\tau_{\mathrm{cov}}^{(\varepsilon)}(\delta)}(z)>0\right)\\
   &\leq & \int\mathbf{N}\left(d\mathcal{T}\right)\mathbf{1}_{\{2\delta,\varepsilon\in (0, d_\mathcal{T}(\rho,x_1))\}} P_\rho^\mathcal{T}\left(E(\mathcal{T}_i)^c\mbox{ holds for some }i\in\mathcal{I}^{(\varepsilon)}\right)\\
   &\leq & \int\mathbf{N}\left(d\mathcal{T}\right)\mathbf{1}_{\{2\delta,\varepsilon\in (0, d_\mathcal{T}(\rho,x_1))\}} \sum_{i\in \mathcal{I}^{(\varepsilon)}}P_\rho^\mathcal{T}\left(E(\mathcal{T}_i)^c\right),
\end{eqnarray*}
where $E(\mathcal{T}_i)$ is defined exactly as $E$ is, but with $\mathcal{T}$ replaced by $\mathcal{T}_i$. Applying the tower property of conditional expectation and the Ray-Knight result of Proposition \ref{rk} to the inner probability, we deduce that
\begin{eqnarray*}
P_\rho^\mathcal{T}\left(E(\mathcal{T}_i)^c\right)&=&E_\rho^\mathcal{T}\left(P_\rho^\mathcal{T}\left(E(\mathcal{T}_i)^c\:\vline\:L^\mathcal{T}_{\tau_{\mathrm{cov}}^{(\varepsilon)}(\delta)}(z),\:z\in\mathcal{T}^{(\varepsilon)}\right)\right)\\
&=&\int P_\rho^{(\varepsilon)}\left(dL^{(\varepsilon)}_{\tau(\varepsilon,\delta)}\right) p\left(\mathcal{T}_i,L^{(\varepsilon)}_{\tau(\varepsilon,\delta)}(\rho_i)\right),
\end{eqnarray*}
where $p(\mathcal{T},\ell)$ is the probability a $\mathcal{T}$-indexed $\mathrm{BESQ}^0(\ell)$ process satisfies the claim equivalent to $E^c$, i.e.\ it hits zero, but is not zero on any open ball. (The measurability of this function, as is required for the above integral to be well-defined, was discussed at the end of Section \ref{51}; technically, to proceed as described there, one also needs to incorporate suitable measures on the trees, but there is no problem in equipping each tree $\mathcal{T}_i$ with the probability measure $\mu(\cdot\cap \mathcal{T}_i)/\mu(\mathcal{T}_i)$, and checking that this measure has the desired properties, almost-surely.) Next, we insert this into the above upper bound for $\mathcal{N}(\varepsilon,\delta)$, decompose $\mathbf{N}(d\mathcal{T})$ by conditioning on $\mathcal{T}^{(\varepsilon)}$ and apply Fubini to yield the following:
\begin{eqnarray*}
{\mathcal{N}(\varepsilon,\delta)} &\leq & \int\mathbf{N}\left(d\mathcal{T}\right) \mathbf{1}_{\{2\delta,\varepsilon\in (0, d_\mathcal{T}(\rho,x_1))\}} \int P_\rho^{(\varepsilon)}\left(dL^{(\varepsilon)}_{\tau(\varepsilon,\delta)}\right) \sum_{i\in \mathcal{I}^{(\varepsilon)}}p\left(\mathcal{T}_i,L^{(\varepsilon)}_{\tau(\varepsilon,\delta)}(\rho_i)\right)\\
  &=&  \int\mathbf{N}\left(d\mathcal{T}^{(\varepsilon)}\right)\mathbf{1}_{\{2\delta,\varepsilon\in (0, d_\mathcal{T}(\rho,x_1))\}} \int P_\rho^{(\varepsilon)}\left(dL^{(\varepsilon)}_{\tau(\varepsilon,\delta)}\right) \\
  &&\hspace{80pt}\int \mathbf{N}\left(d\left((\mathcal{T}_i,\rho_i)_{i\in\mathcal{I}^{(\varepsilon)}}\right)\:\vline\:\mathcal{T}^{(\varepsilon)}\right)  \sum_{i\in \mathcal{I}^{(\varepsilon)}}p\left(\mathcal{T}_i,L^{(\varepsilon)}_{\tau(\varepsilon,\delta)}(\rho_i)\right).
\end{eqnarray*}
Under the conditional law $\mathbf{N}(\cdot \:\vline\:\mathcal{T}^{(\varepsilon)})$, we have from Proposition \ref{prop:poissonTepsilon} (see also Remark \ref{remrem}) that $(\mathcal{T}_i,\rho_i)_{i\in\mathcal{I}^{(\varepsilon)}}$ are the atoms of a Poisson process with intensity measure bounded above by
$\lambda^{(\varepsilon)}(d\tilde{\rho})  \mathbf{N}^{H\leq \varepsilon}(d\tilde{\mathcal{T}})$.  Hence Campbell's theorem (see \cite[Section 3.2]{Kingman}, for example) implies that
\begin{eqnarray*}
\lefteqn{\int\mathbf{N}\left(d\left((\mathcal{T}_i,\rho_i)_{i\in\mathcal{I}^{(\varepsilon)}}\right)\:\vline\:\mathcal{T}^{(\varepsilon)}\right)  \sum_{i\in \mathcal{I}^{(\varepsilon)}}p\left(\mathcal{T}_i,L^{(\varepsilon)}_{\tau(\varepsilon,\delta)}(\rho_i)\right)}\\
&\leq &\int \lambda^{(\varepsilon)}\left(d\tilde{\rho}\right)\int
  \mathbf{N}^{H\leq\varepsilon}\left(d\tilde{\mathcal{T}}\right) p\left(\tilde{\mathcal{T}},L^{(\varepsilon)}_{\tau(\varepsilon,\delta)}(\tilde{\rho})\right).
\end{eqnarray*}
Finally, we know from Theorem \ref{th:snake} that
\[\int  \mathbf{N}^{H\leq \varepsilon}\left(d\tilde{\mathcal{T}}\right) p\left(\tilde{\mathcal{T}},\ell\right)\leq \int  \mathbf{N}\left(d\tilde{\mathcal{T}}\right) p\left(\tilde{\mathcal{T}},\ell\right)=0,\]
for any $\ell\geq 0$, and substituting this into our earlier bound gives that
${\mathcal{N}(\varepsilon,\delta)}=0$, as desired.
\end{proof}

\begin{lem}\label{rootres} For $\mathbf{P}$-a.e.\ realisation of $\mathcal{T}$, the equality at \eqref{twodefs} holds $P^\mathcal{T}_\rho$-a.s.
\end{lem}
\begin{proof}
By Lemmas \ref{a1check} and \ref{a2check}, $\mathbf{N}$-a.e.\ realisation of $\mathcal{T}$ satisfies Assumptions \ref{assu:1} and \ref{assu:2}. From the scaling property of $\mathbf{N}$, as set out at \eqref{scaling}, the same is true under $\mathbf{P}$. Hence we obtain the result from Proposition \ref{propmain}.
\end{proof}

\subsection{Extending to arbitrary starting points}\label{53}

The proof of Theorem \ref{thm:mainres} is almost complete. It remains to explain how to generalise to arbitrary starting points. For this, the following root invariance property is useful.

\begin{lem}[{\cite[{(20)}]{Aldous2}, \cite[Proposition 4.8]{DLeG}}]\label{rootinv} The law of the rooted compact real tree $(\mathcal{T},\rho)$ under $\mathbf{N}$ is the same as that of $(\mathcal{T},x)$ under $\mathbf{N}(d\mathcal{T})\mu_\mathcal{T}(dx)/\mu_\mathcal{T}(\mathcal{T})$.
\end{lem}

From the scaling property of $\mathbf{N}$ (see \eqref{scaling}), it readily follows that the conclusion of Lemma \ref{rootinv} holds when $\mathbf{N}$ is replaced by $\mathbf{P}$. Combining this observation with Lemma \ref{rootres} yields the following.

\begin{lem}\label{asres} For $\mathbf{P}$-a.e.\ realisation of $\mathcal{T}$, for $\mu_\mathcal{T}$-almost-every $x$, the equality at \eqref{twodefs} holds $P^\mathcal{T}_x$-a.s.
\end{lem}

We can now conclude our proof.

\begin{proof}[Proof of Theorem \ref{thm:mainres} (final step)]
Let us consider the general situation for a compact real tree $(\mathcal{T},d_\mathcal{T})$ equipped with a finite Borel measure of full support $\mu_\mathcal{T}$ such that Assumption \ref{assu:1} holds (so that we have a continuous stochastic process $X^\mathcal{T}$ with jointly continuous local times $L^\mathcal{T}$). Moreover, suppose that the conclusion of Lemma \ref{asres} holds. In particular, there exists a dense subset $D\subseteq\mathcal{T}$ for which \eqref{twodefs} holds, $P^\mathcal{T}_x$-a.s.,\ for all $x\in D$.

Now, let $x\in \mathcal{T}$ and $(x_n)_{n\geq 0}$ be a sequence in $D$ such that $d_\mathcal{T}(x_n,x)\rightarrow 0$. Clearly, from the commute time identity \eqref{commute}, we have for any $\varepsilon>0$ that
\[P^\mathcal{T}_x\left(\tau_{x_n}>\varepsilon\right)\leq 2\varepsilon^{-1}d_\mathcal{T}(x_n,x)\rightarrow 0.\]
Hence, by the continuity of $X^\mathcal{T}$, there exists a deterministic sequence $(\varepsilon_n)_{n\geq 0}$ with $\varepsilon_n\downarrow 0$ such that
\begin{equation}\label{conc1}
P^\mathcal{T}_x\left(X^\mathcal{T}_{[0,\tau_{x_n}]}\subseteq B_\mathcal{T}(x,\varepsilon_n)\right)\rightarrow 1.
\end{equation}

For the next part of the proof, we define $\theta_t$ to be the temporal shift (by time $t$), so that $\tau_{x}\circ\theta_{\tau_{x_n}}$ is the time taken to return to $x$ after hitting $x_n$. Then, for $\varepsilon, \eta>0$, we have that
\begin{eqnarray*}
\lefteqn{P^\mathcal{T}_x\left(B_\mathcal{T}(x,\varepsilon)\not\subseteq X^\mathcal{T}_{[\tau_{x_n},\tau_{\mathrm{cov}}]}\right)}\\
&\leq &
P^\mathcal{T}_x\left(B_\mathcal{T}(x,\varepsilon)\not\subseteq X^\mathcal{T}_{[\tau_{x_n}+\tau_x\circ\theta_{\tau_{x_n}},\eta+\tau_{x_n}+\tau_x\circ\theta_{\tau_{x_n}}]}\right)
+P^\mathcal{T}_x\left(\eta+\tau_{x_n}+\tau_x\circ\theta_{\tau_{x_n}}>\tau_{\mathrm{cov}}\right)\\
&\leq &P^\mathcal{T}_x\left(B_\mathcal{T}(x,\varepsilon)\not\subseteq X^\mathcal{T}_{[0,\eta]}\right)
+P^\mathcal{T}_x\left(2\eta>\tau_{\mathrm{cov}}\right)+2\eta^{-1}d_\mathcal{T}(x_n,x),
\end{eqnarray*}
where the first term has been simplified using the strong Markov property, and we have used the commute time identity again to control the probability of the event $\tau_{x_n}+\tau_x\circ\theta_{\tau_{x_n}}\geq \eta$. Next, since local times increase at the starting point immediately (i.e.,\ \eqref{ltinc}) holds, it is readily deduced from the continuity of the local times that, $P_x^\mathcal{T}$-a.s.,\ $\forall \eta>0$, $X^\mathcal{T}_{[0,\eta]}\supseteq B_\mathcal{T}(0,\varepsilon)$ for some (random) $\varepsilon>0$ (cf.\ the argument of Lemma \ref{midlem}). Thus, for any $\delta,\eta>0$, by choosing (deterministic) $\varepsilon_\eta$ suitably small, it is possible to suppose that
\begin{equation}\label{rrtt}
P_x^\mathcal{T}\left(X^\mathcal{T}_{[0,\eta]}\not \supseteq B_\mathcal{T}(0,\varepsilon_\eta)\right)\leq \delta.
\end{equation}
Moreover, since $\tau_{\mathrm{cov}}>0$, $P_x^\mathcal{T}$-a.s.\ (for which we assume that $\mathcal{T}$ has at least two points to avoid the trivial case that arises otherwise), we have that $P^\mathcal{T}_x\left(2\eta>\tau_{\mathrm{cov}}\right)<\delta$ for suitably small $\eta>0$. With this choice of $\eta$ and $\varepsilon_\eta$ satisfying \eqref{rrtt}, and taking $n$ suitably large, it follows that
\begin{equation}\label{conc2}
P^\mathcal{T}_x\left(B_\mathcal{T}(x,\varepsilon_\eta)\subseteq X^\mathcal{T}_{[\tau_{x_n},\tau_{\mathrm{cov}}]}\right)\geq 1- 3 \delta.
\end{equation}

On the intersection of the events in the probabilities shown at \eqref{conc1} and \eqref{conc2}, and with $n$ suitably large, we have that
\[X^\mathcal{T}_{[0,\tau_{x_n}]}\subseteq B_\mathcal{T}(x,\varepsilon_n)\subseteq B_\mathcal{T}(x,\varepsilon_\eta)\subseteq X^\mathcal{T}_{[\tau_{x_n},\tau_{\mathrm{cov}}]},\]
which implies that $\tau_{\mathrm{cov}}=\tau_{x_n}+\tau_{\mathrm{cov}}\circ\theta_{\tau_{x_n}}$. Consequently, applying the conclusion of Lemma \ref{asres}, we have that $P_x^\mathcal{T}$-a.s.\ on the same event,
\begin{eqnarray*}
\inf\left\{t:\:L^\mathcal{T}_t(y)>0,\:\forall y\in\mathcal{T}\right\}
&\leq &\inf\left\{t\geq \tau_{x_n}:\:L^\mathcal{T}_t(y)-L^\mathcal{T}_{\tau_{x_n}}(y)>0,\:\forall y\in\mathcal{T}\right\}\\
&=&\tau_{x_n}+\tau_{\mathrm{cov}}\circ\theta_{\tau_{x_n}}\\
&=&\tau_{\mathrm{cov}}.
\end{eqnarray*}
Since this inequality holds with probability greater than $1-4\delta$, and $\delta$ was arbitrary, it therefore holds $P_x^\mathcal{T}$-a.s. Since the reverse inequality (i.e.\ \eqref{bb2}) is also true (as can be shown by the same simple continuity argument), we obtain that in fact \eqref{twodefs} holds  $P_x^\mathcal{T}$-a.s.

Finally, by Lemma \ref{a1check} and Lemma \ref{asres} (and the scaling property of \eqref{scaling}), we know that the assumptions on $\mathcal{T}$ hold $\mathbf{P}$-a.s.\ for the Brownian CRT, and so we are done.
\end{proof}

\section{Convergence to the cover time of the Brownian CRT}\label{sec:convsec}

The aim of this section is to prove Theorem \ref{cor:convres} and Corollary \ref{anncor}. The basic strategy is as set out for the Sierpi\'nski gasket in \cite{croydonmoduli}, see \cite[Corollary 7.3 and Remark 7.4]{croydonmoduli} in particular, though in that paper the equivalent to the conclusion of Theorem \ref{thm:mainres} was not known.

We start by presenting a convenient embedding result, which characterises convergence in the space $\mathbb{M}_{L,c}$. The result is essentially a simplification of \cite[Theorem 2.28]{Noda}, which also covered the non-compact case. The notation $d_M^H$, $d_M^P$ and $d_M^L$ was introduced in Section \ref{51}.

\begin{lem}[{\cite{Noda}}]\label{embedding} Let $(K_n,d_{K_n},\mu_{K_n},\rho_{K_n},P^{K_n})$, $n\geq 1$, and $(K,d_{K},\mu_{K},\rho_{K},P^K)$ be elements of $\mathbb{M}_{L,c}$ such that
\begin{equation}\label{k0}
\left({K_n},d_{K_n},\mu_{K_n},\rho_{K_n},P^{K_n}\right)\rightarrow \left({K},d_{K},\mu_{K},\rho_{K},P^{K}\right)
\end{equation}
in $\mathbb{M}_{L,c}$. It is then the case that there exists a rooted compact metric space $(M,d_M,\rho_M)$ and isometric embeddings $\phi_n:K_n\rightarrow M$, $n\geq 1$, and $\phi:K\rightarrow M$ such that
\begin{equation}
\phi_n\left(\rho_{K_n}\right)=\rho_M=\phi\left(\rho_{K}\right),\label{k1}
\end{equation}
\begin{equation}
d_M^H\left(\phi_n\left({K_n}\right),\phi\left({K}\right)\right)\rightarrow 0,\label{k2}
\end{equation}
\begin{equation}
d_M^P\left(\mu_{K_n}\circ\phi_n^{-1},\mu_{K}\circ\phi^{-1}\right)\rightarrow 0,\label{k3}
\end{equation}
\begin{equation}
d_M^L\left(P^{K_n}\circ\phi_n^{-1},P^{K}\circ\phi^{-1}\right)\rightarrow 0.\label{k4}
\end{equation}
Moreover, given \eqref{k2}, the final condition is equivalent to the existence of a coupling of $P^{K_n}\circ\phi_n^{-1}$, $n\geq 1$, and $P^{K}\circ\phi^{-1}$ under which, almost-surely,
\begin{equation}
\left(\phi_n\left(X^{K_n}_t\right)\right)_{t\geq 0}\rightarrow \left(\phi\left(X^{K}_t\right)\right)_{t\geq 0}\label{k5}
\end{equation}
in $D(\mathbb{R}_+,M)$ and also, for any $T>0$,
\begin{equation}
\lim_{\delta\rightarrow0}\limsup_{n\rightarrow\infty}\sup_{\substack{x_n\in K_n,\:x\in K:\\d_M(\phi_n(x_n),\phi(x))\leq \delta}}\sup_{t\in [0,T]}\left|L^{K_n}_t\left(x_n\right)-L^{K}_t\left(x\right)\right|=0.\label{k6}
\end{equation}
\end{lem}
\begin{proof}
See \cite[Theorem 2.28]{Noda} for an embedding result that covers \eqref{k1}, \eqref{k2}, \eqref{k3} and \eqref{k4}. The coupling at \eqref{k5} and \eqref{k6} follows from \cite[Theorem 2.19]{Noda} and Skorokhod's representation theorem.
\end{proof}

\begin{rem}\label{nodarem}
The main result of \cite{Noda} provides conditions under which \eqref{k0} holds. In addition to the convergence of spaces, this involves the metric entropy, as introduced in Section \ref{51}. In \cite{Noda}, it is shown that if $(K_n,d_{K_n},\mu_{K_n},\rho_{K_n})$, $n\geq 1$, are elements of $\check{\mathbb{F}}_c$ satisfying
\[\left({K_n},d_{K_n},\mu_{K_n},\rho_{K_n}\right)\rightarrow \left({K},d_{K},\mu_{K},\rho_{K}\right)\]
for some element of $\mathbb{F}_c$ and there exists an $\alpha\in (0, 1/2)$ such that
\begin{equation}\label{nodacond1}
\lim_{m\to \infty} \limsup_{n\to \infty}\sum_{k=m}^\infty N\left(K_n,d_n, 2^{-k}\right)^2 \exp{\left(-2^{\alpha k}\right)}=0,
\end{equation}
then $(K,d_{K},\mu_{K},\rho_{K})\in \check{\mathbb{F}}_c$ and \eqref{k0} holds (with the laws $P^{K_n}$, $n\geq 1$, and $P^K$ being those naturally associated with the spaces in question). This is a (simplification to compact spaces of) \cite[Theorem 1.7]{Noda}. (See also \cite[Theorem 1.9]{Noda} for a version of this result concerning random spaces.) As discussed in \cite{Noda}, the metric entropy condition at \eqref{nodacond1} is somewhat natural, and can be checked using lower volume estimates (see \cite[Section 7]{Noda}). It is similar to conditions used to deduce continuity of Gaussian processes, and ensures the tightness of the laws of local times in the appropriate space. We note that, in \cite{Noda}, for non-compact spaces, this metric entropy condition is imposed on balls of finite radius, and an additional non-explosion condition is assumed.
\end{rem}

Applying the embedding result of Lemma \ref{embedding}, we prove the following adaptation of Theorem \ref{cor:convres}.

\begin{theorem} \label{covertimeconv}
Suppose $(K_n,d_{K_n},\mu_{K_n},\rho_{K_n})$, $n\geq 1$, and $(K,d_{K},\mu_{K},\rho_{K})$ are elements of $\check{\mathbb{F}}_c$ such that
\[\left({K_n},d_{K_n},\mu_{K_n},\rho_{K_n},P^{K_n}\right)\rightarrow \left({K},d_{K},\mu_{K},\rho_{K},P^{K}\right)\]
in $\mathbb{M}_{L,c}$. Moreover, suppose that, for each $n$, either $K_n$ is at most a countable set or $X^{K_n}$ is $P^{K_n}$-a.s.\ continuous, and that $P^K$-a.s.,\ $X^K$ is continuous and
\begin{equation}\label{taucovK}
\tau_{\mathrm{cov}}(K)=\inf\left\{t\geq 0:\:L^K_t(x)>0,\:\forall x\in K\right\}.
\end{equation}
It is then the case that the law of $\tau_{\mathrm{cov}}(K_n)$ under $P^{K_n}$ converges to that of $\tau_{\mathrm{cov}}(K)$ under $P^K$.
\end{theorem}
\begin{proof} To prove the result, we will, as noted at the start of the section, generalise the argument of \cite[Corollary 7.3]{croydonmoduli}. By Lemma \ref{embedding}, we may assume that the spaces $K_n$ and $K$ have been isometrically embedded into $(M,d_M)$ so that \eqref{k1}-\eqref{k4} hold, and that the measures $P^{K_n}\circ\phi_n^{-1}$, $n\geq 1$, and $P^K\circ\phi^{-1}$ have been coupled so that \eqref{k5}-\eqref{k6} hold. From now on in the proof, we will omit the mappings $\phi_n$ and $\phi$ from the notation to simplify the presentation.

Suppose $t<\tau_{\text{cov}}(K)$, which means that there exists an $x\in K$ such that $x$ is not contained in $X^K_{[0,t]}$. From the compactness of $[0,t]$ and the continuity of $X^K$, there exists a $\varepsilon>0$ such that
\[X^K_{[0,t]}\cap B_K(x,\varepsilon)=\emptyset.\]
Since we have the almost-sure convergence of processes as at \eqref{k5}, it follows that, for large $n$,
\[X^{K_n}_{[0,t]}\cap B_K(x,\varepsilon/2)=\emptyset.\]
However, we also obtain from \eqref{k2} that, for large $n$,
\[K_n\cap B_K(x,\varepsilon/2)\neq\emptyset.\]
Hence it must be the case that, for large $n$, $t\leq \tau_{\text{cov}}(K_n)$. In particular, on the coupled probability space, we have shown that $\liminf_{n\to \infty} \tau_{\text{cov}}(K_n)\ge \tau_{\text{cov}}(K)$ almost-surely.

Next we prove the opposite bound. Write $\tilde{\tau}_{\mathrm{cov}}(K)$ for the right-hand side of \eqref{taucovK}, i.e.\ the first time at which the local times $L^K$ are strictly positive at each point of $K$, and suppose that $t>\tilde{\tau}_{\text{cov}}(K)$. Since local times are increasing in $t$, it must be the case that $L_t^K(y) >0$, for every $y\in K$. Together with the joint continuity of local times, this implies that for some $\varepsilon>0$, $L_t^K(y)>\varepsilon$ for every $y\in K$. Consequently, applying \eqref{k2} and \eqref{k6}, we deduce that, for all large $n$, it is also the case that $L_t^{K_n}(y)>\varepsilon/2$ for every $y\in K_n$. It therefore holds that $t\geq \tilde{\tau}_{\text{cov}}(K_n)$, where the latter quantity is defined analogously to $\tilde{\tau}_{\mathrm{cov}}(K)$. Hence this argument yields that, on the coupled probability space, $\limsup_{n\to \infty} \tilde{\tau}_{\text{cov}}(K_n)\le \tilde{\tau}_{\text{cov}}(K)$. Finally, if $X^{K_n}$ is continuous, we can check that $\tilde{\tau}_{\text{cov}}(K_n)\geq{\tau}_{\text{cov}}(K_n)$ as we did for trees at \eqref{bb2}, and if $K_n$ is at most a countable set, we have from the property of local times at \eqref{ltinc2} that $\tilde{\tau}_{\text{cov}}(K_n)={\tau}_{\text{cov}}(K_n)$. (Neither of the arguments used in the previous sentence are specific to trees.) Since we also have \eqref{taucovK}, it follows that $\limsup_{n\to \infty} {\tau}_{\text{cov}}(K_n)\le {\tau}_{\text{cov}}(K)$, and consequently the proof is complete.
\end{proof}

\begin{proof}[Proof of Theorem \ref{cor:convres}]
It is confirmed in \cite[Corollary 8.1]{Noda} that $(\mathcal{T},2d_\mathcal{T},\mu_\mathcal{T},\rho)\in\check{\mathbb{F}}_c$, $\mathbf{P}$-a.s. Moreover, as discussed in Section \ref{bmsec}, for $\mathbf{P}$-a.e.\ realisation of $\mathcal{T}$, $X^\mathcal{T}$ has continuous sample paths, $\mathbf{P}$-a.s. Hence the result follows on combining Theorems \ref{thm:mainres} and \ref{covertimeconv}.
\end{proof}

\begin{proof}[Proof of Corollary \ref{anncor}]
Since the space $\mathbb{M}_{L,c}$ is a separable metric space (when equipped with the metric described in Section \ref{51}), the result readily follows from Theorem \ref{thm:mainres} on application of Skorokhod's representation theorem.
\end{proof}

\section{Moments of cover times for critical Galton-Watson trees}\label{uisec}

The aim of this section is to prove Corollary \ref{pocor}. Unless otherwise stated, we will suppose that $(T_n)_{n\geq1}$ is a sequence of critical Galton-Watson trees, as described in the statement of that result. Given Corollary \ref{anncor} and the discussion of Example 1, to obtain the desired conclusion, it will suffice to check the uniform boundedness of the $L^p$ moments of the cover times in question. In particular, the main result of this section is as follows.

\begin{propn}\label{momentbound}
In the setting of Corollary \ref{pocor}, it holds that, for $p\geq 1$,
\[\sup_{n\geq 1} n^{-3p/2}\mathbb{E}^{T_n}\left(\tau_{\rm cov}({T}_n)^p\right)<\infty.\]
\end{propn}

\begin{rem}\label{semom}
In fact, our proof of the above result gives that
\[\sup_{n\geq 1} \mathbb{E}^{T_n}\left(e^{c\sqrt{n^{-3/2}\tau_{\rm cov}({T}_n)}}\right)\]
is finite for suitably small $c>0$. However, we are not able to extend this to full exponential moments (i.e.\ we are not able to remove the square root).
\end{rem}

In Proposition \ref{momentbound}, $\tau_{\rm cov}({T}_n)$ could be taken to be the cover time of the discrete-time or constant-speed random walk. To avoid constantly switching between these, we make the following observation, which will allow us to concentrate solely on the discrete-time version of the process for the remainder of the section. In the statement, we write $\|\cdot\|^{T_n}_p$ for the $L^p$-norm under the measure $\mathbb{P}^{T_n}$. In fact the claim is true for random walks on graphs more generally, but we restrict to the specific case to avoid introducing new notation.

\begin{lem}
In the setting of Corollary \ref{pocor}, if $\tau_{\rm cov}^D({T}_n)$ is the cover time of the discrete-time random walk and $\tau_{\rm cov}^C({T}_n)$ is the cover time of the constant-speed random walk (each started from the root), then, for $p=q=1$ or for $1<p<q$, it holds that
\[\left\Vert\tau_{\rm cov}^C({T}_n)\right\Vert^{T_n}_p\leq C_{p,q}\left\Vert\tau_{\rm cov}^D({T}_n)\right\Vert^{T_n}_q,\]
where $C_{p,q}$ is a constant depending only upon $p$ and $q$.
\end{lem}
\begin{proof} Since the CSRW has unit mean exponential holding times, we clearly have that, in distribution,
\begin{equation}\label{couptctd}
\tau_{\rm cov}^C({T}_n)=\sum_{i=1}^{\tau_{\rm cov}^D({T}_n)}\xi_i,
\end{equation}
where $(\xi_i)_{i\geq 1}$ is a sequence of independent and identically distributed unit mean exponential random variables, independent of $\tau_{\rm cov}^D({T}_n)$. Thus the result for $p=q=1$ is immediate (and indeed an equality with $C_{1,1}=1$). For the case $1<p<q$, we first note that, for any $\lambda>0$ and $N\geq 1$,
\[\mathbf{P}\left(\sum_{i=1}^N\xi_i\geq (1+\lambda)N\right)\leq \mathbf{E}\left(e^{\xi_1/2}\right)^Ne^{-N(\lambda+1)/2}=2^Ne^{-N(\lambda+1)/2}=e^{-N(\lambda+1-2\log 2)/2},\]
where we have applied the fact that $\mathbf{E}(e^{\theta\xi_1})=(1-\theta)^{-1}$ for $\theta<1$. From this, supposing that $\tau_{\rm cov}^C({T}_n)$ and $\tau_{\rm cov}^D({T}_n)$ are coupled so that \eqref{couptctd} holds almost-surely, it is straightforward to deduce that
\begin{eqnarray*}
\mathbb{P}^{T_n}\left(\frac{\tau_{\rm cov}^C({T}_n)-\tau_{\rm cov}^D({T}_n)}{\tau_{\rm cov}^D({T}_n)}\geq \lambda\right) &=& \mathbb{P}^{T_n}\left(\tau_{\rm cov}^C({T}_n)\geq (1+\lambda)\tau_{\rm cov}^D({T}_n)\right) \\
&\leq & \mathbb{E}^{T_n}\left(\min\left\{1,e^{-\tau_{\rm cov}^D({T}_n)(\lambda+1-2\log 2)/2}\right\}\right)\\
&\leq & \min\left\{1,e^{-(\lambda+1-2\log 2)/2}\right\},
\end{eqnarray*}
where for the final inequality, we use the obvious bound $\tau_{\rm cov}^D({T}_n)\geq 1$. Since the right-hand side is independent of $n$ and decays exponentially in $\lambda$, it follows that, for $r\geq1$,
\[\sup_{n\geq 1}\left\Vert\frac{\tau_{\rm cov}^C({T}_n)-\tau_{\rm cov}^D({T}_n)}{\tau_{\rm cov}^D({T}_n)}\right\Vert_r^{T_n}\leq C_r,\]
where $C_r$ is a constant only depending on $r$. Hence, applying the triangle inequality and H\"{o}lder's inequality, we find that
\begin{eqnarray*}
\left\Vert\tau_{\rm cov}^C({T}_n)\right\Vert^{T_n}_p&\leq&\left\Vert\tau_{\rm cov}^C({T}_n)-\tau_{\rm cov}^D({T}_n)\right\Vert^{T_n}_p+\left\Vert\tau_{\rm cov}^D({T}_n)\right\Vert^{T_n}_p\\
&\leq & \left\Vert\frac{\tau_{\rm cov}^C({T}_n)-\tau_{\rm cov}^D({T}_n)}{\tau_{\rm cov}^D({T}_n)}\right\Vert^{T_n}_{\frac{pq}{q-p}}\left\Vert\tau_{\rm cov}^D({T}_n)\right\Vert^{T_n}_q+\left\Vert\tau_{\rm cov}^D({T}_n)\right\Vert^{T_n}_p\\
&\leq & C_{p,q}\left\Vert\tau_{\rm cov}^D({T}_n)\right\Vert^{T_n}_q,
\end{eqnarray*}
as desired.
\end{proof}

To prove Proposition \ref{momentbound} for the discrete-time random walks, we will apply a general estimate for the expected cover time of a random walk on a finite graph from \cite{BDNP}. To state this, we introduce the notation
\[t_{\mathrm{cov}}(T_n):=\sup_{x\in T_n}E_x^{T_n}\left(\tau_{\mathrm{cov}}(T_n)\right)\]
for the quenched expectation of $\tau_{\mathrm{cov}}(T_n)$, starting from the worst possible vertex. Similarly to the comment preceding the previous lemma, the result of \cite{BDNP} is for general graphs, but we summarise it for the graphs $(T_n)_{n\geq 1}$ for simplicity of presentation. We recall the notation $N(K,R_K,\varepsilon)$ for the minimal size of an $\varepsilon$-cover of a metric space $(K,R_K)$ from \eqref{coveringdef}.

\begin{lemma}[{\cite[Theorem 1.1]{BDNP}}]\label{cc1}
For $i,n\geq 1$, set
\[A_i^n:=N\left(T_n,d_n,2^{-i}D_n\right),\]
where $d_n$ is the shortest path graph metric on $T_n$ (and so also the resistance metric associated with unit conductances) and $D_n$ is the diameter of $(T_n,d_n)$. Then, there exists a universal constant such that
\[t_{\mathrm{cov}}(T_n)\leq C\left(\sum_{i=1}^{\log_2\log n}\sqrt{2^{-i}\log A_i^n}\right)nD_n.\]
\end{lemma}

Naturally, to appeal to the above result, we will need to provide bounds on the diameter and cover sizes of $T_n$, $n\geq 1$. These are provided by the next two lemmas. Note that we only require a finite variance assumption for the first of these, which was essentially proved in \cite{ABDJ}.

\begin{lemma}[{\cite{ABDJ}}]\label{cc2}
Let $T_n$ be the tree generated by a Galton-Watson process whose offspring distribution is non-trivial, critical (mean one) and has finite variance, conditioned to have size $n$. Writing $D_n$ for the diameter of $T_n$ with respect to its shortest path metric, it then holds that, for all $\lambda \geq 1$,
\[\sup_{n\geq 2}\mathbf{P}\left(D_n\not\in \left[\lambda^{-1} n^{1/2},\lambda n^{1/2}\right]\right)\leq Ce^{-c\lambda^2},\]
where $C$ and $c$ are constants.
\end{lemma}

\begin{proof}
A bound of the form $Ce^{-c\lambda^2}$ on the probability that the height of $T_n$ is greater than $\lambda n^{1/2}$ is provided by \cite[Theorem 1.2]{ABDJ}. Since the diameter is less than twice the height, this gives one half of the result. Moreover, as is explained at the end of \cite[Section 1]{ABDJ}, it follows from \cite[Theorem 1.1]{ABDJ} that one may also bound by $Ce^{-c\lambda^2}$ the probability that the height of $T_n$ is less than $\lambda^{-1} n^{1/2}$. Since the diameter is greater than the height, this completes the proof.
\end{proof}

We now give our estimate on the distribution of $A_i^n$. This is the only part of the proof where exponential moments of the offspring distribution are required; see Remark \ref{gittrem} below for further comment on this point. If one could establish the bound that appears at \eqref{impest4} below under weaker conditions, then the extension of Proposition \ref{momentbound} would also follow.

\begin{lemma}\label{cc3}
In the setting of Corollary \ref{pocor}, it holds that, for every $\alpha<1/2$,
\[\mathbf{P}\left(A_i^n\geq \lambda\right)\leq C 2^i \lambda^{-\alpha},\]
for all $i,n,\lambda\geq 1$, where $C$ is a constant independent of the latter three variables.
\end{lemma}

\begin{proof}
We start by considering the contour function $C_n=(C_n(t))_{t\in [0,2(n-1)]}$ of $T_n$. In particular, $C_n$ records the $d_n$-distance from the root of a depth-first walker as it traverses the outline of $T_n$ from left to right at unit speed; notice that the walker visits every vertex apart from the root a number of times given by its degree. (We extend $C_n$ to be a continuous time process by linear interpolation.) See \cite{rrt} for background, including a detailed definition of this process. Let $\tilde{C}_n=(\tilde{C}_n(t))_{t\in [0,1]}$ be the normalised version of $C_n$ obtained by setting
\[\tilde{C}_n(t)=n^{-1/2}C_n(2(n-1)t).\]
Now, it is established in \cite{Gittenberger} (see equation (1) of that article in particular) that, under the assumptions of the lemma,
\[\mathbf{P}\left(\left|\tilde{C}_n(s)-\tilde{C}_n(t)\right|\ge \varepsilon\right)\le \frac{c_1}{|s-t|} \exp\left(-c_2 \frac{\varepsilon}{\sqrt{|s-t|}}\right),\]
for all $s,t\in [0,1]$ and $\varepsilon>0$, where $c_1$ and $c_2$ are constants that do not depend on $s,t, \varepsilon$ or $n$. (In \cite{Gittenberger}, the bound is written for $s,t\in [0,1]\cap(2(n-1))^{-1}\mathbb{Z}$, but there is no problem in extending to the full range.) Integrating this gives that, for any $p\geq 1$,
\[\mathbf{E}\left(\left|\tilde{C}_n(s)-\tilde{C}_n(t)\right|^p\right)\le c_p |s-t|^{p/2-1},\]
where $c_p$ depends on $p$, but not $s,t$ or $n$. This is the generalisation of Kolmogorov's criterion, as considered in \cite[Theorem I.2.1]{RevuzYor}, and therein taking $\gamma=2p$, $d=1$, $\varepsilon=p-2$ with $p$ large, we deduce that
\begin{align} \label{impest4}
\mathbf{E}\left(\sup_{s, t: \:|s-t|\le \lambda^{-1}} \left|\tilde{C}_n(s)-\tilde{C}_n(t)\right|\right)\le c\lambda^{-\alpha},
\end{align}
for every $\alpha<1/2$, where the constant $c$ does not depend on $n$ or $\lambda>0$.

We next explain how to transfer the bound at \eqref{impest4} to an estimate on the covering sizes for $T_n$. For $1\leq k\leq 2(n-1)$, let $\chi_n^k$ be the set
\[\left\{i\left\lfloor\frac{2(n-1)}{k}\right\rfloor:\:i\in \mathbb{Z}\right\}\cap[0,2(n-1)].\]
Note that $|\chi_n^k|< 3k$ for $n\geq k+1$. Moreover, let $T_n^k$ be those vertices in $T_n$ that the depth-first walker visits at a time in $\chi_n^k$. From the definition of the contour process, it is easy to see that every vertex of $T_n$ is within a distance
\[2n^{1/2}\sup_{s, t: \:|s-t|\le \frac{1}{2(n-1)}\left\lfloor\frac{2(n-1)}{k}\right\rfloor} \left|\tilde{C}_n(s)-\tilde{C}_n(t)\right|\]
of a vertex in $T_n^k$. In particular, if the above quantity is smaller than $\varepsilon n^{1/2}$ (and $n\geq k+1$), then
\[N\left(T_n,d_n,\varepsilon n^{1/2}\right)< 3k.\]

Finally, if $D_n\geq \gamma^{-1}n^{1/2}$, then $A_i^n\leq N(T_n,d_n,2^{-i}\gamma^{-1}n^{1/2})$, and so from the observation of the previous paragraph and Lemma \ref{cc2} we deduce that, for $n\geq \lambda$,
\begin{eqnarray}
\mathbf{P}\left(A_i^n\geq \lambda\right)&\leq& \mathbf{P}\left(D_n<\gamma^{-1}n^{1/2}\right)+\mathbf{P}\left(N(T_n,d_n,2^{-i}\gamma^{-1}n^{1/2}) \geq \lambda\right)\nonumber\\
&\leq & Ce^{-c\gamma^2} + \mathbf{P}\left(2\sup_{s, t: \:|s-t|\le \frac{1}{2(n-1)}\left\lfloor\frac{6(n-1)}{\lambda}\right\rfloor} \left|\tilde{C}_n(s)-\tilde{C}_n(t)\right|>2^{-i}\gamma^{-1} \right)\nonumber\\
&\leq & Ce^{-c\gamma^2} +C2^i\gamma \lambda^{-\alpha},\label{preceding}
\end{eqnarray}
where we have applied \eqref{impest4} to deduce the final inequality. Obviously the second probability on the right-hand side is zero for $6n<\lambda$. Moreover, for $\lambda\in [n,6n]$, we have that, deterministically,
\[\sup_{s, t: \:|s-t|\le \frac{1}{2(n-1)}\left\lfloor\frac{6(n-1)}{\lambda}\right\rfloor} \left|\tilde{C}_n(s)-\tilde{C}_n(t)\right|\leq 6n^{-1/2}\leq C\lambda^{-\alpha},\]
and so we can obtain a bound of the form $C2^i\gamma \lambda^{-\alpha}$ for the probability in question by applying Markov's inequality. In particular, the bound at \eqref{preceding} holds for all $i,\lambda,\gamma\geq 1$ and $n\geq 2$. (The bound on the term involving the diameter is also satisfied for such parameters.) To complete the proof, we choose $\gamma=\gamma(\lambda)\sim\log \lambda$, which (after adjusting $\alpha$) yields the desired conclusion.
\end{proof}

\begin{rem}\label{gittrem}
In fact, in the reference we cite in proving \eqref{impest4}, i.e.\ \cite{Gittenberger}, it is claimed that a finite variance condition on the offspring distribution is sufficient for the relevant estimate to hold, but in \cite[footnote (2)]{Marzouk}, it is noted that the proof in \cite{Gittenberger} uses exponential moments. To further elaborate on where the deficiency lies, in \cite{Gittenberger}, on several occasions, results are used from another paper \cite{GittenbergerContour} (\cite[reference 3]{Gittenberger}), which studies the contour function of simply generated random trees. In the latter (especially in Theorem 1.1, as well as in Lemma 3.1), there is an assumption that \cite[equation (3.1)]{Gittenberger} has a minimal positive solution $\tau<R$. For conditioned trees, as in our setting, \cite[equation (3.1)]{GittenbergerContour} designates that the simply generated random tree is associated with a critical Galton-Watson branching process with finite variance. A minimal positive solution $\tau<R$ is equivalent to the moment generating function of the offspring distribution having a radius of convergence equal to $R/\tau>1$, and hence the distribution having exponential moments.

We further note that an alternative route to proving a bound such as \eqref{impest4} is described in the proof of \cite[Lemma 1]{JM}. However, in that paper, as well as the result of \cite{MM} that is cited there, an exponential moment assumption is made, and further work would be required to check whether or not this is necessary for our purposes.
\end{rem}

Putting together the three previous results, we arrive at the following conclusion, which is a refinement of a bound originally proved in \cite[Theorem 3.2]{BDNP}.

\begin{lemma}\label{combin}
In the setting of Corollary \ref{pocor}, it holds that
\[\mathbf{P}\left(n^{-3/2}t_{\rm{cov}}(T_n)\geq \lambda\right)\leq Ce^{-c\lambda},\]
for all $n,\lambda\geq1$, where the constants are independent of the latter two variables.
\end{lemma}
\begin{proof}
The proof is an adaptation of the proof of \cite[Theorem 3.2]{BDNP}. From Lemma \ref{cc1}, we have that, for $\lambda,\gamma\geq 1$,
\begin{eqnarray*}
\mathbf{P}\left(n^{-3/2}t_{\rm{cov}}(T_n)\geq \lambda\right)&\leq&\mathbf{P}\left(C\left(\sum_{i=1}^{\log_2\log n}\sqrt{2^{-i}\log A_i^n}\right)n^{-1/2}D_n\geq  \lambda\right)\\
&\leq &\mathbf{P}\left(C\left(\sum_{i=1}^{\log_2\log n}\sqrt{2^{-i}\log A_i^n}\right) \geq\lambda/\gamma\right)+\mathbf{P}\left(D_n\geq\gamma n^{1/2}\right).
\end{eqnarray*}
By Lemma \ref{cc2}, the second term is bounded above by $Ce^{-c\gamma^2}$. Using a union bound and Lemma \ref{cc3}, the first term can be estimated from above by
\[\sum_{i=1}^\infty\mathbf{P}\left(2^{-i}\log A_i^n \geq\frac{(2^{1/4}-1)\lambda^2}{2^{i/2}\gamma^2} \right)\leq C\sum_{i=1}^\infty
2^i e^{-c\lambda^2 \gamma^{-2}2^{i/2}}.\]
Hence, setting $\gamma=\lambda^{1/2}$, we have shown that
\[\mathbf{P}\left(n^{-3/2}t_{\rm{cov}}(T_n)\geq \lambda\right)\leq Ce^{-c\lambda}+C\sum_{i=1}^\infty
2^i e^{-c\lambda2^{i/2}},\]
from which the result follows
\end{proof}

To replace $t_{\rm{cov}}(T_n)$ by $\tau_{\rm cov}(T_n)$ in the conclusion of the above lemma, we apply the following exponential tail bound for the cover time of discrete-time random walks on general trees. We will write a general finite graph tree as $T$, and suppose the associated random walk is started from a root vertex. The result is an adaptation of the concentration result of \cite[Corollary 1.3]{Ding}.

{\lem \label{treeexp} There exist universal constants $c_1,c_2$ such that, for any finite rooted graph tree $T$,
\[P^{T}\left(\tau_{\rm cov}(T)\geq \lambda t_{\rm cov}(T)\right)\leq c_1e^{-c_2\lambda},\hspace{20pt}\forall\lambda\geq 1.\]}
\begin{proof} It is proved as \cite[Theorem 1.2]{Ding} that there exist universal constants $c_1,c_2$ such that, for any finite rooted graph tree $T$,
\begin{equation}\label{b1}
P^T\left(\left|\tau_{\rm cov}(T)-E(T)\eta_T^2\right|\geq \lambda E(T)\eta\sqrt{{\rm diam}(T)}\right)\leq c_1e^{-c_2\lambda},\qquad\forall\lambda\geq 1,
\end{equation}
where $\eta_T$ is the expected maximum of the Gaussian free field on $T$ that takes the value 0 at the root $\rho$ and has covariance matrix $(\frac{1}{2}(d_T(\rho,x)+d_T(\rho,y)-d_T(x,y)))_{x.y\in T}$, with $d_T$ being the shortest path graph distance on $T$. We have also written $E(T)$ for the number of edges of $T$ and ${\rm diam}(T)$ for the diameter of $(T,d_T)$. (Precisely, \cite[Theorem 1.2]{Ding} is written for the CSRW, but as is explained in \cite[Section 1.2]{Ding}, the result also holds for the discrete-time random walk. Moreover, the main result of \cite{DLP} yields that there exist further universal constants $c_3,c_4$ such that, for any finite rooted graph tree $T$,
\begin{equation}\label{b2}
c_3E(T)\eta_T^2\leq t_{\rm cov}(T)\leq c_4E(T)\eta_T^2.
\end{equation}
(Clearly, there is no distinction between $t_{\rm cov}(T)$ for the CSRW and the discrete-time random walk, and so the result is applicable to either.) Combining (\ref{b1}) and (\ref{b2}), we obtain that there exists a universal constant $c_5$ such that
\[P^T\left(\tau_{\rm cov}(T)\geq
c_5\left(t_{\rm cov}(T)+ \lambda\sqrt{t_{\rm cov}(T)E(T){\rm diam}(T)}\right)
\right)\leq c_1e^{-c_2\lambda},\qquad\forall\lambda\geq 1.\]
By applying that $t_{\rm cov}(T)\geq E(T){\rm diam}(T)$ (which is an elementary consequence of the commute time identity for random walks, see \cite{LPW}, for example), the result follows.
\end{proof}

We are now ready to conclude the proof of the desired result.

\begin{proof}[Proof of Proposition \ref{momentbound}]
We first observe that, for any $\lambda\geq \gamma\geq 1$,
\begin{eqnarray*}
P^{T_n}\left(n^{-3/2}\tau_{\rm{cov}}(T_n)\geq \lambda \right)
&\leq&\mathbf{1}_{\{n^{-3/2}t_{\rm{cov}}(T_n)\geq \lambda /\gamma\}}+P^T\left(\tau_{\rm{cov}}(T_n)\geq \gamma t_{\rm{cov}}(T_n)\right)\\
&\leq & \mathbf{1}_{\{n^{-3/2}t_{\rm{cov}}(T_n)\geq \lambda /\gamma\}}+c_1e^{-c_2\gamma},
\end{eqnarray*}
where $c_1,c_2$ are the universal constants of Lemma \ref{treeexp}. Taking expectations with respect to $\mathbf{P}$, Lemma \ref{combin} allows us to deduce from this that, for any $\varepsilon>0$, there exists a $c_3$ such that
\[\mathbb{P}^{T_n}\left( n^{-3/2}\tau_{\rm{cov}}(T)\geq \lambda \right)
\leq Ce^{-c\lambda/\gamma}+c_1e^{-c_2\gamma}.\]
On setting $\gamma=\lambda^{1/2}$, this readily implies the result.
\end{proof}

\section{Cover-and-return times}\label{sec:cr}

In this final section, we consider cover-and-return times, and in particular prove Corollary \ref{crcor}. We will again work in the general setting of Section \ref{sec:convsec}. Given our other preparations, the following result will readily yield parts (a) and (b) of the result.

\begin{lem}\label{l81}
In the setting of Theorem \ref{covertimeconv}, the law of $\tau_{\mathrm{cov}}^+(K_n)$ under $P^{K_n}$ converges to that of $\tau_{\mathrm{cov}}^+(K)$ under $P^K$.
\end{lem}
\begin{proof}
As in the proof of Theorem \ref{covertimeconv}, we will assume that the spaces $K_n$ and $K$ have been isometrically embedded into $(M,d_M)$ so that \eqref{k1}-\eqref{k4} hold, and that the measures $P^{K_n}\circ\phi_n^{-1}$, $n\geq 1$, and $P^K\circ\phi^{-1}$ have been coupled so that \eqref{k5}-\eqref{k6} hold. Also as there, we omit the mappings $\phi_n$ and $\phi$ from the notation.

Now, we have from the proof of Theorem \ref{covertimeconv} that, on the coupled space, $\tau_{\mathrm{cov}}(K_n)\rightarrow\tau_{\mathrm{cov}}(K)$, almost-surely. Since the processes converge as at \eqref{k5} and the limit process is continuous, it follows that
\begin{equation}\label{xktconv}
X^{K_n}_{\tau_{\mathrm{cov}}(K_n)}\rightarrow X^{K}_{\tau_{\mathrm{cov}}(K)}
\end{equation}
in $(M,d_M)$, almost-surely.

We next claim that if $x_n\in K_n$, $n\geq 1$, and $x\in K$ are such that $d_M(x_n,x)\rightarrow0$, then the law of $\tau_{\rho_{K_n}}^{K_n}$ under $P^{K_n}_{x_n}$ converges to that of $\tau_{\rho_{K}}^{K}$ under $P^{K}_{x}$, where we write $\tau^K_{\rho_K}$ for the hitting time of $\rho_K$ by $X^K$ and $P^K_x$ for the law of $X^K$ started from $x$ (and similarly for the $n$-indexed processes). The proof of this is an adaptation of the argument we applied in the first part of the proof of Lemma \ref{lowerlem}. Firstly, it is an easy extension of the conclusion of \cite{Croydonscalelim} to deduce that, with the spaces embedded as they are, the law of $X^{K_n}$ under $P^{K_n}_{x_n}$ converges to that of $X^K$ under $P^{K}_{x}$. (In \cite{Croydonscalelim}, processes were started from the root, but this is not an essential restriction, and indeed we can just consider the spaces rerooted at $x_n$ and $x$, respectively.) Hence, we can suppose that we have a coupling so that this convergence holds almost-surely. Under such a coupling, we have from \eqref{k5} that if $\tau^K_{\rho_K}<t$, then, for every $\varepsilon>0$,
\begin{equation}\label{jio1}
\limsup_{n\rightarrow \infty}\tau^{K_n}_{\bar{B}_M(\rho_M,\varepsilon)}<t,
\end{equation}
where $\tau^{K_n}_{\bar{B}_M(\rho_M,\varepsilon)}$ is the hitting time of the relevant closed ball in $(M,d_M)$ by $X^{K_n}$. Applying the strong Markov property at time $\tau^{K_n}_{\bar{B}_M(\rho_M,\varepsilon)}$, it follows that, for all $\varepsilon,\delta>0$,
\begin{eqnarray*}
\lefteqn{\limsup_{n\rightarrow \infty}P^{K_n}_{x_n}\left(\tau_{\rho_{K_n}}^{K_n}\geq t\right)}\\
&\leq& \limsup_{n\rightarrow \infty}\left(P^{K_n}_{x_n}\left(\tau^{K_n}_{\bar{B}_M(\rho_M,\varepsilon)}\geq t-\delta\right)+\sup_{y\in \bar{B}_M(\rho_M,\varepsilon)\cap K_n}P^{K_n}_y\left(\tau^{K_n}_{\rho_{K_n}}\geq \delta\right)\right)\\
&\leq & P^{K}_x\left(\tau^{K}_{\rho_K}\geq t-\delta\right)+\varepsilon\delta^{-1}\mu_K(K),
\end{eqnarray*}
where for the final inequality follows from \eqref{k3}, \eqref{jio1} and the commute time identity for stochastic processes associated with resistance forms (see \cite[Lemma 2.4]{Croydonscalelim}, for example). Taking $\varepsilon\rightarrow 0$ and then $\delta\rightarrow 0$, we find that
\[\limsup_{n\rightarrow \infty}P^{K_n}_{x_n}\left(\tau_{\rho_{K_n}}^{K_n}\geq t\right)\leq P^{K}_x\left(\tau^{K}_{\rho_K}\geq t\right).\]
Conversely, if $\tau^K_{\rho_K}>t$, then, by continuity, there exists an $\varepsilon>0$ such that $X^K_{[0,t+\varepsilon]}\cap \bar{B}_M(\rho_M,\varepsilon)=\emptyset$. Hence, from \eqref{k5}, we deduce that $X^{K_n}_{[0,t+\varepsilon]}\cap \bar{B}_M(\rho_M,\varepsilon/2)=\emptyset$ for large $n$, and consequently, $\liminf_{n\rightarrow\infty}\tau_{\rho_{K_n}}^{K_n}> t$. Hence $\liminf_{n\rightarrow\infty}\tau_{\rho_{K_n}}^{K_n}\geq \tau_{\rho_{K}}^{K}$, which implies
\[\liminf_{n\rightarrow \infty}P^{K_n}_{x_n}\left(\tau_{\rho_{K_n}}^{K_n}> t\right)\geq \liminf_{n\rightarrow \infty}P^{K}_{x}\left(\tau_{\rho_{K}}^{K}> t\right).\]
Combining this and the corresponding upper bound confirms the claim.

Finally, together with the convergence at \eqref{xktconv}, the result of the previous paragraph readily implies that, almost-surely on our coupled probability space, we have the following weak convergence statement:
\begin{align*}
\lefteqn{P^{K_n}\left(\tau_{\mathrm{cov}}^+(K_n)-\tau_{\mathrm{cov}}(K_n)\in \cdot\:\vline\:\left(X^{K_n}_t\right)_{t\leq \tau_{\mathrm{cov}}(K_n)}\right)}\\
&\hspace{50pt}\rightarrow P^{K}\left(\tau_{\mathrm{cov}}^+(K)-\tau_{\mathrm{cov}}(K)\in \cdot\:\vline\:\left(X^{K}_t\right)_{t\leq \tau_{\mathrm{cov}}(K)}\right).
\end{align*}
Since we also know the cover times converge almost-surely on the coupled space, this is enough to complete the proof.
\end{proof}

To check the remaining part of Corollary \ref{crcor}, we strengthen Proposition \ref{momentbound} as follows. We stress that, as for Proposition \ref{momentbound}, this holds for the discrete-time or constant-speed random walk.

\begin{lem}\label{momentboundplus}
In the setting of Corollary \ref{pocor}, it holds that, for $p\geq 1$,
\[\sup_{n\geq 1} n^{-3p/2}\mathbb{E}^{T_n}\left(\tau_{\rm cov}^+({T}_n)^p\right)<\infty.\]
\end{lem}
\begin{proof}
Since
\[\tau_{\rm cov}^+({T}_n)^p\leq 2^{p-1}\left(\tau_{\rm cov}({T}_n)^p+\left(\tau_{\rm cov}^+({T}_n)-\tau_{\rm cov}({T}_n)\right)^p\right),\]
and we already have moment bounds for $\tau_{\rm cov}({T}_n)$ from Proposition \ref{momentbound}, it will suffice to check corresponding moment bounds for $\tau_{\rm cov}^+({T}_n)-\tau_{\rm cov}({T}_n)$. Specifically, we will show that, for $p\geq 1$,
\begin{equation}\label{hereto}
\sup_{n\geq 1}n^{-3p/2}\mathbf{E}\left(\sup_{x,y\in T_n}{E}^{T_n}_x\left(\left(\tau_y^{T_n}\right)^p\right)\right)<\infty,
\end{equation}
from which the desired result will readily follow. To this end, we simply observe that, by applying the Markov property repeatedly, for $k\in \mathbb{N}$,
\begin{align*}
\sup_{x,y\in T_n}P^{T_n}_{x}\left(\tau^{T_n}_y\geq  4 k n D_n\right)&\leq \sup_{x,y\in T_n}P^{T_n}_{x}\left(\tau^{T_n}_y\geq  4 (k-1) n D_n\right)\sup_{x,y\in T_n}P^{T_n}_{x}\left(\tau^{T_n}_y\geq  4  n D_n\right)\\
&\leq \sup_{x,y\in T_n}P^{T_n}_{x}\left(\tau^{T_n}_y\geq  4 (k-1) n D_n\right)\sup_{x,y\in T_n}\frac{E^{T_n}_{x}\left(\tau^{T_n}_y\right)}{4  n D_n}\\
&\leq \frac{1}{2}\sup_{x,y\in T_n}P^{T_n}_{x}\left(\tau^{T_n}_y\geq  4 (k-1) n D_n\right)\\
&\leq 2^{-k},
\end{align*}
where we again write $D_n$ for the diameter of $(T_n,d_n)$ and have applied the commute time identity for random walks on graph (which applies to both the CSRW and the discrete-time random walk), see \cite[Proposition 10.6]{LPW}, for example, to deduce the third inequality. From this, it is easy to conclude that, for $\lambda\geq 1$,
\begin{equation}\label{expdecc}
\sup_{n\geq 1}\sup_{x,y\in T_n}P^{T_n}_{x}\left(\tau^{T_n}_y\geq  \lambda n D_n\right)\leq Ce^{-c\lambda},
\end{equation}
where $C$ and $c$ are deterministic constants. Since we clearly have that
\begin{align*}
\lefteqn{n^{-3p/2}\mathbf{E}\left(\sup_{x,y\in T_n}{E}^{T_n}_x\left(\left(\tau_y^{T_n}\right)^p\right)\right)}\\
&= \mathbf{E}\left(\sup_{x,y\in T_n}{E}^{T_n}_x\left(\left(\frac{\tau_y^{T_n}}{nD_n}\right)^p\right)\times\left(\frac{D_n}{n^{1/2}}\right)^p\right)\\
&\leq \mathbf{E}\left(\sup_{x,y\in T_n}{E}^{T_n}_x\left(\left(\frac{\tau_y^{T_n}}{nD_n}\right)^{2p}\right)\right)^{1/2}\mathbf{E}\left(\left(\frac{D_n}{n^{1/2}}\right)^{2p}\right)^{1/2},
\end{align*}
the result at \eqref{hereto} follows from \eqref{expdecc} and Lemma \ref{cc2}.
\end{proof}

\begin{proof}[Proof of Corollary \ref{crcor}]
With Lemma \ref{l81} in place of Theorem \ref{covertimeconv}, parts (a) and (b) are checked in exactly the same way as Theorem \ref{cor:convres} and Corollary \ref{anncor} were. As for part (c), since we have Lemma \ref{momentboundplus}, again this can be proved in the same way as the corresponding result for the cover time, i.e.\ Corollary \ref{pocor}.
\end{proof}

\begin{funding}
DC received support from JSPS Grant-in-Aid for Scientific Research (C) 19K03540, 24K06758 and the Research Institute for Mathematical Sciences, an International Joint Usage/Research Center located in Kyoto University. LM was supported by the ANR grant ProGraM (ANR-19-CE40-0025).
\end{funding}


\bibliographystyle{imsart-number} 
\bibliography{coverCRT}

\end{document}